\documentclass[12pt,a4paper]{article}

\usepackage[utf8]{inputenc}
\usepackage[frenchb,english]{babel}
\usepackage[T1]{fontenc}
\usepackage{lmodern}

\usepackage{dsfont,pifont,xfrac}
\usepackage{amsthm}
\usepackage{verbatim}
\usepackage{bbm}

\usepackage[left=.9in, right=.9in,top=1.2in,bottom=1.2in]{geometry}

\usepackage{graphicx}
\usepackage{wrapfig,flafter,color,graphicx}
\usepackage{float}
\usepackage{enumerate,enumitem}
\usepackage[colorlinks=true]{hyperref}
\usepackage{multirow,array}   
\usepackage{tabu}
\usepackage[normalem]{ulem}

\usepackage{mathsymb-linxiao}
\usepackage{thm-linxiao}
\usepackage{twoopt}

\newcommand*{\Map}[1]{\mathfrak{#1}} 

\newcommand*{\map}{\Map{m}}  
\newcommand*{\tmap}{\Map{t}} 
\newcommand*{\emap}{\Map{e}} 
\newcommand*{\umap}{\Map{u}} 
\newcommand*{\bmap}{\Map{b}} 

\newcommand*{\bt}[1][]{(\tmap#1,\sigma#1)}
\newcommand*{\btsq}[1][]{[\tmap#1,\sigma#1]}
\newcommand*{\bts}{\mathcal{BT}}

\newcommand*{\law}{\mathcal{L}}
\renewcommand*{\cv}[2][d\1{loc}]{\xrightarrow[\raisebox{0.1em}{$\scriptstyle #2\to\infty$}]{#1}}

\newcommand*{\faces}{\mathcal{F}}
\newcommand*{\medges}{\mathcal{E}}

\newcommand*{\steps}{\mathcal{S}}
\newcommand*{\step}{\mathtt{s}}
\newcommand*{\Step}{\mathtt{S}}
\newcommand*{\CC}{\mathtt{C}}
\newcommand*{\LL}{\mathtt{L}}		\newcommand*{\RR}{\mathtt{R}}
\newcommand*{\cp}{\CC^\+}			\newcommand*{\cm}{\CC^\<}
\newcommand*{\lp}[1][k]{\LL^\+_{#1}}	\newcommand*{\lm}[1][k]{\LL^\<_{#1}}
\newcommand*{\rp}[1][k]{\RR^\+_{#1}}	\newcommand*{\rn}[1][k]{\RR^\<_{#1}}
\newcommand*{\iroot}{\mathcal I}

\newcommandtwoopt*{\weight}[4][][]{\nu^{\abs{\medges(#3#4)}#1}\, t^{\abs{\faces(#3)}#2}}
\newcommandtwoopt*{\weightc}[4][][]{\nu_c^{\abs{\medges(#3#4)}#1}\, t_c^{\abs{\faces(#3)}#2}}
\newcommand*{\zinv}[1][p,q+1]{\frac1{z_{#1}}}
\newcommand*{\zz}[2][p,q+1]{\,\frac{z_{#2}}{z_{#1}}}
\newcommand*{\zzz}[3][p,q+1]{\,\frac{z_{#2}\,z_{#3}}{z_{#1}}}

\newcommand{\pqq}{{p,(q_1,q_2)}}
\newcommand{\py}[1][p]{_{#1}}
\newcommand{\yy}{_{\infty}}

\newcommand*{\+}{\ensuremath{\text{\rm\texttt{+}}}}
\newcommand*{\<}{\ensuremath{\text{\rm\texttt{-}}}}
\newcommand*{\jj}{{\text{\normalsize\texttt{\textpm}}}}

\newcommand{\emapo}{\emap^\circ}
\newcommand{\frontier}{\partial\emap}

\newcommand*{\algo}{\mathcal{A}}

\newcommand{\nseq}[2][0]{(#2_n)_{n\ge #1}}

\setlist[enumerate,1]{label=(\roman*),topsep=6pt,itemsep=0pt}
\setlist[itemize,1]{topsep=4pt,itemsep=0pt}

\newcommand{\refp}[2]{\hyperref[#2]{\ref*{#2}(#1)}}

\newcommand{\limsupp}{\limsup_{p\to\infty}}

\title{Critical Ising model on random triangulations of the disk: enumeration and local limits}

\author{Linxiao Chen\footnote{University of Helsinki, Department of Mathematics and Statistics, linxiao.chen@helsinki.fi}\ ,\
Joonas Turunen\footnote{University of Helsinki, Department of Mathematics and Statistics, joonas.am.turunen@helsinki.fi}
}
\date{}

\begin{document}
\maketitle

\begin{abstract}
We consider Boltzmann random triangulations coupled to the Ising model on their faces, under Dobrushin boundary conditions and at the critical point of the model. 
The first part of this paper computes explicitly the partition function of this model by solving its Tutte's equation, extending a previous result by Bernardi and Bousquet-M\'elou \cite{BBM11} to the model with Dobrushin boundary conditions. 
We show that the perimeter exponent of the model is $7/3$ in contrast to the exponent $5/2$ for uniform triangulations.
In the second part, we show that the model has a local limit in distribution when the two components of the Dobrushin boundary tend to infinity one after the other. 
The local limit is constructed explicitly using the peeling process along an Ising interface. Moreover, we show that the main interface in the local limit touches the (infinite) boundary almost surely only finitely many times, a behavior opposite to that of the Bernoulli percolation on uniform maps. Some scaling limits closely related to the perimeters of finite clusters are also obtained.
\end{abstract}

\tableofcontents

\section{Introduction}\label{sec:intro}

Recent years have seen an increasing number of works devoted to random planar maps decorated by additional combinatorial structures such as trees, orientations and spin models. We refer to \cite{BM11uncoloured} for a survey from an enumerative combinatorics point of view.
From a probabilistic point of view, one important motivation for studying decorated random maps is to understand models of two-dimensional random geometry that escape from the now well-understood universality class of the Brownian map \cite{LG11,Mie11}. This is in turn motivated by an effort to give a solid mathematical foundation to the physical theory of Liouville quantum gravity by discretization \cite{ADJ97}.

The critical Ising model is one of the simplest combinatorial structures that, when coupled to a random planar map, have a non-trivial impact on the geometry of the latter. The systematic study of the Ising model on random lattices was pioneered by Boulatov and Kazakov back in the eighties \cite{Kaz86,BouKaz87}. Using relations to the two-matrix model, they computed the partition function of the Ising model on random triangulations and quadrangulations in the thermodynamic limit, identifying its phase transitions and computing the associated critical exponents. This approach was later refined and generalized to deal with Ising models on more general maps as well as the Potts model \cite{EynOra05,EynBon99}. A more mathematical derivation of the partition function on the discrete level was later given by Bernardi, Bousquet-M\'elou and Schaeffer in \cite{BMS02,BBM11}. In these works, the partition function is shown to be algebraic and having a rational parametrization. 
Our work complements the ones in \cite{BouKaz87,BBM11} by dealing with Ising-decorated triangulations with a large boundary and a Dobrushin boundary condition. In addition, we exploit these combinatorial results using the so-called \emph{peeling process} to derive some scaling limits of quantities describing the geometry of the Ising-interface, and ultimately construct the local limit of the Ising-decorated random maps themselves.

\begin{figure}
\centering
\includegraphics[scale=1,page=1]{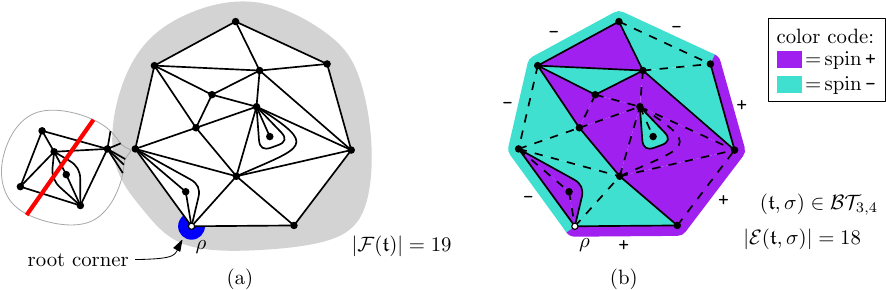}
\caption{(a) A triangulation $\tmap$ of the 7-gon with 19 internal faces. The boundary will no longer be simple if one attaches to $\tmap$ the map inside the bubble to its left.
(b) a bicolored triangulation of the $(3,4)$-gon with 18 monochromatic edges (dashed lines).
}
\label{fig:def-map}
\end{figure}

Let us define our conventions and terminology before stating the main results.

\paragraph{Planar maps.}

We refer to \cite{StFlour14,CurPeccot} for self-contained introductions to random planar maps. Here we consider planar maps in which loops and multiple edges are allowed. A map is \emph{rooted} when it has a distinguished corner. This corner determines a distinguished vertex $\rho$, called the \emph{origin}, and a distinguished face, called the \emph{external face}. The other faces are called \emph{internal faces}. We denote by $\faces(\map)$ the set of internal faces of a map $\map$.
\begin{center} \vspace{-1.ex}
	\emph{In the following, all maps are assumed to be planar and rooted.}
\end{center}   \vspace{-1.ex}
A map is a \emph{triangulation of the $\ell$-gon} ($\ell\ge 1$) if the internal faces all have degree three, and the contour of its external face is a simple closed path (i.e.\ it visits each vertex at most once) of length $\ell$. The number $\ell$ is called the \emph{perimeter} of the triangulation, and an edge (resp.\ vertex) adjacent to the external face is called a \emph{boundary edge} (resp.\ \emph{boundary vertex}). Figure~\refp{a}{fig:def-map} gives an example of a triangulation of the 7-gon. By convention, the edge map --- the map containing only one edge and no internal face --- is a triangulation of the 2-gon.

\paragraph{Bicolored triangulations of the $(p,q)$-gon.}

We consider the Ising model with spins on the internal faces of a triangulation of a polygon. The triangulation together with an Ising spin configuration on it is represented by a pair $\bt$ where $\sigma \in \{\+,\<\}^{\faces(\tmap)}$.
An edge $e$ of $\tmap$ is said to be \emph{monochromatic} if the spins on both sides of $e$ are the same. When $e$ is a boundary edge, this definition requires a boundary condition which specifies a spin outside each boundary edge. By an abuse of notation, we consider the information about the boundary condition to be contained in the coloring $\sigma$, and denote by $\medges\bt$ the set of monochromatic edges in $\bt$.
\begin{center} \vspace{-.3ex}
\pbox{0.8}{\emph{In this work, we concentrate on the \emph{Dobrushin boundary conditions} which assign a sequence of spins of the form $\+\cdots\+\,\<\cdots\<$ to the boundary edges in the counter-clockwise order starting from the origin.}}
\end{center}   \vspace{-.3ex}
Let $p$ and $q$ be respectively the numbers of \+ and of \< in this sequence. Then we call $\bt$ a \emph{bicolored triangulation of the $(p,q)$-gon}. Figure~\refp{b}{fig:def-map} gives an example in the case $p=3$ and $q=4$. We denote by $\bts_{p,q}$ the set of all bicolored triangulations of the $(p,q)$-gon.

We enumerate the elements of $\bts_{p,q}$ by the generating function
\begin{equation*}
z_{p,q}(\nu,t)
\ = \!\!\sum_{\bt \in \bts_{p,q}}\!\!
		\nu^{\abs{\medges\bt}}\, t^{\abs{\faces(\tmap)}}
\end{equation*}
where $\nu>0$ is related to the coupling constant of the Ising model, and $t$ is a parameter that controls the volume of the triangulation. Actually, $\nu$ equals the exponential of two times the inverse temperature. When $q=0$ and $p$ is small, the above generating function has already been computed by Bernardi and Bousquet-M\'elou in \cite{BBM11}. (More precisely, they computed the generating function of a model that is dual to ours. See Section~\ref{sec:z solution} for more details.) A part of their result can be translated in our setting as follows.

\begin{citeproposition}[{\cite[Section 12.2]{BBM11}}]\label{prop:BBM}
For $\nu\ge 1$, the coefficient of $t^n$ in $z_{1,0}$ satisfies
\begin{equation*}
	[t^n] z_{1,0}(\nu,t) \eqv{n} \begin{cases}
		\kappa\ \tau^{-n}\, n^{-5/2}		&\text{if }\nu\ne \nu_c
	\\	\kappa_c\, t_c^{-n}\, n^{-7/3}	&\text{if }\nu  = \nu_c
	\end{cases}
\end{equation*}
where $\nu_c=1+2\sqrt7$ and $t_c = \frac{\sqrt{10}}{8(7+\sqrt7)^{3/2}}
$, and $\kappa, \tau$ are continuous functions of $\nu$ such that 
$\tau(\nu_c)=t_c$. In particular, $z_{1,0}(\nu,\tau(\nu))<\infty$ for all $\nu\ge 1$.
\end{citeproposition}

\noindent
This result suggests that $\nu_c=1+2\sqrt7$ is the unique value of $\nu$ at which the asymptotic behavior of the Ising-decorated random triangulation escapes from the pure gravity universality class (corresponding to $\nu=1$). This is in agreement with the prediction of the celebrated KPZ relation \cite{KPZ88} between the string susceptibility exponent $\gamma$ and the central charge $\mathfrak c$ of the conformal field theory (CFT) on a surface of genus zero. 
According to CFT, the critical Ising model has a central charge $\mathfrak c = \frac12$, whereas pure gravity corresponds to $\mathfrak c=0$. The string susceptibility exponent $\gamma$ is related to the asymptotics of $[t^n] z_{1,0}(\nu,t)$ by 
\begin{equation*}
[t^n] z_{1,0}(\nu,t) \eqv{n} \kappa(\nu)\tau(\nu)^{-n}n^{\gamma(\nu)-2} \,.
\end{equation*}
Thus Proposition~\ref{prop:BBM} gives that $\gamma(\nu_c)=-\frac13$ and $\gamma(\nu)=-\frac12$ for $\nu \ne \nu_c$. This is in agreement with the KPZ prediction of $\gamma = \frac1{12}\left(\mathfrak c-1-\sqrt{ (1-\mathfrak c)(25-\mathfrak c) }\right)$, see \cite[(4.223)]{ADJ97}.

In this work, we will concentrate on the critical value of the parameters, and leave the general case, as well as the phase transitions, to an upcoming work. In all that follows, we fix $(\nu,t)=(\nu_c,t_c)$ and write $z_{p,q}=z_{p,q}(\nu_c,t_c)$.

\begin{theorem}[Asymptotics of $z_{p,q}$]\label{thm:z_p,q}
The generating function $Z(u,v) = \sum_{p,q\ge 0} z_{p,q} u^pv^q$ is algebraic and can be expressed in terms of a rational parametrization which is described in Section~\ref{sec:sing anal} and given explicitly in \cite{CAS1}. 
The asymptotics of the coefficients $z_{p,q}$ are given by
\begin{align*}
z_{p,q}	& \eqv{q} \frac{a_p}{\Gamma(-4/3)} u_c^{-q} q^{-7/3}	\\
a_p		& \underset{p\to\infty}= \frac{b}{\Gamma(-1/3)} u_c^{-p} p^{-4/3} + O(p^{-5/3})
\end{align*}
where $u_c = \frac65(7+\sqrt{7}) t_c$ and $b=-\frac{27}{20}(\frac32)^{2/3}$, and the sequence $(a_p)_{p\ge 0}$ is determined by its generating function $A(u)=\sum_{p\ge 0}a_p u^p$ given by the following rational parametrization:
\begin{equation*}
\left\{\
\begin{aligned}
  u= \hat u(H) &\, :=\, u_c \mB({ 1-\frac23 (1-H)^3 -\frac13 (1-H)^4 }\,
\\A= \hat A(H) &\, :=\, \frac1{10}\m({\frac32}^{7/3}\frac{3H^2-8H+9}{(H^2-3H+3)^2}\,
\end{aligned}
\right. 
\end{equation*}
where $u=0$ and $u=u_c$ correspond to $H=0$ and $H=1$, respectively.
Moreover, for all $u \in \complex$ such that $|u|\le u_c$, we have the asymptotics
\begin{equation*}
Z_q(u) :=  \sum_{p\ge 0} z_{p,q} u^p 
\ \eqv{q}\ \frac{A(u)}{\Gamma(-4/3)} u_c^{-q} q^{-7/3} \,.
\end{equation*}
\end{theorem}

\begin{remark*}
(i) The coefficients $z_{p,q}$ decay with a \emph{perimeter exponent} $7/3$, which is different from the perimeter exponent $5/2$ in the Brownian map universality class. This is in agreement with the behavior of the \emph{volume exponent} in Proposition~\ref{prop:BBM}.
\smallskip


\noindent
(ii) The exponents of the two asymptotics in Theorem~\ref{thm:z_p,q} differ by $7/3-4/3=1$. This difference dictates how the length of the main Ising interface in the Ising-decorated random triangulation scales when the perimeters $p$ and $q$ are large. The value $1$ implies that the length of the interface scales linearly with the perimeter (see in Theorem~\refp{2}{thm:scaling limit} and Proposition~\ref{prop:real perimeter}). If the difference was 0, then the main interface would not grow with the perimeter, and this interface would become a bottleneck in a large Ising-decorated triangulation. In an upcoming work, we will show that such a bottleneck actually appears at low temperatures (i.e.\ when $\nu>\nu_c$).
\end{remark*}

\paragraph{Boltzmann Ising-triangulation and peeling along its interface.} 
Thanks to the finiteness of $z_{p,q}$, we can define a probability measure on $\bts_{p,q}$ by
\begin{equation*}
\prob_{p,q}\bt\ =\ \frac1{z_{p,q}}\, \nu_c^{|\medges \bt|} \, t_c^{|\faces(\tmap)|} \,.
\end{equation*}
Under $\prob_{p,q}$, the law of the spin configuration $\sigma$ conditionally on $\tmap$ is given by the classical Ising model on $\tmap$. And when $\nu=1$, the triangulation $\tmap$ follows the distribution of a Boltzmann triangulation of the $(p+q)$-gon as introduced in \cite{AS03}, with a weight $2t_c$ per internal face. For these reasons we call $\prob_{p,q}$ the law of a (critical) \emph{Boltzmann Ising-triangulation of the $(p,q)$-gon}. The expectation associated to $\prob_{p,q}$ is denoted $\E_{p,q}$.

In order to extract information on the geometry of Boltzmann Ising-triangulations from Theorem~\ref{thm:z_p,q}, we use a peeling process that explores the triangulation along the Ising-interface.\footnote{In order to be tractable, the peeling process has to follow the Ising interface so that the boundary condition remain Dobrushin after any number of peeling steps. See Section~\ref{sec:interface} for details. This is reminiscent to the peeling exploration of a Bernoulli percolation on the UIHPT, see \cite{AC13}.}
More precisely, an \emph{interface} refers to a non-self-intersecting (but not necessarily simple) path formed by non-monochromatic edges. Assuming that the boundary of $\bt$ is not monochromatic, there must be exactly two boundary vertices where the \+ and \< boundary components meet. One of them is the origin $\rho$. We call $\rho^\dagger$ the other one. We denote by $\iroot$ the \emph{leftmost interface from $\rho$ to $\rho^\dagger$} as given in Figure~\refp{a}{fig:rmi}.\,\footnote{The non-monochromatic edges in $\bt$ form a subgraph of $\tmap$ which has even degree at every vertex except for $\rho$ and $\rho^\dagger$. Therefore $\rho$ and $\rho^\dagger$ must belong to the same connected component of this subgraph, that is, there is at least one interface from $\rho$ to $\rho^\dagger$.}

\begin{figure}
\centering
\includegraphics[scale=1]{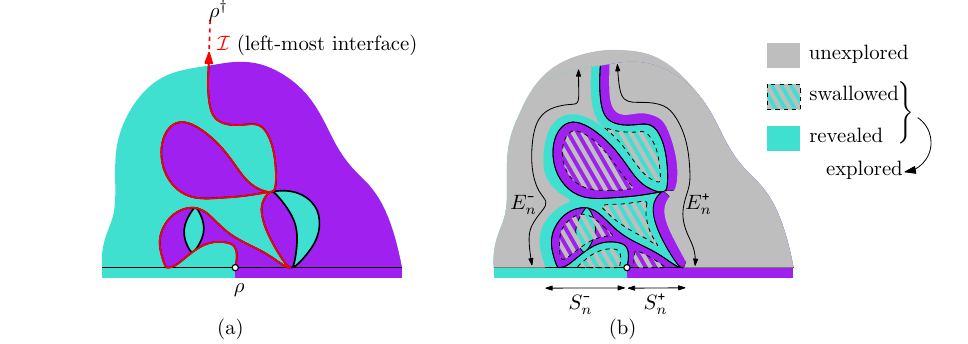}
\caption{(a) The leftmost interface $\iroot$ from $\rho$ to $\rho^\dagger$ in a bicolored triangulation. For clarity, only the non-monochromatic edges connected to $\rho$ are drawn.
(b) An illustration of the map $\emap_n$ explored by the peeling process by the time $n$. Notice that the perimeter variations $(X_n,Y_n)$ can be read from the map $\emap_n$ as $X_n=E^\+_n-S^\+_n$ and $Y_n=E^\<_n-S^\<_n$. }
\label{fig:rmi}
\end{figure}


We will consider a peeling process that explores $\iroot$ by revealing one triangle adjacent to $\iroot$ at each step, and possibly swallowing a finite number of other triangles. Formally, we define the \emph{peeling process} as an increasing sequence of \emph{explored maps} $\nseq \emap$. The precise definition of $\emap_n$ will be left to Section~\ref{sec:peeling}. See Figure~\refp{b}{fig:rmi} for an illustration. 

The peeling process can also be encoded by a sequence of \emph{peeling events} $\nseq[1] \Step$ taking values in some countable set of symbols, where $\Step_n$ indicates the position of the triangle revealed at time $n$ relative to the explored map $\emap_{n-1}$. The detailed definition is again left to Section~\ref{sec:peeling}. 
The sequence $\nseq[1] \Step$ contains slightly less information than $\nseq \emap$, but it has the advantage that its law can be written down fairly easily and one can perform explicit computations with it. We denote by $\Prob_{p,q}$ the law of the sequence $\nseq[1] \Step$ under $\prob_{p,q}$.

In order to understand the geometry of large Boltzmann Ising-triangulations, we want to study the peeling process in the limit $p,q\to\infty$. The regime where $p$ and $q$ go to infinity at comparable speeds is probably the most natural and interesting one. 
However, extracting the asymptotics of $z_{p,q}$ from its generating function $Z(u,v)$ in this limit poses a significant technical challenge. We leave the study of this regime to an upcoming work.
Instead, we will look into the regime where $q$ goes to infinity before $p$. The first step consists of showing that the law of the sequence $\Prob_{p,q}$ converges weakly as follows: 
\begin{proposition}\label{prop:comb cv}
$\Prob_{p,q} \cv[]q \Prob\py \cv[]p \Prob\yy$, where $\Prob\py$ and $\Prob\yy$ are probability distributions.
\end{proposition}
\noindent
The distributions $\Prob\py$ and $\Prob\yy$ will be constructed explicitly in Section~\ref{sec:limit S}, thus no tightness argument is needed in the proof of the above convergence. Geometrically, Proposition~\ref{prop:comb cv} should be understood as the convergence in distribution of the explored map $\emap_n$ for \mbox{any fixed $n$}.

\paragraph{The perimeter processes and their scaling limits.}
One crucial point in the definition of the peeling process is that the \emph{unexplored map}, i.e.\ the complement of the explored \mbox{map $\emap_n$}, remains an Ising-triangulation with Dobrushin boundary condition for all $n$.
We denote by $(P_n,Q_n)$ the boundary condition of the unexplored map at time $n$, and by $(X_n,Y_n)$ its variations, that is, $X_n=P_n-P_0$ and $Y_n=Q_n-Q_0$. Geometrically, $X_n$ (resp.\ $Y_n$) is the number of newly discovered \+ boundary edges (resp.\ \< boundary edges), minus the number of \+ boundary edges (resp.\ \< boundary edges) \emph{swallowed} by the peeling process up to \mbox{time $n$}. See Figure~\refp{b}{fig:rmi}. 

It will be clear from the definition of the peeling process that $(X_n,Y_n)$ is a deterministic function of the peeling events $(\Step_k)_{1\le k\le n}$ with a well-defined limit when $p,q\to\infty$. This allows us to define the law of the process $\nseq{X_n,Y}$ under $\Prob\yy$ despite the fact that $P_n=Q_n=\infty$ almost surely in this case. Similarly, $\nseq{X_n,Y}$ is also well-defined under $\Prob\py$. However, it is easier to study the process $\nseq P$ in this case because it is Markovian under $\Prob\py$. These processes have the following scaling limits.

\begin{theorem}[Scaling limit of the perimeter processes]~\label{thm:scaling limit}\\
(1) Under $\Prob\yy$, the process $\nseq{X_n,Y}$ is a random walk (i.e.\ with i.i.d.\ increments) on $\integer^2$ starting from $(0,0)$. Its two components have the same positive drift: $\EE\yy[X_1] = \EE\yy[Y_1]=\mu:=\frac1{4\sqrt7}$. Moreover, the fluctuation of $\nseq{X_n,Y}$ around its mean has the scaling limit:
\begin{equation*}
\frac1{n^{3/4}}\mb({	 X_{\floor{nt}} -\mu nt,
					 Y_{\floor{nt}} -\mu nt }_{t\ge 0}
\cv[]n \mb({\mathcal X_t,\mathcal Y_t}_{t\ge 0} \,,
\end{equation*}
where $\mathcal X$ and $\mathcal Y$ are two independent spectrally-negative $\frac43$-stable L\'evy processes of L\'evy measure $\frac{c_x}{\abs x^{7/3}}\idd{x<0}\dd x$ and $\frac{c_y}{\abs y^{7/3}}\idd{y<0}\dd y$, for some explicit constants $c_x>c_y>0$.
\smallskip

\noindent
(2) Under $\Prob\py$, the process $\nseq P$ is a Markov chain on $\integer_{\ge 0}$ which starts from $p$ and hits zero almost surely in finite time. It has the following scaling limit:
\begin{equation*}
p^{-1} (P_{\floor{pt}})_{t\ge 0} \cv[]p (\mathcal D_t)_{t\ge 0} \,,
\end{equation*}
where $(\mathcal D_t)_{t\ge 0}$ is the deterministic drift process $(1+\mu t)_{t\ge 0}$ that jumps to zero and stays there after a random time $\zeta$ whose law is given by
\begin{equation*}
\prob(\zeta>t)=(1+\mu t)^{-4/3} \,.
\end{equation*}
Both convergences take place in distribution with respect to the Skorokhod topology.
\end{theorem}

An important point in Theorem~\refp{1}{thm:scaling limit} is that $\mu$, the common drift of $\nseq X$ and $\nseq Y$, is strictly positive so that both $X_n$ and $Y_n$ tend to $+\infty$ when $n\to\infty$. Geometrically, it means that under $\Prob\yy$, the peeling process discovers more and more edges on both sides of the interface $\iroot$ and comes back to the boundary only finitely many times. This is in contrast with the behavior of the percolation interface on uniform random maps of the half plane (e.g.\ the UIHPT) with the same boundary condition, which comes back to the boundary infinitely often (see \cite{Ang05,AC13}). This difference of the interface behaviour is reminiscent to the difference of SLE(3) and SLE(6), which arise respectively as scaling limits of critical Ising and percolation interfaces on regular lattices \cite{CDHKS14,Smi01}. In the case of critical face percolation on the UIHPQ, Gwynne and Miller recently proved that the percolation interface converges towards SLE(6) in a LQG (or Brownian) half plane  \cite{GM18}.

Theorem~\refp{2}{thm:scaling limit} says that on time scales $n\ll p$, the process $\nseq X$ under $\Prob\py$ increases with a drift $\mu$ like under $\Prob\yy$. However on the time scale $n=O(p)$, the effect of the finiteness of the \+ boundary appears and makes $P_n$ hit zero in finite time. Geometrically, the large negative jump of $\nseq P$ corresponds to the first time that the peeling process hits a boundary vertex close to $\rho^\dagger$, swallowing most of the \+ edges on the boundary. 
The random time $\zeta$ should be interpreted as a length: for large $p$, the total length of the interface $\iroot$ under $\Prob\py$ is almost surely finite and roughly $\zeta p$. There is a conjectural interpretation of $\zeta$ as the length of the interface in a gluing of a $\sqrt{3}$-Liouville quantum disk with a thick quantum wedge, in which the perimeter of the quantum disk is sampled from the Lévy measure of a stable process. See \cite{DMS14}, \cite{AG19} for the definitions and basic properties of the aforementioned objects. Whether there is a relationship between the two parts of Theorem~\ref{thm:scaling limit} involved in this interpretation is also an open problem. More discussion on this is given in Section~\ref{sec:interface}.

Notice that in Theorem~\refp{1}{thm:scaling limit}, although the drifts are equal, there is an asymmetry between the fluctuations of the processes $\nseq X$ and $\nseq Y$. This is not surprising because they are defined by the peeling process that explores the \emph{leftmost} interface. Nevertheless, this asymmetry is \emph{not} related to the fact that we have taken first the limit $q\to\infty$ and then the limit $p\to\infty$. In fact, one can check that taking the limit $p\to\infty$ and then $q\to\infty$ yields the same distribution $\Prob\yy$. See the discussion on the peeling process along the rightmost interface in Section~\ref{sec:interface}. We conjecture that the distribution $\Prob\yy$ actually arises when $p,q\to\infty$ at any relative speed.

\begin{conjecture*}
$\Prob_{p,q}\longrightarrow \Prob\yy$ weakly whenever $p,q\to\infty$.
\end{conjecture*}

\paragraph{Local limits and geometry.} Another way to improve Proposition~\ref{prop:comb cv} is to strengthen it to the local convergence of the underlying map. The local distance between bicolored maps is a straightforward generalization of local distance between uncolored maps:
\begin{equation*}
d\1{loc}(\bt,\bt[']) = 2^{-R}\qtq{where}
	R = \sup\Set{r\geq 0}{ \btsq_r=\btsq[']_r }
\end{equation*}
and $\btsq_r$ denotes the ball of radius $r$ around the origin in $\bt$ which takes into account the colors of the faces. See Section~\ref{sec:def P(p)} for a more precise definition of $\btsq_r$. Similarly to the uncolored maps, the set $\bts$ of (finite) bicolored triangulations of polygon is a metric space under $d\1{loc}$. Let $\overline{\bts}$ be its Cauchy completion. 

Recall that an (infinite) graph is \emph{one-ended} if the complement of any finite subgraph has exactly one infinite connected component. It is well known that a one-ended map has either zero or one face of infinite degree \cite{CurPeccot}. We call an element of $\overline{\bts}\setminus \bts$ a \emph{bicolored triangulation of the half plane} if it is one-ended and its external face has infinite degree. Such a triangulation has a proper embedding in the upper half plane without accumulation points and such that the boundary coincides with the real axis, hence the name. We denote by $\bts_\infty$ the set of all bicolored triangulations of the half plane.

\begin{theorem}[Local limits of Ising-triangulation]~\label{thm:cv}\\
(1) There exist probability distributions $\prob\py$ and $\prob\yy$ supported on $\bts_\infty$, such that
\begin{equation*}
\prob_{p,q}\ \cv q\ \prob_p\ \cv p\ \prob_\infty
\end{equation*}
weakly. In addition, if $\prob_\pqq$ denotes the pushforward of $\prob_{p,q_1+q_2}$ by the mapping that translates the origin $q_1$ edges to the left along the boundary, then for all fixed $p\ge 0$, we have $\prob_\pqq \xrightarrow{d\1{loc}} \prob_0$ weakly as $q_1,q_2\to\infty$.
\smallskip

\noindent
(2)	$\prob_p$-almost surely, $\bt$ contains only one infinite spin cluster, which is of spin \<.
\smallskip

\noindent
(3) $\prob_\infty$-almost surely, 
$\bt$ contains exactly two infinite spin clusters. One of them is of spin \+ on the right of the root, and the other is of spin \< on the left of the root. They are separated by a strip of finite clusters, which only touches the boundary of $\bt$ in a finite interval.
\end{theorem}

\begin{figure}[t]
\centering
\includegraphics[scale=1,page=7]{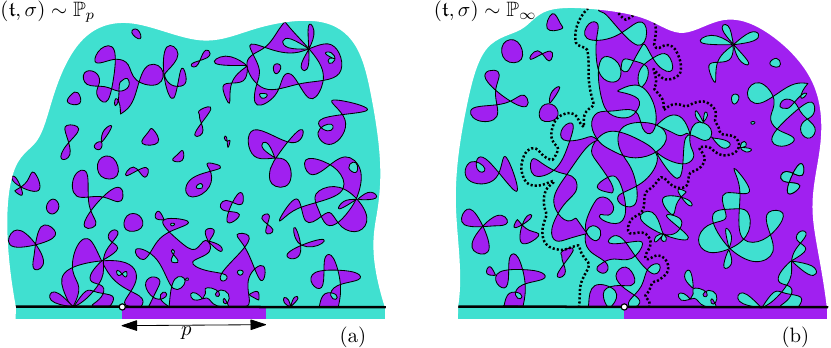}
\caption{An artistic representation of the cluster structure of an Ising-triangulation of distribution $\prob\py$ and $\prob\yy$. The dashed lines in (b) highlight the strip of finite clusters separating the two infinite clusters.}
\label{fig:cluster-topology}
\end{figure}

See Figure~\ref{fig:cluster-topology} for an illustration of the cluster structure in the Ising-triangulations of laws $\prob\py$ and $\prob\yy$.

The construction of the limits $\prob\py$ and $\prob\yy$ is based on the laws $\Prob\py$ and $\Prob\yy$ of the peeling process in Proposition~\ref{prop:comb cv}.
Under $\Prob\py$, one can extend the peeling process after it finishes exploring the leftmost interface $\iroot$, in such a way that the explored map $\emap_n$ eventually covers all the internal faces of the Ising-triangulation. Consequently $\prob\py$ can be constructed directly as the law of the union $\cup_{n\ge 0} \emap_n$ under $\Prob\py$. However, almost surely under $\Prob\yy$, the interface $\iroot$ is infinite and visits the boundary of $\bt$ only finitely many times (see the discussion after Theorem~\ref{thm:scaling limit}). Thus the peeling process only explores the faces of $\bt$ along a strip around the interface $\iroot$. 
For this reason, the Ising-triangulation of law $\prob\yy$ is constructed by gluing two infinite bicolored triangulations to both sides of the strip given by $\cup_{n\ge 0} \emap_n$ under $\Prob\yy$.
The proof of the convergences in Theorem~\refp{1}{thm:cv} follows closely the above construction of the distributions $\prob\py$ and $\prob\yy$. The structure of the proof is summarized in Figure~\ref{fig:proof-scheme} at the beginning of Section~\ref{sec:metric peeling}.
The statements \hyperref[thm:cv]{(2)} and \hyperref[thm:cv]{(3)} of Theorem~\ref{thm:cv} are direct consequences of our construction of the distributions $\prob\py$ and $\prob\yy$. More discussions about them, as well as about other properties of the spin clusters under $\prob\py$ and $\prob\yy$, will be given in Section~\ref{sec:interface}.

\paragraph{Related works}
This paper has been greatly inspired by the work \cite{BBM11} of Bernardi and Bousquet-M\'elou, which computed (among other things) the partition function of Ising-decorated triangulations with spins on the vertices and a fixed boundary length. As mentioned in Proposition~\ref{prop:BBM}, the same work also gives the critical temperature that this article focuses on. Results in \cite{BBM11} will also be used in Section~\ref{sec:z solution} to derive the partition functions $z_1(\nu,t)$ and $z_3(\nu,t)$, bypassing a tricky computation which will be included as Appendix~\ref{sec:2nd cat}.

Another important source of enumerative results on Ising-decorated triangulations is Chapter~8 of the book \cite{EynardBook} by Eynard. The chapter describes a method for enumerating extremely general Ising-decorated maps, including features like external magnetic field for the Ising model, mixed boundary conditions, and maps with several boundaries in higher genera. Our method for eliminating the first catalytic variable described in Section~\ref{sec:1st cat} can actually be viewed as a special case of the method used in \cite{EynardBook}, although the difference in presentation makes this link hard to see. 
It is also possible to derive the rational parametrizations \eqref{eq:RP H} and \eqref{eq:RP Z} of $Z(u,v)$ using the method of \cite{EynardBook}, provided that one properly relates the generating functions of triangulations with non-simple boundary (the setting in \cite{EynardBook}) and those of triangulations with simple boundary.

A similar model of Ising-decorated triangulations is studied in the recent independent work \cite{AMS18} by Albenque, M\'enard and Schaeffer. To be precise, for each $n$, they consider the set of triangulations of the \emph{sphere} with $n$ edges and decorated by spins on the \emph{vertices}, in which each monochromatic edge is given a weight $\nu$. They show that for \emph{any} fixed $\nu>0$, the law $\prob_n$ of the random triangulation thus obtained converges weakly for the local topology when $n\to\infty$. They follow an approach akin to the one used by Angel and Schramm to construct the UIPT \cite{AS03}, namely, showing that all finite dimensional marginals of $\prob_n$ converge, and that the family $(\prob_n)_{n\ge 1}$ is tight. This is very different from our approach: here we use the peeling process to construct explicitly the limit distribution (which is a probability), so a tightness argument is unnecessary. From a combinatorics viewpoint, they first establish asymptotics as $n\to\infty$ of the partition function of their model with a Dobrushin boundary. For this purpose, they use Tutte's invariants to solve an equation with two catalytic variables, similarly to \cite{BBM11}. Then they use some recursion relation (which can be understood as the peeling of an Ising-triangulation with an arbitrary boundary condition) to show that the partition function defined by \emph{any} fixed boundary condition also has the same asymptotic behavior when $n\to\infty$. This last step is crucial for their proof of the finite-dimensional-marginal convergence of $(\prob_n)_{n\ge 1}$. Moreover, they show that the simple random walk on the local limit is almost surely recurrent.

\paragraph{Outline.}
The rest of the paper is organized as follows.
We derive the so-called Tutte's equation (or loop equation) satisfied by $Z(u,v)$ in Section~\ref{sec:Tutte derivation} and define the peeling process of a bicolored triangulation of the $(p,q)$-gon in Section~\ref{sec:peeling}. The derivation is formulated in probabilistic language to highlight its relation with the first step of the peeling process. For our model, Tutte's equation is a functional equation with two catalytic variables. In Section~\ref{sec:1st cat} we eliminate one of the catalytic variables by coefficient extractions, leading to a functional equation with one catalytic variable for $Z(u,0)$. Section~\ref{sec:z solution} details the connection between our model and a model studied in \cite{BBM11}, which is then used to translate some of their results (in particular Proposition~\ref{prop:BBM}) in our setting. These results can also be obtained independently via a trick due to Tutte, which is presented in the Appendix~\ref{sec:2nd cat}. Section~\ref{sec:sing anal} solves the functional equation on $Z(u,v)$ at the critical point $(\nu,t)=(\nu_c,t_c)$ by a rational parametrization, and completes the proof of Theorem~\ref{thm:z_p,q} with standard methods of singularity analysis. Some specific techniques for conducting singularity analysis using rational parametrizations are summarized in Appendix~\ref{sec:RP}.

Section~\ref{sec:limit peeling} is devoted to the study of the limits of the peeling process and the associated perimeter processes, and the proof of Theorem~\ref{thm:scaling limit}. It also includes an important one-jump lemma of the perimeter processes, which is proven in Appendix~\ref{sec:lemma proof}.  In Section~\ref{sec:metric peeling} we construct the distributions $\prob\py$ and $\prob\yy$ and prove the local convergences in Theorem~\refp{1}{thm:cv}. Finally, we discuss in Section~\ref{sec:interface} some properties of the spins clusters and the interfaces that follows from our construction of the infinite Ising-triangulation of law $\prob\py$ and $\prob\yy$. It contains the proof of Theorem~\refp{2-3}{thm:cv} and a  scaling limit result for the perimeter of a spin cluster.

\section{Tutte's equation and peeling along the interface} \label{sec:Tutte equation}

Recall that we have fixed the critical parameters $(\nu_c,t_c)$ and defined $Z(u,v)=\sum_{p,q\ge 0} z_{p,q} u^p v^q$ with $z_{0,0}=1$ and $z_{p,q}=z_{p,q}(\nu_c,t_c)$ for $p+q\ge 1$. However, many of the discussions below will be valid for any $\nu,t>0$ such that $z_{p,q}(\nu,t)<\infty$. In this case we will write $(\nu,t)$ instead of $(\nu_c,t_c)$.

\begin{figure}[b!]
\centering
\includegraphics[scale=1,page=2]{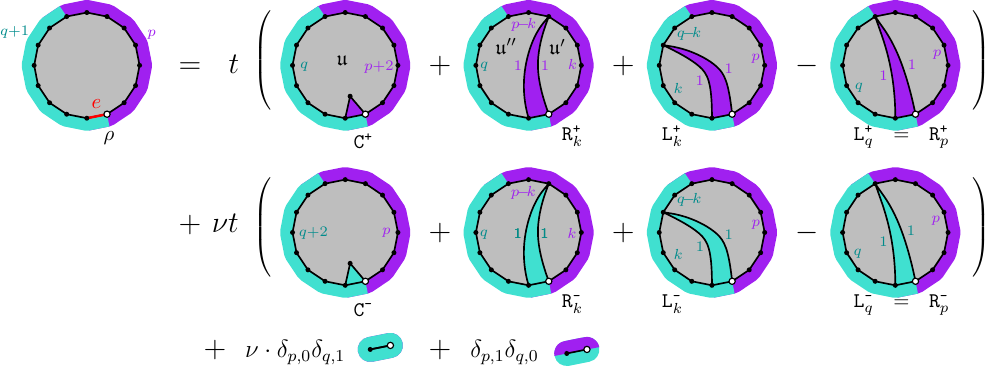}
\caption{A graphical representation of the derivation of Tutte's equation.}\label{fig:Tutte-eq}
\end{figure}

The primary goal of this section is to derive a recurrence relation for the double sequence $(z_{p,q})_{p,q\ge 0}$, and then a functional equation --- the so-called Tutte's equation (a.k.a.\ loop equation, or Schwinger-Dyson equation) --- for its generating function. The basic idea, which goes back to Tutte \cite{Tuttecp3}, is to consider the removal of one face on the boundary, which relates one bicolored triangulation of polygon to other ones with fewer faces.
We will present a probabilistic derivation of Tutte's equation. This is a bit more cumbersome than a direct combinatorial derivation, but will shed light on the relation between Tutte's equation and the peeling process, which we define in the second half of this section.

\subsection{Derivation of Tutte's equation}\label{sec:Tutte derivation}

Let $p,q\ge 0$ so that the bicolored triangulation $\bt\in \bts_{p,q+1}$ has at least one boundary edge with spin \<. We remove the boundary edge $e$ immediately on the left of the origin (which has spin \<) and reveal the internal face $f$ adjacent to it. It is possible that $f$ does not exist if $(p,q+1)=(1,1)$ or $(0,2)$. In this case $\tmap$ is the edge map and $\bt$ has a weight 1 or $\nu$. When $f$ does exist, let $*\in\{\+,\<\}$ be the spin on $f$ and $v$ be the vertex at the corner of $f$ not adjacent to $e$. There are three possibilities for the position of $v$.
\begin{description}[noitemsep]
\item[Event $\CC^*$:]	$v$ is not on the boundary of $\tmap$;
\item[Event $\RR^*_k$:] $v$ is at a distance $k$ to the right of $e$ on the boundary of $\tmap$; ($0\le k\le p$);
\item[Event $\LL^*_k$:] $v$ is at a distance $k$ to the left of $e$ on the boundary of $\tmap$. ($0\le k\le q$).
\end{description}
These events, as well as the discussion below, are illustrated in Figure~\ref{fig:Tutte-eq}.

When the event $\CC^*$ occurs, the unexplored part of $\bt$, denoted $\umap$, is again a bicolored triangulation of polygon. If $*=\+$, then $\umap$ has the boundary condition $\+^{p+2}\<^q$ and the numbers of monochromatic edges and internal faces in $\umap$ are respectively $\abs{\medges\bt}$ and $ \abs{\faces(\tmap)}-1$. It follows that for all $\bt[_0]\in\bts_{p+2,q}$,
\begin{equation*}
\prob_{p,q+1}(\cp\text{ and } \umap=\bt[_0])
\ =\ \zinv \weight[][+1]{\tmap_0}{,\sigma_0}
\ =\ t\zz{p+2,q} \cdot \frac{\weight{\tmap_0}{,\sigma_0}}{z_{p+2,q}} \,.
\end{equation*}
In other words, $\prob_{p,q+1}(\cp) = t\zz{p+2,q}$ and conditionally on $\cp$, the law of $\umap$ is $\prob_{p+2,q}$. Similarly when $*=\<$, we have $\prob_{p,q+1}(\cm) = \nu t\zz{p,q+2}$ and conditionally on $\cm$, the law of $\umap$ is $\prob_{p,q+2}$.

When the event $\rp$ occurs for some $0\le k\le p$, the vertex $v$ is on the \+ boundary of $\bt$, and the unexplored part is made of two bicolored triangulations of polygons joint together at the vertex $v$. We denote by $\umap'$  the right one and by $\umap''$ the left one. Then $\umap'$ has the boundary condition $\+^{k+1}$ and $\umap''$ the boundary condition $\+^{p+1-k}\<^q$. Again one can relate the numbers of monochromatic edges and of internal faces in $\umap'\cup \umap''$ to $\medges\bt$ and $\faces(\tmap)$. It then follows that for all $\bt[']\in\bts_{k+1,0}$ and $\bt['']\in\bts_{p+1-k,q}$,
\begin{align*}
\prob_{p,q+1}\big( \rp,\ \umap'=\bt[']
			\text{ and }\umap''=\ &\bt[''] \big)
\ =\ \zinv\, \weight[ +\abs{\medges\bt['']} ][ +\abs{\faces(\tmap'')}+1 ]{\tmap'}{,\sigma'}
\\&=\ t\zzz{k+1,0}{p+1-k,q} \cdot \frac{\weight{\tmap'}{,\sigma'}}{z_{k+1,0}} \cdot \frac{\weight{\tmap''}{,\sigma''}}{z_{p+1-k,q}}\,.
\end{align*}
In other words, $\prob_{p,q+1}(\rp) = t\zzz{k+1,0}{p+1-k,q}$ and conditionally on $\rp$, the maps $\umap'$ and $\umap''$ are independent and follow respectively the laws $\prob_{k+1,0}$ and $\prob_{p+1-k,q}$.

Similarly, one can work out the probabilities that the events $\rn$ ($0\le k\le p$) or $\LL^\jj_k$ ($0\le k\le q$) occur:
\begin{equation*}
		\prob_{p,q+1}(\rn) = \nu t\zzz{k,1}{p-k,q+1}
\qquad	\prob_{p,q+1}(\lp) = t\zzz{p+1,q-k}{1,k}
\qquad	\prob_{p,q+1}(\lm) = \nu t\zzz{p,q-k+1}{0,k+1}\,.
\end{equation*}
In each case, the unexplored part consists of two bicolored triangulations of some polygons which are conditionally independent and follow the law of Boltzmann Ising-triangulations of appropriate Dobrushin boundary conditions (See Figure~\ref{fig:Tutte-eq}). Tutte's equation simply expresses the fact that the probabilities of the events $\cp,\cm,\lp,\lm,\rp,\rn$ under $\prob_{p,q+1}$ sum to 1:
\begin{align*}
1\ &=\
\prob_{p,q+1}(\cp) + \sum_{k=0}^p \prob_{p,q+1}(\rp) + \sum_{k=0}^q \prob_{p,q+1}(\lp) - \prob_{p,q+1}(\LL^\+_q) + \frac1{z_{1,1}}\delta_{p,1}\,\delta_{q,0}
\\&\ \,+ \prob_{p,q+1}(\cm) + \sum_{k=0}^p \prob_{p,q+1}(\rn) + \sum_{k=0}^q \prob_{p,q+1}(\lm) - \prob_{p,q+1}(\LL^\<_q) + \frac\nu{z_{0,2}}\delta_{p,0}\,\delta_{q,1} \,.
\end{align*}
In each line on the right hand side of this equation, the last term corresponds to the case where $\tmap$ is the edge map, which is a special case that does not belong to any of the events above. The negative term is needed to compensate for the fact that $\RR^*_p$ and $\LL^*_q$ actually represent the same event. Multiplying both sides by $z_{p,q+1}$ yields the following recurrence relation, valid for all $p,q\ge 0$:%
\newcommand*{\sumon}[1]{\!\!\sum_{#1_1+#1_2=#1}\!\!\!}%
\begin{align*}
z_{p,q+1} =&\
	\, t	\mB({ z_{p+2,q}
		+ \sumon{p} z_{p_1+1,0}\,z_{p_2+1,q} + \sumon{q} z_{1,q_1}\,z_{p+1,q_2}
		- z_{p+1,0}\,z_{1,q}	} + \delta_{p,1}\,\delta_{q,0}
\\+&\, 	\nu	 t \mB({ z_{p,q+2}
		+ \sumon{q} z_{0,q_1+1}\,z_{p,q_2+1} + \sumon{p} z_{p_1,1}\,z_{p_2,q+1}
		- z_{p,1}\,z_{0,q+1}	} + \nu\, \delta_{p,0}\,\delta_{q,1} \,,
\end{align*}
where $p_1,p_2,q_1,q_2$ are summed over non-negative values. Summing the last display over $p,q\ge 0$, we get Tutte's equation satisfied by $Z(u,v)$. 
By exchanging $u$ and $v$ we obtain another functional equation of $Z$. The two equations can be written compactly as the following linear system.
\begin{equation}\label{eq:2cat}
\begin{bmatrix} \Delta_u Z \\ \Delta_v Z \end{bmatrix} =
\begin{bmatrix}
	\nu	&	\!1
\\	1	&	\!\nu
\end{bmatrix}
\begin{bmatrix}
t\mB({ \Delta_u^2 Z + \mb({\Delta_u Z_0(u) + Z_1(v)} \Delta_u Z
						- \Delta_u Z_0(u)   Z_1(v) 			} + u\, \\
t\mB({ \Delta_v^2 Z + \mb({\Delta_v Z_0(v) + Z_1(u)} \Delta_v Z
						- \Delta_v Z_0(v)   Z_1(u) 			} + v\,
\end{bmatrix}
\end{equation}
where we write $Z=Z(u,v)$ and $Z_k(u) = [v^k]Z(u,v)$ for short, and $\Delta_x f(x) = \frac{f(x)-f(0)}x$ denotes the discrete derivative with respect to the variable $x\in\{u,v\}$. Geometrically, the other equation in the system describes the removal of a boundary edge with spin \+ next to the origin. When viewed as a system of algebraic equations, the list of unknowns of \eqref{eq:2cat} contains not only the generating function $Z\equiv Z(u,v)$, but also its coefficients in the variables $u$ and $v$, namely $Z_0(u)=[v^0]Z$, $Z_1(u)=[v^1]Z$ and $Z_0(v)=[u^0]Z$, $Z_1(v)=[u^1]Z$. For this reason, $u$ and $v$ are called \emph{catalytic variables}. 

The system \eqref{eq:2cat} will be the starting point of the asymptotic analysis of the double sequence $(z_{p,q})_{p,q\ge0}$ in Section~\ref{sec:Tutte solution}. But let us first turn our attention to the geometric implications of the above derivation of Tutte's equation and define the peeling process mentioned in the introduction.

\subsection{Peeling exploration of the leftmost interface}\label{sec:peeling}

\begin{figure}[t!]
\centering
\includegraphics[scale=1]{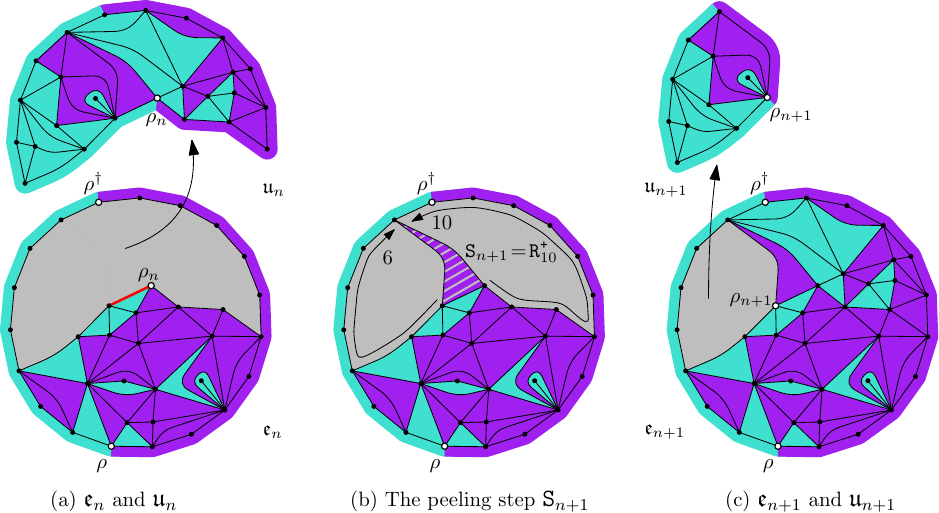}
\caption{An example of the $n$-th and the $(n+1)$-th steps of a peeling process.
The unexplored map $\umap_n$ is rooted at $\rho_n$, similarly for $\umap_{n+1}$.
The peeling step $\Step_{n+1}$ is $\rp[10]$ rather than $\lp[6]$ because we choose to fill the unexplored region on the right.
}\label{fig:def-peeling}
\end{figure}

\newcommand*{\front}{\Map f}
The peeling process along the leftmost interface $\iroot$ is constructed by iterating the face-revealing operation used in the derivation of Tutte's equation. Formally, we define the peeling process as an increasing sequence $\nseq \emap$ of \emph{explored maps}. At each time $n$, the explored map $\emap_n$ consists of a subset of faces of $\bt$ containing at least the external face and separated from its complementary set by a simple closed path. We view $\emap_n$ as a bicolored triangulation of a polygon with a special uncolored internal face (not necessarily triangular) called the \emph{hole}. It inherits its root and its boundary condition from $\bt$. The complementary of $\emap_n$ is called the \emph{unexplored map at time $n$} and denoted $\umap_n$. It is a bicolored triangulation of a polygon (without holes).\footnote{In the literature the peeling process is sometimes defined as the sequence of unexplored map $\nseq \umap$ or as the sequence of closed paths that separate $\emap_n$ and $\umap_n$. For a given $\bt$, these sequences all contain the same information. However, it is $\nseq \emap$ that generates the filtration that makes the peeling process Markovian.}
Notice that $\umap_n$ may be the edge map, in which case $\emap_n$ is simply $\bt$ in which an edge is replaced by an uncolored digon. However, this may only happen at the last step of the peeling process (see below).

We have seen in Figure~\ref{fig:Tutte-eq} that revealing an internal face on the boundary splits $\bt$ into \emph{one or two} unexplored regions delimited by closed simple paths. To iterate this face-revealing operation, one needs a rule that chooses one of the two unexplored regions, when there are two, as the next unexplored map. At first glance, the natural choice would be to keep the unexplored region containing $\rho^\dagger$, the end point of the interface $\iroot$. However, this choice does not fit well with the limit $q\to\infty,\, p\to\infty$ that we would like to take. Instead, we choose the unexplored region with greater number of \< boundary edges (in case of a tie, choose the region on the right). This guarantees that when $q=\infty$ and $p<\infty$, we will automatically choose the unbounded region as the next unexplored map.

We apply this rule inductively to build the peeling process starting from $\umap_0=\bt$. At each step, the construction proceeds differently depending on the boundary condition of $\umap_n$:
\begin{enumerate}
\item	If $\umap_n$ has a non-monochromatic Dobrushin boundary condition, let $\rho_n$ be the boundary vertex of $\umap_n$ with a \< on its left and a \+ on its right ($\rho_0=\rho$). Then $\umap_{n+1}$ is obtained by revealing the internal face of $\umap_n$ adjacent to the boundary edge on the left of $\rho_n$ and, if necessary, choose one of the two unexplored regions according to the \mbox{previous rule}.
Figure~\ref{fig:def-peeling} gives a possible realization of the peeling process in this case.
\item	If $\umap_n$ has a monochromatic boundary condition of spin \<, then we choose the boundary vertex $\rho_n$ according to some deterministic function $\algo$ of the explored map $\emap_n$, called the \emph{peeling algorithm}, which we specify later in Section~\ref{sec:metric peeling}. We then construct $\umap_{n+1}$ from $\umap_n$ and $\rho_n$ in the same way as in the previous case.
\item	If $\umap_n$ has a monochromatic boundary condition of spin \+ or has no internal face (i.e.\ it is the edge map), then we set $\emap_{n+1}=\bt$ and terminate the peeling process at time $n+1$.
\end{enumerate}
We will explain why the above construction defines the peeling exploration of the \emph{leftmost} interface in Section~\ref{sec:interface}.

By induction, $\umap_n$ always has a Dobrushin boundary condition. As mentioned in the introduction, $(P_n,Q_n)$ denotes the boundary condition of $\umap_n$, and $(X_n,Y_n)=(P_n-P_0,Q_n-Q_0)$. Also, $\Step_n$ denotes the peeling event that occurred when constructing $\umap_n$ from $\umap_{n-1}$, which takes values in the set of symbols $\steps = \{\cp,\cm\}\cup\{\lp,\lm,\rp,\rn: k\ge 0\}$. The above quantities are all deterministic functions of the bicolored triangulation $\bt$. We view them as random variables defined on the sample space $\Omega=\bts=\bigcup_{p,q} \bts_{p,q}$.

\newcommand*{\hl}{\\\cline{1-6}}
\tabulinesep=1.2mm
\newcolumntype{L}{>{\!$\displaystyle}l<{$\!}}
\newcolumntype{S}{>{\!$\displaystyle}l<{$\!\!}}
\newcolumntype{R}{>{\!$\displaystyle}r<{$\!}}
\newcolumntype{C}{>{\!$\displaystyle}c<{$\!}}

\newcommand*{\zp}[1]{t\zz{#1}}
\newcommand*{\zzp}[2]{t\zzz{#1}{#2}}

\newcommand*{\zq}[1]{t\zz[p,q+1]{#1}}
\newcommand*{\zqo}[1]{t\zz[0,q+1]{#1}}

\newcommand*{\zzq}[2]{t\zzz[p,q+1]{#1}{#2}}
\newcommand*{\zzqo}[2]{t\zzz[0,q+1]{#1}{#2}}

\newcommand*{\ap}[1]{\,\frac{a_{#1}}{a_p}\,}
\newcommand*{\apo}[1]{\,\frac{a_{#1}}{a_0}}

\newcommand{\pq}[1][P_n,Q_n]{\raisebox{-2pt}{$\!_{#1}\!$}}

\begin{table}[t!]
\centering
\begin{tabu}{|L|S|C|  |L|S|C| L}
\cline{1-6}
\step &\prob_{p,q+1}(\Step_1 = \step) &(X_1,Y_1) &
\step &\prob_{p,q+1}(\Step_1 = \step) &(X_1,Y_1)
\hl \cp	& 		\zp{p+2,q}				&(2,-1)
&	\cm	& \nu \zp{p,q+2} 				&(0,1)
\hl \lp	& 		\zzp{p+1,q-k}{1,k}		&(1,-k-1)
&	\lm	& \nu 	\zzp{p,q-k+1}{0,k+1}		&(0,-k)	&	(0\le k \le \frac{q}{2})
\hl	\rp	& 		\zzp{k+1,0}{p-k+1,q}		&(-k+1,-1)
&	\rn & \nu 	\zzp{k,1}{p-k,q+1}		&(-k,0)	&	(0\le k\le p)
\hl \rp[p+k]	& 		\zzp{p+1,k}{1,q-k}	&(-p+1,-k-1)
&	\rn[p+k]	& \nu 	\zzp{p,k+1}{0,q-k+1}	&(-p,-k)
&(0<k<\frac{q}2)
\hl
\end{tabu}
\caption{Law of the first peeling event $\Step_1$ under $\prob_{p,q+1}$ and the corresponding $(X_1,Y_1)$.
}\label{tab:prob(p,q)}
\end{table}

According to the discussion in the derivation of Tutte's equation, under the probability $\prob_{p,q}$ and conditionally on $(P_n,Q_n)$, the unexplored map $\umap_n$ is a Boltzmann Ising-triangulation of the $(P_n,Q_n)$-gon --- this is called the \emph{spatial Markov property} of $\prob_{p,q}$. In particular, the pair $(P_n,Q_n)$ determines the conditional law of $\Step_{n+1}$ in the same way as $(p,q)$ determines the law of $\Step_1$, and the peeling event $\Step_{n+1}$ determines the increment $(P_{n+1}-P_n,Q_{n+1}-Q_n)$ in the same way as $\Step_1$ determines $(X_1,Y_1)$. It follows that:
\begin{enumerate}
\item	Both $\nseq{P_n,Q}$ and $\nseq{X_n,Y}$ are adapted to the filtration generated by $\nseq[1]\Step$.
\item	$\nseq{P_n,Q}$ is a Markov chain under $\Prob_{p,q}$, which we recall is the law of $\nseq[1]\Step$ under $\prob_{p,q}$. Its transition probabilities can be deduced from Table~\ref{tab:prob(p,q)}.
\item	The mapping $\nseq[1] \Step \mapsto \nseq{X_n,Y}$ has a well-defined limit when $p,q\to\infty$.
\end{enumerate}

Notice that the law $\Prob_{p,q}$ is completely determined by the data in Table~\ref{tab:prob(p,q)}, and in particular is independent of the peeling algorithm $\algo$. In particular all our results on the limit of $\Prob_{p,q}$ and of the perimeter processes are independent of the peeling algorithm.
The choice of $\algo$ will only become important in the construction of the local limits $\prob\py$ and $\prob\yy$, and will be specified in  Section~\ref{sec:def P(p)}. This independence reflects the invariance of the law of a Boltzmann Ising-triangulation with monochromatic boundary condition under the change of origin. A similar observation was made for the peeling of non-decorated maps in \cite{CLGpeeling}.

In order to study the limits of $\Prob_{p,q}$, let us first solve Tutte's equation and derive the asymptotics of $(z_{p,q})_{p,q\ge 0}$ stated in Theorem~\ref{thm:z_p,q}.

\section{Solution of Tutte's equation}\label{sec:Tutte solution}

Inverting the matrix on the right hand side of \eqref{eq:2cat}, we obtain the following equations:
\begin{align}\label{eq:Tutte1}
(\nu^2-1)^{-1} (\nu \Delta_u Z - \Delta_v Z)	&=
  u + t\mb({ \Delta_u^2 Z + \mb({\Delta_u Z_0(u) + Z_1(v) } \Delta_u Z
							 -  \Delta_u Z_0(u)   Z_1(v) 			}\!\!\!\!\ \\\label{eq:Tutte2}
(\nu^2-1)^{-1} (\nu \Delta_v Z - \Delta_u Z)	&=
  v + t\mb({ \Delta_v^2 Z + \mb({\Delta_v Z_0(v) + Z_1(u) } \Delta_v Z
							 -  \Delta_v Z_0(v)   Z_1(u) 			}\!\!\!\!\
\end{align}
Remark that both equations are affine in $Z$. Solving the first one gives the following expression of $Z$ as a rational function of the univariate series $Z_0$ and $Z_1$:
\begin{equation}\label{eq:Zexplic}
Z(u,v) = Z_0(v) +
\frac{ u(Z_0(v)-Z_0(u)) + (\nu^2-1) v (u^2-tZ_0(u)Z_1(v))
	}{ \nu v - u - (\nu^2-1) t v \mb({ \frac{Z_0(u)}{u} + Z_1(v) }	}\,.
\end{equation}

\subsection{Elimination of the first catalytic variable}\label{sec:1st cat}

It turns out one can obtain a closed functional equation for $Z_0(u)$ by coefficient extraction. More precisely, by extracting the coefficients of $v^0$ and $v^1$ in \eqref{eq:Tutte1} and \eqref{eq:Tutte2}, seen as formal power series in $v$, we get four algebraic equations between $Z_i(u)$ ($i=0,1,2,3$) and $u$, with coefficients in $\complex[\nu][[t]]$:%
\begin{align}\label{eq:Z0}
	(\nu^2-1)^{-1}(\nu\Delta Z_0 - Z_1) &= t\m({ \Delta^2 Z_0 + (\Delta Z_0)^2 } + u
\\	(\nu^2-1)^{-1}(\nu\Delta Z_1 - Z_2) &=
	t\m({ \Delta^2 Z_1 + \Delta Z_0 \Delta Z_1 + z_1 \Delta Z_1	} \label{eq:Z1}
\\	(\nu^2-1)^{-1}(\nu Z_1 -\Delta Z_0) &= t\m({ Z_2 + Z_1^2 }    \label{eq:Z2}
\\	(\nu^2-1)^{-1}(\nu Z_2 -\Delta Z_1) &= t\mb({ Z_3 + Z_1 Z_2 + z_1 Z_2 } + 1
\label{eq:Z3}
\end{align}
where we write $\Delta Z_i = \Delta_u Z_i(u)$ and $z_i=z_{i,0}=z_{0,i}$ for short. Notice that only \eqref{eq:Z3} contains the unknown $Z_3$, so it can be discarded without loss. On the other hand, the equations \eqref{eq:Z0} and \eqref{eq:Z1} are linear in $(Z_1,Z_2)$. So we can easily solve them and plug the results into \eqref{eq:Z2} to obtain a polynomial equation on $Z_0(u)$ of the form: (see \cite{CAS1} for details of the computation)
\begin{equation*}
	\hat{\mathcal P}(Z_0(u), u, z_1, z_{1,1}; \nu,t ) = 0\,.
\end{equation*}
This is not yet a closed functional equation for $Z_0(u)$ because it involves the series $z_{1,1}$ which is \emph{a priori} not related to $Z_0(u)$. (It comes from the term $\Delta^2_u Z_1(u) = \frac{Z_1(u) - z_1 - u z_{1,1} }{u^2}$ in \eqref{eq:Z1}.)
To relate them, we can view the above equation as a formal power series in $u$, and extract its coefficients. The first two non-zero coefficients yield two equations relating $z_{1,1}$ to $z_i$ ($i=1,2,3$) and which are linear in $(z_{1,1},z_2)$. Solving them gives
\begin{equation*}
z_{1,1}	= (\nu^2-1)(2t\, z_1^3-t\,z_3-1) -(3\nu-2) z_1^2 + \frac{\nu  z_1}{(\nu+1)t}\,.
\end{equation*}
Plugging this into $\hat{\mathcal P}=0$ yields a closed functional equation (with one catalytic variable) satisfied by $Z_0(u)$.
This equation can be written as
\begin{equation}\label{eq:1cat}
Z_0(u) = 1+ \nu u^2 + t\, \mathcal{R}(Z_0(u),u,z_1,z_3;\nu,t)
\end{equation}
where the rational function $\mathcal R=\mathcal R(y,u,z_1,z_3;\nu,t)$ is given by (See \cite{CAS1})
\begin{align}\label{eq:R as rational}
\mathcal R=\ &(\nu^2-1)^2t^2 \mB({ (y-1)\m({\frac{y}u}^3-z_1 \m({\frac{y}u}^2 -\frac{1+y-2y^2+z_1 u}t +2z_1^3-z_3 }
\\&	-(\nu^2-1)t \mB({ 2\nu(y-1)\m({\frac{y}u}^2 -(\nu+1)z_1 \frac{y}u + 3(\nu-1)z_1^2 }	\notag
\\&	+\nu(\nu+1)(y-1)\frac{y}u -\nu(\nu^2-1)u(2y-1) +\nu(\nu-3)z_1 +(\nu^2-1)^2u^3. \notag
\end{align}
Notice that $\mathcal R(Z_0(u),u,z_1,z_3;\nu,t)$ is a formal power series of $t$ with coefficients in $\complex(\nu,u)$. Therefore \eqref{eq:1cat} determines $Z_0(u)$ order by order as a formal power series in $t$. According to the general theory on polynomial equations with one catalytic variable \cite[Theorem 3]{BMJ06}, the generating function $Z_0(u)$ is algebraic.%
\footnote{To apply literally \cite[Theorem 3]{BMJ06} to \eqref{eq:1cat}, we must be able to write $\mathcal R$ as a polynomial function of the discrete derivatives $\Delta^i Z_0(u)$ ($i\ge 0$) and the parameters $u,t$, which is not obvious here. However, 
we can multiply both sides of \eqref{eq:1cat} by $u^3$, and view it as a functional equation for the unknown $\mathcal Z_0(u):=u^3 Z_0(u)$. Then, since $z_i = \Delta^i Z_0(u) - u\cdot \Delta^{i+1} Z_0(u)$ and $\Delta^i Z_0(u) = \Delta^{i+3} \mathcal Z_0(u)$ for all $i\ge 0$, the term $t u^3 \mathcal R$ on the new \rhs\ \emph{is} a polynomial function of $u$, $t$ and $\Delta^i \mathcal Z_0(u)$ ($i\ge 0$). Therefore \cite[Theorem 3]{BMJ06} applies.}
The same holds for $Z_1(u)$ and $Z(u,v)$, since according to \eqref{eq:Z0} and \eqref{eq:Zexplic}, they are rational functions of $Z_0(u)$ and of its coefficients.

\subsection{Connection with previous work and solution for $z_i$\!} \label{sec:z solution}

In principle, we could apply the general strategy developed in \cite{BMJ06} to eliminate the catalytic variable $u$ from \eqref{eq:1cat} and obtain an explicit algebraic equation relating $z_1$ (resp.~$z_3$) and $t$. However, in practice this gives an equation of exceedingly high degree. Instead, we need to exploit specific features of \eqref{eq:1cat} to eliminate $u$ while keeping the degree low. We will explain how this can be done in Appendix \ref{sec:2nd cat}. 
Here we forego the procedure of eliminating the catalytic variable $u$ and jump directly to the solution of $z_i(\nu,t)$ ($i=1,2,3$) by importing the corresponding results from \cite{BBM11}.

\newcommand{\Z}{\check{Z}}
\newcommand{\z}{\check{z}}

\begin{figure}[t]
\centering
\includegraphics[scale=1]{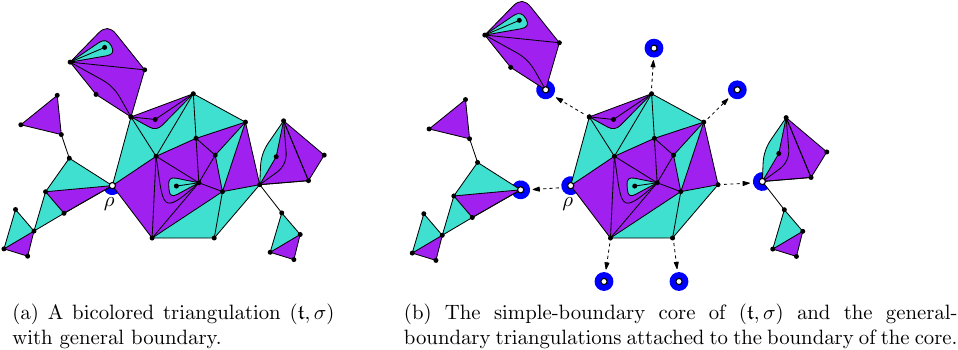}
\caption{By convention, the simple-boundary core of $\bt$ is the component following the root corner of $\bt$ (marked in blue) in the counter-clockwise direction. \emph{Pruning} consists of decomposing $\bt$ into this simple-boundary core, and one general-boundary component attached to each boundary vertex of the core. 
As shown in the example, these general-boundary components may be reduced to a single vertex, and this is taken into account by the constant term 1 in the generating series $\Z_0(u)$.
For visual clarity, the monochromatic boundary condition is omitted in the drawings.
}
\label{fig:pruning}
\end{figure}

In \cite{BBM11}, the quantity $2Q_i(2,\nu,t)$ is the generating series of vertex-bicolored triangulations with a general (i.e.~not necessarily simple) boundary of length $i$ and free boundary conditions (i.e.~the spins on the boundary vertices are not fixed). The parameter
$t$ counts the number of edges and $\nu$ the number of monochromatic edges (and the parameter $2$ represents the fact that the Ising model is equivalent to the 2-Potts model). To avoid confusion, we replace the symbols $\nu$ and $t$ of \cite{BBM11} by $\nu^*$ and $t^*$ in the following.

Let $\z_i(\nu,t)$ be the generating series of face-bicolored triangulations with a general boundary of length $i$ and monochromatic boundary condition.
By using the Kramers-Wannier duality between the low-temperature expansion and high-temperature expansion of the Ising partition function (see e.g.~\cite[Section 1.2]{BDC-duality}), one can show that if $(\nu^*,t^*)$ and $(\nu,t)$ satisfy
\begin{equation*}
	(\nu^*-1)(\nu-1)=2	\qtq{and}	2(t^*)^3 = (\nu-1)^3 t^2
\end{equation*}
then for all $i\ge 1$,
\begin{equation}	\label{eq:duality}
	(t^*)^iQ_i(2,\nu^*,t^*) = (\nu-1)^i t^i \z_i(\nu,t).
\end{equation} See page 41 in \cite{BBM11} for the details of the computation.
On the other hand, $\z_i(\nu,t)$ is nothing but the version of $z_i(\nu,t)$ where we remove the constraint of simple boundary. Let $\Z_0(u) \equiv \Z_0(\nu,t;u)= 1+ \sum_{i\ge 1} \z_i(\nu,t) u^i$. By decomposing a general boundary triangulation into its simple boundary core and general boundary triangulations attached to each boundary vertex of the core, one can show that $\Z_0(u) = Z_0(u \Z_0(u))$. This decomposition is known as \emph{pruning}. It is explained in Figure~\ref{fig:pruning}.

Extracting the first coefficients of $u$, we get
\begin{equation}\label{eq:simple-general}
	z_1 = \z_1			\qquad
	z_2 = \z_2 - z_1^2	\qtq{and}
	z_3 = \z_3 - 3z_1 \z_2 + 2 z_1^3\,.
\end{equation} Using \eqref{eq:duality} and \eqref{eq:simple-general}, we can easily translate the results in \cite[Thm.~23]{BBM11} to get the following rational parametrizations of $z_i(\nu,t)$ ($i=1,3$):
\begin{equation}\label{eq:RP z1-3}
\left\{\
\begin{aligned}
t^2\ &=\ \frac{ (\nu-S)\, (S+\nu-2) }{ 32 (\nu^2-1)^3 S^2 }\
				(4S^3 -S^2 -2S +\nu^2 -2\nu),
\\t^3z_1\ &=\ \frac{ (\nu-S)^2 (S+\nu-2) }{ 64 (\nu^2-1)^4 S^2 }
					(3S^3 -\nu S^2 -\nu S +\nu^2 -2\nu),
\\t^9z_3\ &=\ \frac{ (\nu-S)^5 (S+\nu-2)^5}{2^{22} (\nu^2-1)^{12} S^8}
		\cdot \big(\, 160S^{10} -128S^9 -16(2\nu^2-4\nu+3)S^8
\\& +\ 32(2\nu^2-4\nu+3)S^7 - 7(16\nu^2-32\nu+27)S^6 - 2(32\nu^2-64\nu+57)S^5
\\& +\ (32\nu^4-128\nu^3+183\nu^2-110\nu+20)S^4 - 4(7\nu^2-14\nu-2)S^3
\\& +\ \nu(\nu-2)(9\nu^2-18\nu-20)S^2 + 14\nu^2(\nu-2)^2 S - 3 \nu^3 (\nu-2)^3
		\,\big).
\end{aligned}
\right.
\end{equation}
These rational parametrizations will be checked in the appendix.
The singularity analysis of these series can also be imported from \cite[Claim 24]{BBM11}, which gives Proposition~\ref{prop:BBM}.
One can also give a proof to this theorem using the tools provided in Appendix~\ref{sec:RP}.

\subsection{Singularity analysis at the critical point}\label{sec:sing anal}

To get $Z_0(\nu,t;u)$, the generating function for Ising triangulations with a monochromatic boundary of arbitrary length, we plug the rational parametrization \eqref{eq:RP z1-3} into Equation \eqref{eq:1cat}. This gives us an equation of the form $\mathcal E(Z_0,u;\nu,S)=0$ where $\mathcal{E}$ is a polynomial of four variables. Under the change of variables $\tilde u=tu$ and $\tilde y=\frac tu Z_0(u)$, we obtain an equation of degree 5 in its main variables $\tilde u$ and $\tilde y$ (but of degree 21 overall, see \cite{CAS1}).

It is well known that a complex algebraic curve has a rational parametrization if and only if it has genus zero \cite{SWP08}. Both the genus of the curve and its rational parametrization, when exists, can be computed algorithmically, and these functions are implemented in the \textsc{algcurves} package of \textsc{Maple}. It turns out that the genus of the curve $\mathcal{E}(Z_0,u)=0$ is zero, thus a rational parametrization exists. However, the equation is too complicated for \textsc{Maple} to compute a rational parametrization in its full generality in reasonable time. The computation simplifies considerably in the critical case $(\nu,t)=(\nu_c,t_c)$, where $t_c$ corresponds to $S_c=3$ in \eqref{eq:RP z1-3}.
In this case, we found the following parametrization of $Z_0(u)$ and the corresponding parametrization of $Z_1(u)$ deduced from \eqref{eq:Z1}:
\begin{equation}\label{eq:RP H}
\left\{\
\begin{aligned}
u  =\hat u(H)	&\, :=\, \frac{u_c}3 H \mB({ 10 -12H +6H^2 -H^3 }	\\
Z_0=\hat Z_0(H)	&\, :=\, \frac1{10} \mB({ 1 -(1-\sqrt7)H +3H^2 -H^3 }
							 \mB({ 10 -12H +6H^2 -H^3 }			\\
Z_1=\hat Z_1(H)	&\, :=\, \frac3{10\,u_c} \mB({ \sqrt7-1+H-\frac{3(4-3H+H^2)}{10-12H+6H^2-H^3} }
\end{aligned}
\right.
\end{equation}
where $u=0$ is parametrized by $H=0$ and $u_c=\frac65(7+\sqrt{7})t_c$, as mentioned in Theorem~\ref{thm:z_p,q}.
By making the substitution $(u,Z_0(u),Z_1(u))\gets (\hat u(H),\hat Z_0(H),\hat Z_1(H))$ and $(v,Z_0(v),Z_1(v))\gets (\hat u(K),\hat Z_0(K),\hat Z_1(K))$ in \eqref{eq:Zexplic}, one obtains a rational parametrization of $Z(u,v)$ of the form%
\begin{equation}\label{eq:RP Z}
	u=\hat u(H),\quad v=\hat u(K) \qtq{and} Z=\hat Z(H,K)
\end{equation}
where $\hat Z(H,K)$ is a ratio of two symmetric polynomials of degree 10 and 4, respectively. Its expression is given in \cite{CAS1}.

Next, we would like to apply the standard transfer theorem of analytic combinatorics \cite[Corollary VI.1]{FS09} to extract asymptotics of the coefficients of $Z(u,v)$.
The idea is to use the rational parametrization to write that $Z(u,v)=\hat Z(\hat u^{-1}(u),\hat u^{-1}(v))$ in some neighborhood of the origin, and to extend this relation to the dominant singularity for one of the variables.
The main difficulty here is, given a rational parametrization of $v\mapsto Z(u,v)$, to localize rigorously its dominant singularity (or singularities), and to show that it has an analytic continuation on a $\Delta$-domain at this singularity. We will present a method that solves this problem in a generic setting in Appendix \ref{sec:RP}.
For the sake of continuity of exposition, we first summarize the properties of $Z(u,v)$ and $A(u)$ obtained with this method in the following lemma, and leave its proof to Appendix \ref{sec:RP}.

\newcommand*{\slit}[2][\epsilon]{D_{#2}^{|#1}}
\newcommand*{\Dc}{\overline D_{u_c}}
For $x>0$, let $D_x$ (resp.\ $\overline{D}_x$) be the open (resp.\ closed) disk of radius $x$ centered at 0. For $x,\epsilon>0$, the \emph{slit disk at $x$ of margin $\epsilon$} is defined as $\slit{x} = D_{x+\epsilon} \backslash [x,x+\epsilon]$. Notice that a slit disk at $x$ contains a $\Delta$-domain at $x$.

\begin{lem}\label{lem:dom sing Z}
\begin{enumerate}
\item $\sum_{p,q\ge 0} z_{p,q}u^pv^q$ is absolutely convergent if and only if $(u,v) \in (\overline{D}_{u_c})^2$.
\item There is a neighborhood $V$\! of $H=0$ such that $\hat u|_V$ is a conformal bijection onto a slit disk at $u_c$ and $\hat u(H)\to u_c$ as $H$ tends to 1 in $V$.
\item For each $u\in \overline{D}_{u_c}$, the function $v\mapsto Z(u,v)$ has its dominant singularity at $u_c$ and has an analytic continuation on a slit disk at $u_c$ (whose margin depends on $u$).
\item Similarly, the function $A(u)$ defined by the rational parametrization in Theorem~\ref{thm:z_p,q} has its dominant singularity at $u_c$ and has an analytic continuation on a slit disk at $u_c$.
\end{enumerate}
\end{lem}

Now let us carry out the singularity analysis of $Z(u,v)$ and finish the proof of Theorem~\ref{thm:z_p,q}. By Lemma~\refp{ii}{lem:dom sing Z}, the asymptotic expansion of $v\mapsto Z(u,v)$ at its dominant singularity $u_c$ is determined by the behavior of its parametrization in a neighborhood of $K=1$. One can check that the first and second derivatives of $K\mapsto\hat{Z}(H,K)$ \emph{both vanish} at $K=1$. Therefore the function has the Taylor expansion
\begin{equation*}
\hat Z(H,K)\ =\ \hat Z(H,1)	-\frac{\partial_K^3\hat Z(H,1)}6    (1-K)^3
							+\frac{\partial_K^4\hat Z(H,1)}{24} (1-K)^4 +O((1-K)^5)\,.
\end{equation*}
On the other hand, we can rewrite the equation $v=\hat u(K)$ as $(1-K)^3 = \frac32 \mb({1-\frac{v}{u_c}} - \frac12(1-K)^4$. In particular, we have $1-K \sim (\frac32)^{1/3} (1-\frac v{u_c})^{1/3}$ as $K\to 1$. Plugging this into the Taylor expansion of $K\mapsto\hat Z(H,K)$, we obtain the following asymptotic expansion of $v\mapsto Z(u,v)$ at $v=u_c$:
\begin{equation*}
Z(u,v) = Z(u,u_c) - \partial_v Z(u,u_c) (u_c-v) + A(u)\mB({ 1-\frac v{u_c} }^{4/3}
											+ O\mb({ \mB({ 1-\frac v{u_c} }^{5/3} }\,,
\end{equation*}
where $A(u)$ is given by the rational parametrization $u=\hat u(H)$ and
\begin{equation*}
A=\hat A(H)\ :=\	\mB({\frac32}^{4/3} \m({ \frac{\partial_K^4\hat Z(H,1)}{24}
									  + \frac{\partial_K^3\hat Z(H,1)}{12} }
	= \frac1{10}	\mB({\frac32}^{7/3} \frac{9-8H+3H^2}{(3-3H+H^2)^2}\,.
\end{equation*}

Thanks to Lemma~\refp{iii}{lem:dom sing Z}, the transfer theorem \cite[Corollary VI.1]{FS09} applies to $v\mapsto Z(u,v)$, which implies that for all $u\in\overline{D}_{u_c }$,
\begin{equation}\label{eq:asymp p}
	  Z_q(u) = [v^q]Z(u,v) \eqv{q}
		\frac{A(u)}{\Gamma(-4/3)} u_c^{-q} q^{-7/3}\,.
\end{equation}
This is the last asymptotic stated in Theorem~\ref{thm:z_p,q}.
It follows that
\begin{equation*}
	\frac{Z_q(u)}{Z_q(u_c)} \cv[]q \frac{A(u)}{A(u_c)}\,.
\end{equation*}
This can be interpreted as the pointwise convergence of the generating functions of \mbox{the discrete} probability distribution $\mb({ \frac{z_{p,q} u_c^p}{Z_q(u_c)} }_{p\ge 0}$ to the generating function of the sequence $\mb({ \frac{a_p u_c^p}{A(u_c)} }_{p\ge 0}$. According to a general continuity theorem \cite[Theorem IX.1]{FS09}, this implies the convergence of the sequences term by term:
\begin{equation*}
	\frac{z_{p,q}}{Z_q(u_c)} \cv[]q \frac{a_p}{A(u_c)}\,.
\end{equation*}
for all $p\ge 0$. (In fact \cite[Theorem IX.1]{FS09} also assumes the limit sequence to be a probability distribution \emph{a priori}, but a careful reading of the proof shows that this assumption is not necessary.) Comparing the last display with \eqref{eq:asymp p}, we obtain the asymptotics of $(z_{p,q})_{q\ge 0}$ stated in Theorem~\ref{thm:z_p,q}.

This asymptotics implies in particular that $a_p\ge 0$ for all $p\ge 0$. This positivity property is in fact used in the proof of Lemma~\ref{lem:dom sing Z}(iv) in Appendix~\ref{sec:RP}. But there is no vicious circle in the proof since we have used only the assertions (i)-(iii) of Lemma~\ref{lem:dom sing Z} to deduce the asymptotics of $(z_{p,q})_{q\ge 0}$. Now we repeat the same steps to find the asymptotics of $(a_p)_{p\ge 0}$. Contrary to $K\mapsto\hat Z(H,K)$, the first derivative of $H\mapsto \hat A(H)$ does not vanish at $H=1$. This leads to an exponent $1/3$ instead of $4/3$ for the leading order singularity of $A(u)$ at $u_c$:
\begin{equation*}
A(u) = A(u_c) + b\mB({ 1-\frac{u}{u_c} }^{1/3}
		+ O\mb({ \mB({ 1-\frac{u}{u_c} }^{2/3} }
\end{equation*}
where $b=-(\frac32)^{1/3}\hat A'(1) = -\frac{27}{20} (\frac32)^{2/3}$. We apply the transfer theorem again to obtain the asymptotics of $(a_p)_{p\ge 0}$. This completes the proof of Theorem~\ref{thm:z_p,q}.

\section{Limits of the perimeter processes}\label{sec:limit peeling}

Let us recall that the peeling process of a bicolored triangulation $\bt$ is an increasing sequence of explored maps $\nseq \emap$. It is determined by the sequence of peeling events $\nseq[1]\Step$ taking values in the countable set $\steps$, plus the initial condition $(p,q)$. We denote by $\Prob_{p,q}$ the law of $\nseq[1]\Step$ when $\bt$ is a Boltzmann Ising-triangulation of the $(p,q)$-gon.

As stated in Proposition~\ref{prop:comb cv}, the measure $\Prob_{p,q}$ converges weakly when $q\to\infty$ and then $p\to\infty$. In this section we first prove Proposition~\ref{prop:comb cv} and establish the basic properties of the limit distributions $\Prob\py$ and $\Prob\yy$. Then we move on to prove the scaling limits of the perimeter processes stated in Theorem~\ref{thm:scaling limit}. For convenience, we will denote by $\law_{p,q}X$ (resp.\ $\law\py X$ and $\law\yy X$) a random variable which has the same law as $X$ under $\Prob_{p,q}$ (resp.\ under $\Prob\py$ and $\Prob\yy$), where $X$ is a measurable function of the sequence $\nseq[1] \Step$ and of the initial condition $(p,q)$.

\subsection{Construction of $\Prob\py$ and $\Prob\yy$} \label{sec:limit S}

Since the terms of the sequence $\nseq[1]\Step$ live in a countable space, the weak convergence of $\Prob_{p,q}$ simply means the convergence of the probabilities of the form $\Prob_{p,q}(\Step_1=\step_1,\cdots,\Step_n=\step_n)$. In the proof below we will compute explicitly the limits of these probabilities, and verify that the resulting distribution is normalized.

\begin{lemma}[Convergence of the first peeling event]\label{lem:comb cv}
Assume $p\ge 0$. The limits
\begin{equation*}
\Prob\py(\Step_1=\step) := \lim_{q\to\infty} \Prob_{p,q}(\Step_1=\step) \qtq{and}
\Prob\yy(\Step_1=\step) := \lim_{p\to\infty} \Prob\py(\Step_1=\step)
\end{equation*}
exist for all $\step\in\steps$, and we have $\sum\limits_{\step\in \steps} \Prob\py(\Step_1=\step) = \sum\limits_{\step\in \steps} \Prob\yy(\Step_1=\step) = 1$.
\end{lemma}

\begin{table}[t]
\centering
\begin{tabular}{c r}
\begin{tabu}[t]{|L|L|L| |L|L|L| L}
\cline{1-6}
\step	& \Prob\py(\Step_1=\step)				& (X_1,Y_1)	&
\step	& \Prob\py(\Step_1=\step)				& (X_1,Y_1)	&
\hl	\cp & 		t_c \ap{p+2} u_c					& (2,-1)	&
	\cm & \frac{\nu_c t_c}{u_c}					& (0,1)	&
\hl	\lp & 		t_c \ap{p+1} z_{1,k} u_c^{k+1}	& (1,-k-1)	&
	\lm & \nu_c 	t_c z_{0,k+1} u_c^k				& (0,-k)	& (k\ge 0)
\hl	\rp & 		t_c z_{k+1,0}\ap{p-k+1} u_c		& (-k+1,-1)	&
	\rn & \nu_c 	t_c z_{k,1} \ap{p-k} 			& (-k,0)	& (0\le k\le p)
\hl	\rp[p+k] & 		t_c z_{p+1,k} \ap1 u_c^{k+1}	& (-p+1,-k-1)	&
	\rn[p+k] & \nu_c	t_c z_{p,k+1} \ap0 u_c^k		& (-p,-k)	& (k>0)
\hl
\end{tabu}
&\hspace{-1.2cm}\raisebox{-0.4cm}{(a)}\\
\begin{tabu}[t]{|L|L|L||L|L|L| L}
\cline{1-6}
\step  & \Prob\yy(\Step_1=\step)& (X_1,Y_1)	&
\step  & \Prob\yy(\Step_1=\step)& (X_1,Y_1)	&
\hl	\cp	& \frac{      t_c}{u_c}		& (2,-1)		&
	\cm	& \frac{\nu_c t_c}{u_c}		& (0,1) 	&
\hl	\lp	&       t_c u_c^k z_{1,k}	& (1,-k-1)		&
	\lm	& \nu_c t_c u_c^k z_{0,k+1}	& (0,-k)	&	(k\ge 0)
\hl	\rp	&       t_c u_c^k z_{k+1,0}	& (-k+1,-1)		&
	\rn	& \nu_c t_c u_c^k z_{k,1}	& (-k,0)	&	(k\ge 0)
\hl
\end{tabu}
\hspace{-0.3cm}
&\hspace{-1.2cm}\raisebox{-0.4cm}{(b)}
\end{tabular}
\caption{Law of the first peeling event $\Step_1$ under $\Prob\py$, $\Prob\yy$ and the corresponding $(X_1,Y_1)$.}\label{tab:prob(p)}
\end{table}

\begin{proof}
The existence of the limits can be easily checked using the expression of $\Prob_{p,q}(\Step_1=\step)$ in Table~\ref{tab:prob(p,q)} and the asymptotics of $z_{p,q}$ in Theorem~\ref{thm:z_p,q}. The explicit expressions of these limits are given in Table~\ref{tab:prob(p)}.

It is clear that $\sum_{\step\in \steps} \Prob\py(\Step_1=\step) =1$ for all $p\ge 0$ if and only if
\begin{equation*}
\sum_{p\ge0} a_p u^p \ =\
\sum_{p\ge0} \sum\limits_{\step\in \steps} a_p \Prob\py(\Step_1=\step) u^p
\end{equation*}
as formal power series in $u$. With a straightforward (but tedious) calculation using the data in Table~\refp{a}{tab:prob(p)}, one can show that the above condition is equivalent to 
\begin{equation}\label{eq:normalize(p)} 
\begin{aligned}
A(u) \ &=\
t_c u_c \mB({ \Delta^2_u A(u) + \mb({Z_1(u_c) + \Delta_u Z_0(u)}\Delta_u A(u) + a_1 \mb({\Delta_u Z(u,u_c) - \Delta_u Z_0(u)} }
\\&\ + \nu_c t_c \mB({ \frac{A(u)}{u_c} + \mB({\frac{Z_0(u_c)-1}{u_c} +Z_1(u)} A(u) + a_0 \mB({ \frac{Z(u,u_c) - Z_0(u)}{u_c} - Z_1(u) } }
\end{aligned}
\end{equation}
where $\Delta_u$ is the discrete derivative operator defined below \eqref{eq:2cat}. Recall that when $v\to u_c$, we have $Z(u,v) = Z(u,u_c) + \partial_v Z(u,u_c) (v-u_c) + A(u) (1-\frac{v}{u_c})^{4/3} + O\m({ (1-\frac{v}{u_c})^{5/3} }$. Then one can write down the expansion at $v=u_c$ of the second equation in \eqref{eq:2cat}, and verify that the coefficient of the dominant singular term $(1-\frac{v}{u_c})^{4/3}$ gives exactly \eqref{eq:normalize(p)}. This proves that $\sum_{\step\in \steps} \Prob\py(\Step_1=\step) =1$ for all $p\ge 0$.

Similarly, using the data in Table~\refp{b}{tab:prob(p)} one can show that $\sum_{\step\in \steps} \Prob\yy(\Step_1=\step) =1$ if and only if $(\nu_c+1)t_c \m({ \frac{Z_0(u_c)}{u_c} + Z_1(u_c) } = 1$. This equation can also be obtained as the coefficient of $(1-\frac{u}{u_c})^{1/3}$ in the expansion of \eqref{eq:normalize(p)} at $u=u_c$. This completes the proof of the lemma.
\end{proof}

\begin{proof}[Proof of Proposition~\ref{prop:comb cv}]
To have the convergence $\Prob_{p,q}\to\Prob\py$, we need to define
\begin{equation}\label{eq:lim S_n}
\Prob\py(\Step_1=\step_1,\cdots, \Step_n=\step_n)\ :=\ \lim_{q\to\infty}
\Prob_{p,q}(\Step_1=\step_1,\cdots, \Step_n=\step_n)
\end{equation}
for all $n\ge 1$ and all $\step_1,\cdots,\step_n\in \steps$.

As we have seen at the end of Section~\ref{sec:peeling}, the peeling events $(\step_k)_{1\le k\le n}$ completely determine the perimeter variations $(x_k,y_k)_{1\le k\le n}$. So according to the spatial Markov property, \eqref{eq:lim S_n} is equivalent to
\begin{align*}
\Prob\py(\Step_1=\step_1,\cdots, \Step_n=\step_n) :=&\
\lim_{q\to\infty} \Prob_{p,q}(\Step_1=\step_1) \Prob_{p+x_1,q+y_1}(\Step_1=\step_2) \cdots \Prob_{p+x_{n-1},q+y_{n-1}}(\Step_1=\step_n)	\\
=&\ \Prob\py(\Step_1=\step_1) \cdot \Prob\py[p+x_1](\Step_1=\step_2) \cdots \Prob\py[p+x_{n-1}](\Step_1=\step_n)\,.
\end{align*}
Then Lemma~\ref{lem:comb cv} implies that $\Prob\py$ is a probability distribution
on $\steps^{\integer_{\ge 0}}$. In particular, for any finite $n$, the probability under $\Prob\py$ of the peeling process stopping within time $n$ is zero. So the peeling process $\Prob\py$-almost surely never stops.

Similarly, we take the limit $p\to\infty$ in the above equations, and define $\Prob\yy$ by
\begin{equation*}
\Prob\yy(\Step_1=\step_1,\cdots, \Step_n=\step_n)\ :=\ \Prob\yy(\Step_1=\step_1) \cdots \Prob\yy(\Step_1=\step_n)\,.	\qedhere
\end{equation*}
\end{proof}

The above construction of $\Prob\py$ and $\Prob\yy$ implies immediately the following corollary.

\begin{corollary}[Markov property of $\Prob\py$ and $\Prob\yy$]\label{spatialmarkovp}
Under $\Prob\py$ and conditionally on $(\Step_k)_{1\le k\le n}$, the shifted sequence $(\Step_{n+k})_{k\ge 0}$ has the law $\Prob\py[P_n]$. In particular, $\nseq P$ is a Markov chain.\\
Under $\Prob\yy$, the sequence $\nseq \Step$ is i.i.d. In particular, $\nseq{X_n,Y}$ is a random walk.
\end{corollary}

\subsection{The random walk $\law\yy \nseq{X_n,Y}$}\label{sec:XY}
\newcommand{\tser}[2][t]{(#2_{#1})_{#1\ge 0}}
\newcommand{\nt}[1][t]{_{\floor{n#1}}}

The distribution of the first step $\law\yy (X_1,Y_1)$ of this random walk can be readily read from Table~\refp{b}{tab:prob(p)}.
From there it is not hard to compute explicitly its drift and tails, and deduce Theorem~\refp{1}{thm:scaling limit} by standard invariance principles.

\begin{proof}[Proof of Theorem~\refp{1}{thm:scaling limit}]
First, notice that the law of $\law\yy(X_1+Y_1)$ has a particularly simple expression given by
\begin{align*}
\Prob\yy(X_1+Y_1=1)	&\,=\,	\Prob\yy(\Step_1\in\{\cp,\cm\})
			\!\!\!\!		&&=\,	(\nu_c+1)t_c u_c^{-1}
\\\forall k\ge 0,\,
\Prob\yy(X_1+Y_1=-k)&\,=\,	\Prob\yy(\Step_1\in\{\lp,\lm,\rp,\rn\})
			\!\!\!\!		&&=\,	(\nu_c+1)t_c u_c^k (z_{k+1,0}+z_{k,1})	\,.
\end{align*}
It follows that
\begin{align*}
\EE\yy[X_1+Y_1]\ &=\ (\nu_c+1)t_c
\mB({ -\sum_{k=0}^\infty (k-1)u_c^{k-1}z_{k,0} - \sum_{k=0}^\infty k u_c^k z_{k,1} }
				\\\ &=\ (\nu_c+1)t_c
\mB({ \frac{Z_0(u_c)}{u_c} -Z_0'(u_c) - u_c Z_1'(u_c) }\ =\ \frac1{2\sqrt7}
\end{align*}
where the derivatives are computed using the chain rule $Z_0'(u_c)= \frac{\hat Z_0'(1)}{\hat u'(1)}$ and \eqref{eq:RP H}. Similarly, we deduce from Table~\refp{b}{tab:prob(p)} and \eqref{eq:RP H} the following expression and value of $\EE\yy[Y_1]$.
\begin{equation*}
\EE\yy[Y_1] = t_c\m({ (\nu_c-1) \frac{Z_0(u_c)}{u_c} -Z_1(u_c) -\nu_c Z_0'(u_c) - u_c Z_1'(u_c) } = \frac1{4\sqrt7}		\,.
\end{equation*}
We refer to the accompanying Mathematica notebook (\cite{CAS1}) for the computation of the numerical values above. It follows that $\EE\yy [X_1]=\EE\yy [Y_1]$. This is not obvious \emph{a priori}, since under $\Prob\yy$ the peeling process always chooses to reveal a triangle adjacent to a \+ boundary edge, breaking the symmetry between \+ and \<. 

Again from Table~\refp{b}{tab:prob(p)} we read that $\Prob\yy(X_1=-k) = \m({\nu_c z_{k+1,0}+\frac{z_{k-1,1}}{u_c} } t_c u_c^k$ and $\Prob\yy(Y_1=-k) = \m({\nu_c z_{k,1}+ u_c z_{k+2,0} } t_c u_c^k$ for all $k\ge 2$. By Theorem~\ref{thm:z_p,q}, their asymptotics is
\begin{align*}
&\Prob\yy(X_1= -k)\ \eqv{k}\ \frac{c_x}{k^{7/3}} &&\text{and}
&&\Prob\yy(Y_1= -k)\ \eqv{k}\ \frac{c_y}{k^{7/3}}
\\\text{where}\qquad
&c_x = \m({\nu_c \frac{a_0}{u_c} + a_1} \frac{t_c}{\Gamma(-4/3)}	&&\text{and}
&&c_y = \m({\frac{a_0}{u_c} + \nu_c a_1} \frac{t_c}{\Gamma(-4/3)}\,,
\end{align*}
or explicitly, $c_x = \frac1{8\,\Gamma(-4/3)} \frac{2+3\sqrt7}{7+\sqrt7} (\frac32)^{1/3}$ and $c_y = \frac1{8\,\Gamma(-4/3)} \frac{2+\sqrt7}{7+\sqrt7} (\frac32)^{1/3}$. Observe that $c_x>c_y$. It follows from a standard invariance principle (see e.g.\ \cite[Theorem VIII.3.57]{JS03}) that the two components of the random walk $\nseq{X_n,Y}$, after renormalization, converge respectively to the L\'evy processes $\mathcal{X}$ and $\mathcal{Y}$ in Theorem~\ref{thm:scaling limit}.

Now let us show that these two convergences hold jointly, and that the limits $\mathcal{X}$ and $\mathcal{Y}$ are independent. We adapt the proof of a similar result for the peeling of a UIPT \cite[Proposition 2]{CurKPZ}. Observe that the steps of the random walk satisfy $-2\le \max(X_1,Y_1)\le 2$, so that $\nseq X$ and $\nseq Y$ never jump simultaneously. Let us decompose $(X,Y)$ into the sum of two random walks $(X\00,Y\00)$ and $(X\01,Y\01)$ of respective step distributions
\begin{equation*}
(X\00_1,Y\00_1) = \idd{X_1 <  -2} (X_1, Y_1)    \qtq{and}
(X\01_1,Y\01_1) = \idd{X_1\ge -2} (X_1, Y_1)    \,.
\end{equation*}
According to the above observation, $(X\00,Y\00)$ only jumps along the $x$-axis, and $(X\01,Y\01)$ only jumps along the $y$-axis. (More precisely, $\abs{Y\00_k-Y\00_{k-1}} \le 2$ and $\abs{X\01_k-X\01_{k-1}} \le 2$ for all $k\ge 1$.)
Thus according to the same invariance principle as before, we have
\begin{equation*}
\m({\frac{X\00\nt - \EE\yy[X\00\nt]}{ n^{3/4}} ,
    \frac{Y\00\nt - \EE\yy[Y\00\nt]}{ n^{3/4}} }_{t\ge 0}
 \cv[]n (\mathcal{X}_t,0)_{t\ge0}
\end{equation*}
in distribution with respect to the Skorokhod topology. Here we have the joint convergence of the two components because the limit of the second component is a constant. Similarly, the random walk $(X\01,Y\01)$ converges to $\tser{0,\mathcal{Y}}$ after renormalization.

The random walks $(X\00,Y\00)$ and $(X\01,Y\01)$ are correlated. To recover independence, consider their poissonizations defined by $(\tilde{X}\0i_t,\tilde{Y}\0i_t) = (X\0i_{N_t}, Y\0i_{N_t})$, where $i\in\{0,1\}$ and $\tser N$ is an integer-valued Poisson point process of intensity 1. According to Lemma~\ref{lem:poisson} (stated and proved below) applied to $W_n = (X\00_n-\EE\yy [X\00_n],Y\00_n-\EE\yy [Y\00_n])$ and $a_n=n^{3/4}$, the poissonized random walks converge to the same limit after renormalization, namely
\begin{equation*}
\m({\frac{\tilde{X}\00_{nt} - \EE\yy[\tilde{X}\00_{nt}]}{ n^{3/4}} ,
    \frac{\tilde{Y}\00_{nt} - \EE\yy[\tilde{Y}\00_{nt}]}{ n^{3/4}} }_{t\ge 0}
 \cv[]n (\mathcal{X}_t,0)_{t\ge0} \,,
\end{equation*}
and similarly for $(\tilde{X}\01,\tilde{Y}\01)$. By the splitting property of compound Poisson processes, $(\tilde{X}\00,\tilde{Y}\00)$ and $(\tilde{X}\01,\tilde{Y}\01)$ are independent (See e.g.\ \cite[Proposition 6.7]{Nelson95}). Hence their sum $(\tilde X,\tilde Y)$, after renormalization, converges in distribution to the pair of independent L\'evy processes $(\mathcal{X,Y})$. Finally, we apply again Lemma~\ref{lem:poisson} to $W_n=(X_n -\EE\yy[X_n], Y_n -\EE\yy[Y_n])$ and $a_n=n^{3/4}$ to recover the convergence in Theorem~\refp{1}{thm:scaling limit}.
\end{proof}

\begin{lemma}[poissonization and depoissonization]\label{lem:poisson}
Let $\nseq W$ be a discrete-time random process in $\real^d$ ($d\ge 1$) and $\nseq a$ be a sequence of positive real numbers such that
\begin{equation*}
\frac1{n^{1/2}a_n} \sup_{t\in[0,T]} \Vert W\nt \Vert \cv[]n 0
\end{equation*}
in probability for all fixed $T>0$.
If $\tser N$ is a Poisson counting process of intensity 1 and independent of $\nseq W$, then we have
\begin{equation*}
d_{D_\infty}( a_n^{-1}W\nt, a_n^{-1}W_{N_{nt}} ) \cv[]n 0
\end{equation*}
in probability, where $d_{D_\infty}\!\!$ is the Skorokhod distance on the space of c\`adl\`ag functions on $[0,\infty)$.

In particular, if one of $(a_n^{-1}W\nt)_{n\ge 0}$ and $(a_n^{-1}W_{N_{nt}})_{n\ge 0}$ converge in distribution \wrt\ the Skorokhod topology, then the other also converges and has the same limit.
\end{lemma}

\begin{proof}
For $t\ge 0$, let $W\0n(t) = a_n^{-1}W\nt$ and $\tilde{W}\0n(t) = a_n^{-1} W_{N_{nt}}$. Recall from \cite{Billingsley} the definitions of $d_{D_\infty}$ and of $d_{D_m}$, the Skorokhod distance on the space of c\`adl\`ag functions on $[0,m]$.
By the definition of $d_{D_\infty}$, the conclusion of the lemma is equivalent to
\begin{equation*}
d_{D_m}(g_m W\0n,g_m \tilde{W}\0n) \cv[]n 0
\end{equation*}
in probability for all integers $m\ge 1$, where $g_m:[0,m]\to [0,1]$ is the regularization function defined by $g_m(t) = 1\wedge (m-t)$. From the definition of $d_{D_m}$, we see that the left hand side is bounded by
\begin{equation}\label{eq:bound D_m}
\sup_{t\le m} \abs{\lambda(t)-t} \,\vee\, \sup_{t\le m} \norm{(g_m W\0n)(\lambda(t)) - (g_m\tilde{W}\0n)(t)}
\end{equation}
where $\lambda$ is any increasing homeomorphism from $[0,m]$ onto itself.

Let $\lambda\0n$ be the increasing homeomorphism from $[0,\infty)$ onto itself defined by linearly interpolating the function $t\mapsto n^{-1} N_{nt}$. Then we have $W\0n(\lambda\0n(t)) = \tilde{W}\0n(t)$ for all $t\ge 0$.
For each $m$, we modify $\lambda\0n$ to produce a homeomorphism $\lambda\0n_m$ from $[0,m]$ onto itself as follows: let $t_m$ be the $x$-coordinate of the point where the graph of the function $\lambda\0n$ exits the square $[0,m-1/n]^2$. Define $\lambda\0n_m$ by $\lambda\0n_m(t) = \lambda\0n(t)$ for $t\in[0,t_m]$, and by linear interpolation for $t\in [t_m,m]$ so that $\lambda\0n_m(m)=m$.
Now consider \eqref{eq:bound D_m} when $\lambda=\lambda\0n_m$. Using the property of $\lambda\0n$ and the fact that $g_m$ is 1-Lipschitz, we can simplify the bound to get
\begin{equation*}
d_{D_m}(g_mW\0n, g_m\tilde{W}\0n)     \ \le\
\m({\frac1n + \sup_{t\le m} | \lambda\0n_m(t) -t |} \cdot
\m({ 1\vee \sup_{t\le m} \Vert W\0n(t) \Vert }
\end{equation*}
By central limit theorem, $\sqrt{n} \sup_{t\le m} | \lambda\0n_m(t) -t |$ converges in distribution to a finite random variable as $n\to\infty$. Thus the assumption of the lemma implies that the right hand side of the above inequality converges to zero in probability. This completes the proof.
\end{proof}

\subsection{The Markov chain $\law\py\nseq P$}\label{sec:MarkovP}

When $p$ is large, $\law\py \nseq P$ approximates the random walk $p+\law\yy \nseq X$, which has a strictly positive drift $\mu$. This seems to suggest that $\law\py \nseq P$ escapes to $+\infty$ with positive probability (indeed, as $P_n$ increases, the transition probabilities of $\law\py \nseq P$ gets closer to those of $p+\law\yy \nseq X$). However, as stated in Theorem~\refp{2}{thm:scaling limit}, $\law\py \nseq P$ hits zero with probability one. There is no contradiction because, despite the weak convergence $\Prob\py \to \Prob\yy$, the expectation $\EE\py{}[X_1]$ \emph{does not} converge to $\EE\yy[X_1]$ as $p\to\infty$. Actually, we will compute the limit of $\EE\py{}[X_1]$ in the remark after Proposition~\ref{prop:scaling} and see that it is negative.

What happens is that with high probability, the process $\law\py\nseq P$ stays close to the straight line $p_n=p+\mu n$ up to a time of order $\Theta(p)$, and then jumps to a neighborhood of zero in a single step.
The jump occurs because the peeling events of type $\RR^\jj_{p+k}$, for any fixed $k\in\integer$, occur with a probability of order $\Theta(p^{-1})$. (See Table~\refp{a}{tab:prob(p)}.)
To formalize this one-jump phenomenon, let us consider the stopping time
\begin{equation*}
T_m\ =\ \inf\Set{n\ge 0}{P_n\le m},
\end{equation*}
where $m\ge 0$ is some cut-off which will eventually be sent to $\infty$. In particular, $T_0$ is the first time that the boundary of the unexplored map becomes monochromatic.

The following lemma gives an upper bound for the tail distribution of $T_0$, which implies in particular that the process $\law\py \nseq P$ hits zero almost surely. It will also be used as an ingredient in the proof of Lemma~\ref{lem:one jump}.

\begin{lemma}[Tail of the law of $T_0$ under $\Prob\py$]\label{lem:hit 0}
There exists $\gamma_0>0$ such that $\Prob\py(T_0> \Lambda p)\le \Lambda^{-\gamma_0}$ for all $p\ge 1$ and $\Lambda>0$. In particular, $T_0$ is finite $\Prob\py$-almost surely.
\end{lemma}

\begin{proof}
From Table~\refp{a}{tab:prob(p)}, we read $\Prob\py(\Step_1= \rn[p]) = \nu_c t_c a_0\frac{z_{p,1}}{a_p}$ for all $p\ge1$. By Theorem~\ref{thm:z_p,q}, the right hand side decays like $p^{-1}$ when $p\to\infty$. Hence there exists $\delta>0$ such that
\begin{equation*}
\Prob\py(T_0= 1)  \ \ge\  \Prob\py(\Step_1=\rn[p])  \ \ge\  \delta /p
\end{equation*}
for all $p\ge 1$. On the other hand, $P_n$ increases at most by 2 at each step, therefore $P_n\le p+2n$ for all $n\ge 0$ almost surely under $\Prob\py$. It follows that for all $n\ge 0$,
\begin{equation*}
\Prob\py(T_0>n+1)	\ =  \ \EE\py \mb[{ \Prob\py[P_n](T_0\ne 1) \idd{T_0>n}  }
				\ \le\ \mb({ 1-\frac{\delta}{p+2n} } \Prob\py(T_0>n)\,.
\end{equation*}
By induction, we have $\Prob\py(T_0>n)\,\le\,\prod_{k=0}^{n-1} \m({ 1-\frac{\delta}{p+2k} }$ for all $n\ge 0$. Then, we use the inequality $\log(1-x)\le -x$ for $0<x<1$ and bound the Riemann sum by its integral:
\begin{equation*}
\Prob\py(T_0>\Lambda p)
\ \le\ \exp\mB({ -\sum_{k=0}^{\Lambda p-1} \frac{\delta}{p+2k} }
\ \le\ \exp\mB({ -\int_0^\Lambda \frac{\delta\ \dd x}{1+2x}}
\ =\ (1+2\Lambda)^{-\delta/2}.
\qedhere
\end{equation*}
\end{proof}

\newcommand{\tauxy}{\tau^\epsilon_x}
\newcommand{\barrier}[1][x]{#1 f_\epsilon}

Now let us quantify the statement that \emph{$\law\py\nseq P$ stays close to the line $p_n=p+\mu n$}. In fact, we will formulate the stronger result that \whp, both $\law\py \nseq X$ and $\law\py \nseq Y$ stay close to $x_n=\mu n$ up to time $T_m$.
Fix some arbitrary $\epsilon>0$. For $n\ge 0$, let
\begin{equation*}
\barrier[](n) = \mb({ (n+2)(\log(n+2))^{1+\epsilon} }^{3/4}.
\end{equation*}
Then, define the stopping time 
\begin{equation*}
\tauxy = \inf\Set{n\ge 0}{\abs{X_n-\mu n} \vee \abs{Y_n-\mu n} > \barrier(n) }\,.
\end{equation*}
where $x>0$.
If $\law\py \nseq{X_n,Y}$ were replaced by $\law\yy \nseq{X_n,Y}$, then we would have $T_m=\infty$ almost surely, and $\tauxy<\infty$ \whp\ in the limit $x\to\infty$ thanks to the law of iterated logarithm for heavy-tailed random walks \cite{Sch00}. 
The following lemma affirms that we can still use the function $\barrier(n)$ to bound the deviation of $\law\py \nseq{X_n,Y}$ up to time $T_m$ in the limit $p\to\infty$ and when both $x$ and $m$ are large.

\begin{lemma}[One jump to zero]\label{lem:one jump}
For all $\epsilon>0$,
\begin{equation*}
\lim_{x,m \to\infty} \limsupp \Prob\py (\tauxy<T_m) = 0
\end{equation*}
\end{lemma}

The proof of Lemma~\ref{lem:one jump} is based on technical estimates on the transition probabilities of the Markov chain $\law\py \nseq{P_n,Y}$ and is left to Appendix \ref{sec:lemma proof}.

Now let us complete the proof of Theorem~\refp{2}{thm:scaling limit}. We have seen that $\law\py \nseq P$ hits zero almost surely in finite time. It remains to show that its scaling limit is the process $(\mathcal D_t)_{t\ge 0}$ where
\begin{equation*}
\mathcal D_t = \begin{cases}
1+\mu t    &\text{if }t  < \zeta    \\
0          &\text{if }t\ge \zeta
\end{cases}
\qtq{and}
\prob(\zeta>t) = (1+\mu t)^{-4/3} \ .
\end{equation*}
Proposition~\ref{prop:scaling} below ensures that the time $T_m$ of the big jump has $\zeta$ as scaling limit when $p\to\infty$, regardless of the value of $m$. Therefore to prove Theorem~\refp{2}{thm:scaling limit} it suffices to show that the process $\law\py (p^{-1}P_{\floor{pt}})_{t\ge 0}$ converges to $1+\mu t$ before time $p^{-1}T_m$, and to zero after time $p^{-1}T_m$.

According to the definition of $\tauxy$, for all $n<\tauxy$, the distance between $p^{-1} P_n$ and $1+\mu n/p$ is bounded uniformly by $x f_\epsilon (\tauxy)/p$. 
Lemma~\ref{lem:one jump} and Proposition~\ref{prop:scaling} below together ensure that with high probability we have $\tauxy=T_m$ and $T_m$ is of order $p$. This implies that the distance between $p^{-1} P_n$ and $1+\mu n/p$ converges uniformly to zero on $n< T_m$ with probability arbitrarily close to 1 when $p\to\infty$ and for $x,m$ large enough.
On the other hand, since the Markov chain $\law\py \nseq P$ is recurrent (it hits zero almost surely), the rescaled process $\law\py[p_0] (p^{-1} P_{\floor{pt}})_{t\ge 0}$ converges to zero for any fixed initial condition $p_0$.
By the spatial Markov property, the shifted process $\law\py (P_{T_m+n})_{n\ge 0}$ has the same distribution as $\nseq P$ with \emph{some} random initial condition supported on $\{0,\dots,m\}$. Although this random initial condition depends on $p$, since it is always supported on the finite set $\{0,\dots,m\}$, the fact that $\law\py[p_0] (p^{-1} P_{\floor{pt}})_{t\ge 0} \to 0$ for every fixed $p_0 \in \{0,\dots,m\}$ implies that the rescaled process $\law\py (p^{-1} P_{T_m+ \floor{pt}})_{t\ge 0}$ converges identically to zero when $p\to\infty$.
This proves Theorem~\refp{2}{thm:scaling limit} provided that Proposition~\ref{prop:scaling} is true.

\begin{prop}\label{prop:scaling}
For all $m\in\natural$, the jump time $T_m$ has the same scaling limit as follows:
\begin{equation}\label{eq:Tm scaling}
\forall t>0\,,\qquad	\lim_{p\to\infty} \Prob\py\m({T_m>tp} = (1+\mu t)^{-4/3}.
\end{equation}
\end{prop}

\newcommand*{\anom}{\mathcal{E}}
\newcommand*{\nom}{\mathcal{N}}
\begin{proof}
First observe that $T_0\ge T_m$, so by strong Markov property,
\begin{equation*}
\Prob\py(T_0-T_m >n)	\ =	 \	\EE\py\m[{ \Prob\py[P_{T_m}](T_0 >n) }
					\ \le\	\max_{p'\le m}\Prob\py[p'](T_0 >n) \cv[]n 0 \,.
\end{equation*}
In particular, $\Prob\py(T_0-T_m> \epsilon p)\cv[]p 0$ for all $m\in\natural$ and $\epsilon>0$. This explains why the scaling limit of $p^{-1}T_m$ does not depend on $m$.

The rest of the proof is basically a refinement of the estimate of $\Prob\py(T_0>tp)$ given in Lemma~\ref{lem:hit 0}. The idea is that, before time $T_m$, the Markov chain $\nseq P$ stays close to the line $P_n=p+\mu n$. Therefore at time $n$ there is a probability roughly $\Prob_{p+\mu n}(P_1\le m)$ to jump below level $m$ at the next step. On the other hand, from Table~\refp{a}{tab:prob(p)} we can read the exact expression of $\Prob\py(P_1\le m)$ and show that for all $m\ge0$, there is a constant $c_m$ such that
\begin{equation}\label{eq:def c_m}
	\Prob\py(P_1\le m)  \ \eqv{p}\  c_m\, p^{-1}.
\end{equation}
(We leave the reader to check the computation leading to the above asymptotics, since a similar computation will be carried out in detail below \eqref{eq:c_infty as limit} for the value of $c_\infty := \lim_{m \to\infty} c_m$.)
With the above heuristics, the asymptotics \eqref{eq:def c_m} indicates that if the process $\nseq P$ has not jumped below level $m$ by the time $n$, then the probability that it jumps at time $n+1$ is roughly $\Prob_p(T_m=n+1|T_m>n) \approx \frac{c_m}{p+\mu n}$. Then \eqref{eq:Tm scaling} can be obtained by iterating this estimate over $n=0,\ldots,tp$ and taking the limit $m \to\infty$.

To make the above arguments rigorous, let us fix $x>0$, $m\in\natural$ and $\epsilon\in(0,\mu)$. Take $p$ large enough so that $\Prob\py$-almost surely, $\tauxy\le T_m$. Let $\anom = \{ \tauxy<T_m \}$ be the event of small probability in Lemma~\ref{lem:one jump}, on which the process $\nseq P$ deviates significantly from the line $p_n=p+\mu n$ before jumping close to zero ($\anom$ for ``exceptional''). Also let $\nom_n = \{ \tauxy>n  \}$ be the event that the trajectory of $\nseq P$ stays close to $p+\mu n$ up to time $n$ ($\nom$ for ``normal''). Obviously $\nseq \nom$ is a decreasing sequence. Moreover, one can check that
\begin{equation}\label{eq:anomaly inclusion}
\nom_{n+1}	\ \subset\ \nom_n\setminus\{T_m=n+1\}
				\ \subset\ \nom_{n+1} \cup \anom\,.
\end{equation}

On the event $\nom_n$, we have $P_0+\mu n -\barrier(n) \le P_n\le P_0+\mu n + \barrier(n)$. Combining this with the asymptotics \eqref{eq:def c_m}, we obtain that for $P_0=p$ large enough,
\newcommand*{\cmore}[1][p]{ \frac{c_m +\epsilon}{#1+\mu n - \barrier(n)} }
\newcommand*{\cless}[1][p]{ \frac{c_m -\epsilon}{#1+\mu n + \barrier(n)} }
\begin{equation*}
\cless[P_0] \id_{\nom_n} \ \le\ \id_{\nom_n} \Prob\py[P_n](P_1\le m) \ \le\ \cmore[P_0]	\id_{\nom_n} \,.
\end{equation*}
By Markov property, $\Prob\py(\nom_n\setminus\{T_m=n+1\})
= \Prob\py(\nom_n) - \EE\py\m[{ \id_{\nom_n} \Prob\py[P_n](P_1\le m) }$. Therefore
\begin{align*}
			\m({1-\cmore} \Prob\py(\nom_n) &\ \le\ \Prob\py(\nom_n\setminus\{T_m=n+1\})
\\&\ \le\	\m({1-\cless} \Prob\py(\nom_n)\,.
\end{align*}
Combining these estimates with the two inclusions in \eqref{eq:anomaly inclusion}, we obtain that on the one hand,
\begin{equation*}
	\Prob\py(\nom_{n+1})\ \le\ \m({1-\cless} \Prob\py(\nom_n) \,.
\end{equation*}
And on the other hand,
\begin{align*}
\Prob\py(\nom_{n+1}\cup\anom)
  &\ \ge\ \Prob\py\m({ (\nom_n\setminus\{T_m=n+1\}) \cup\anom }
\\&\ \ge\ \Prob\py(\nom_n\setminus\{T_m=n+1\}) + \Prob\py(\anom\setminus\nom_n)
\\&\ \ge\ \m({1-\cmore} \Prob\py(\nom_n) + \Prob\py(\anom\setminus \nom_n)
\\&\ \ge\ \m({1-\cmore} \Prob\py(\nom_n\cup \anom) \,.
\end{align*}
By induction on $n$, we get
\begin{equation*}
\Prob\py(\nom_N) \le \prod_{n=0}^{N-1} \m({1-\cless} \quad\text{and}\quad
\Prob\py(\nom_N\cup \anom) \ge \prod_{n=0}^{N-1} \m({1-\cmore} 
\end{equation*} for any $N\ge 1$.
Since $\nom_n\subset \{T_m>n\}\subset \nom_n \cup\anom$ up to a $\Prob\py$-negligible set, the above estimates imply that
\begin{equation*}
		\prod_{n=0}^{N-1} \m({1-\cmore} - \Prob\py(\anom) \ \le\ \Prob\py(T_m>N)
\ \le\	\prod_{k=0}^{N-1} \m({1-\cless} + \Prob\py(\anom)\,.
\end{equation*}
From the Taylor series of the logarithm we see that for all $x\ge 0$,  $-x-x^2\le \log(1-x)\le -x$. Therefore for any positive sequence $\nseq x$, we have
\begin{equation*}
\exp\mB({ -\sum_{n=0}^{N-1} x_n -\sum_{n=0}^{N-1} x_n^2}
\ \le\ \prod_{n=0}^{N-1}(1-x_n)\ \le\ \exp\mB({ -\sum_{n=0}^{N-1} x_n }\,.
\end{equation*}
On the other hand, in the limit $p\to\infty$ we have
$\frac{c_m\pm \epsilon}{p+\mu n\mp \barrier(n)} = \frac{c_m\pm \epsilon}{p+\mu n} (1+o(1))$ where $o(1)$ is uniform over all $n\in[0,tp]$, for any fixed $t>0$. It follows that
\begin{equation*}
\sum_{n=0}^{tp} \frac{c_m\pm \epsilon}{p+\mu n\mp \barrier(n)}
\ =\ (c_m\pm\epsilon) \int_0^{tp} \frac{\dd s}{p+\mu s} (1+o(1))
\ \cv[]p\ \frac{c_m\pm\epsilon}\mu \log(1+\mu t)\,.
\end{equation*}
We also have $\sum_{n=0}^{tp} (\frac{c_m+\epsilon}{p+\mu n- \barrier(n)})^2 \cv[]p 0$ for all $t>0$. Combining this with the last three displays, we conclude that
\begin{align*}
(1+\mu t)^{- \frac{c_m+\epsilon}\mu} - \limsupp \Prob\py(\anom)
&	\ \le\	\liminf_{p\to\infty} \Prob\py(T_m>tp)
\\&	\ \le\	\limsupp \Prob\py(T_m>tp)
	\ \le\	(1+\mu t)^{- \frac{c_m-\epsilon}\mu}
			+ \limsupp \Prob\py(\anom) \,.
\end{align*}
Now take the limit $m,x\to\infty$. In this limit, the error term $\limsup\Prob\py (\anom)$ tends to zero thanks to Lemma~\ref{lem:one jump}. The middle terms $\liminf\Prob\py(T_m>tp)$ and $\limsup\Prob\py(T_m >tp)$ do not depend on $m$ due to the convergence $\Prob\py(T_0-T_m>\epsilon p) \cv[]p 0$ seen at the beginning of the proof. Moreover, the increasing sequence $(c_m)_{m\ge0}$ has a limit $c_\infty$. Thus by sending $\epsilon\to 0$, we obtain
\begin{equation*}
\lim_{p\to\infty} \Prob\py(T_m>tp) = (1+\mu t)^{- \frac{c_\infty}\mu}\,.
\end{equation*}

Now it remains to show that in fact we have $c_\infty = \frac43\mu$. Using $c_m=\lim\limits_{p\to\infty}p\,\Prob\py(P_1\le m)$ and the data in Table~\refp{a}{tab:prob(p)}, $c_\infty$ can be written as
\newcommand{\sumk}[1][0]{\sum_{k=#1}^\infty}
\newcommand*{\limp}[1][p]{\lim_{#1\to\infty}}
\begin{align}
 c_\infty = \limp[m] c_m
&=	\limp[m] \limp p\m({ \Prob\py[p](P_1=0) + \sum_{k=1}^m \Prob\py[p](P_1=k) } \notag
\\&= \limp p \sumk \Prob\py[p]\m({ \Step_1 \in\{ \RR_{p+k}^\+, \RR_{p+k}^\< \} }
	+\sumk[1] \limp p\, \Prob\py[p]\m({ \Step_1 \in\{ \RR_{p-k}^\+, \RR_{p-k}^\< \} } \,.
\label{eq:c_infty as limit}
\end{align}
The probabilities can be read from Table~\refp{a}{tab:prob(p)}, which gives
\begin{align*}
		\Prob\py[p]\m({ \Step_1 \in\{ \RR_{p+k}^\+, \RR_{p+k}^\< \} }
\ &=\	 t_c z_{p+1,k} \ap1 u_c^{k+1} + \nu_c t_c z_{p,k+1} \ap0 u_c^{k}
\\\text{and}\qquad
		\Prob\py[p]\m({ \Step_1 \in\{ \RR_{p-k}^\+, \RR_{p-k}^\< \} }
\ &=\	 t_c z_{p-k+1,0}\ap{k+1}u_c	+\nu_c t_c z_{p-k,1} \ap{k} \,
\end{align*}
for all $k\ge 0$. 
Plug these expressions into \eqref{eq:c_infty as limit}, and we obtain
\begin{align*}
c_\infty = &
\limp \m({ t_c u_c a_1         \cdot p \frac{Z_{p+1}(u_c)   }{a_p} 
   + \nu_c t_c \frac{a_0}{u_c} \cdot p \frac{Z_p(u_c)-z_{p,0}}{a_p} }  \\
 & + \sumk[1] \limp \m({ t_c u_c a_{k+1} \cdot p \frac{z_{p-k+1,0}}{a_p} 
              + \nu_c t_c     a_k     \cdot p \frac{z_{p-k  ,1}}{a_p} } \,.
\end{align*}
The above limits can be evaluated using the asymptotics in Theorem~\ref{thm:z_p,q} and Equation \eqref{eq:asymp p}:
\begin{align*}
c_\infty = & \m({ 
      t_c u_c a_1 \cdot \frac{\Gamma(-1/3)}{\Gamma(-4/3)} \frac{A(u_c)}{b\, u_c}  +
\nu_c t_c \frac{a_0}{u_c} \cdot \frac{\Gamma(-1/3)}{\Gamma(-4/3)} \frac{A(u_c)-a_0}b }   \\ &+ \sumk[1] \m({ 
t_c u_c a_{k+1} \cdot \frac{\Gamma(-1/3)}{\Gamma(-4/3)} \frac{a_0 u_c^{k-1}}{b} +
\nu_c t_c a_k \cdot \frac{\Gamma(-1/3)}{\Gamma(-4/3)} \frac{a_1 u_c^k}{b}            }   
\\ =& -\frac43\frac{t_c}{b} \m({ \m({
    a_1 A(u_c) + \nu_c \frac{a_0}{u_c}(A(u_c)-a_0) 
} + \m({ 
    a_0 \frac{A(u_c)-a_0-a_1 u_c}{u_c} + \nu_c a_1 (A(u_c)-a_0)
} }.
\end{align*}
After simplification, we obtain
\begin{equation}
c_\infty = - \frac43 \frac{t_c}{b} (\nu_c+1) \m({ \frac{a_0}{u_c} +a_1 } (A(u_c)-a_0).
\end{equation}
The right hand side can be evaluated using the rational parametrization of $A(u)$ (see \cite{CAS1}), and we find indeed $c_\infty = \frac1{3\sqrt7}=\frac43\mu$.
\end{proof}

\begin{remark*}
(i) We remarked at the beginning of the section that the limit of $\EE\py{}[X_1]$ when $p\to\infty$ should be negative. One can actually compute this limit using the value of $c_\infty$ in the above proof, as follows: first, write $\EE\py{}[X_1]$ as the sum
\begin{equation*}
\EE\py{}[X_1\idd{X_1\ge -m}] + \EE\py{}[X_1 \idd{X_1\le -p+m}] + \EE\py{}[X_1\idd{X_1\in (-p+m,-m)}] \,.
\end{equation*}
The random variable in the first term is compactly supported, so the convergence in distribution $\Prob\py \to \Prob\yy$ implies that $\EE\py{}[X_1\idd{X_1\ge -m}] \cv[]p \EE\yy[X_1\idd{X_1\ge -m}] \cv[]m \EE\yy[X_1]$. In the second term, the value of $X_1$ is contained in $[-p,-p+m]$, while we have $\Prob\py \{X_1\le -p+m\} = \Prob\py \{P_1\le m\} \sim c_m p^{-1}$. It follows that $\EE\py{}[X_1 \idd{X_1\le -p+m}] \cv[]p -c_m \cv[]m -c_\infty$.
Using the exact distribution of $X_1$ in Table~\ref{tab:prob(p)}, it is not hard to bound the third term and show that it converges to zero as $p\to\infty$ and $m\to\infty$. Therefore $\lim \limits_{p\to\infty}\EE\py{}[X_1] = \EE\yy[X_1] - c_\infty = -\frac13\mu$.
\\(ii)
With our approach, it is quite amazing to find such a simple exponent $4/3$ for the scaling limit of the jump time $T_m$. Currently we do not have any rigorous explanation of this exponent apart from the computation above. 
Going one step back, one can see that the value $4/3$ relies on the algebraic identity
\begin{equation*}
\mu\ =\ \frac{(\nu_c+1)t_c}2 \m({ \frac{Z_0(u_c)}{u_c} - Z'_0(u_c)-u_c Z'_1(u_c) }
\ =\ -\frac{(\nu_c+1)t_c}b \m({ \frac{a_0}{u_c} +a_1 } (A(u_c)-a_0)\,,
\end{equation*}
together with the fact that $\EE\yy[X_1]=\EE\yy[Y_1]$. More importantly, we expect the same phenomenon to appear in any reasonable model of critical Ising-decorated maps, because the exponent $4/3$, which describes the believed scaling limit of an Ising-decorated map, ought to be universal (see also Section \ref{sec:interface} for a heuristic explanation via Liouville Quantum Gravity). In a work in progress, we have checked that this is indeed the case when we consider Boltzmann Ising-triangulations with spins on the vertices. It would be very interesting to have an algebraic or probabilistic explanation of this universality.
\end{remark*}

\section{Local convergence of Boltzmann Ising-triangulations}\label{sec:metric peeling}

\begin{figure}[t]
\centering
\includegraphics[scale=1]{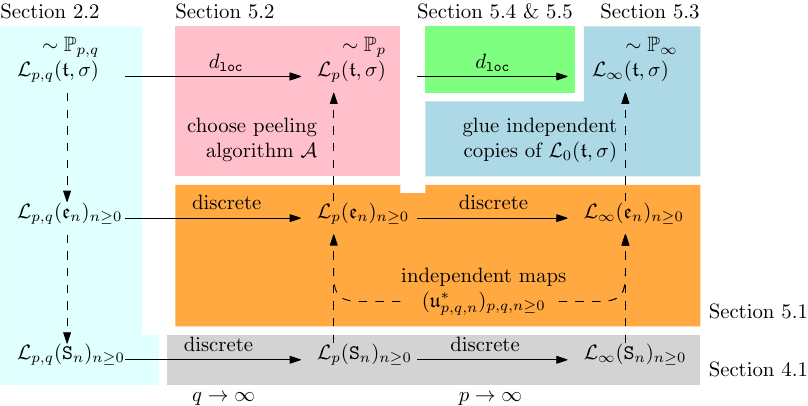}
\caption{Scheme of proof of the local weak convergence $\prob_{p,q} \cv[]q \prob\py \cv[]p \prob\yy$. A dashed arrow $A\dashrightarrow B$ indicates that the object $B$ is constructed from $A$. The label ``discrete'' over the solid arrows indicates that the convergences take place with respect to the discrete topology on the first $n$ terms of the sequences.}
\label{fig:proof-scheme}
\end{figure}

In this section we construct the local limit of the finite Boltzmann Ising-triangulations when $q\to\infty$ and $p\to\infty$. Both the construction and the proof of the convergence rely on the peeling process. 
More precisely, a finite Ising-triangulation can be encoded by its peeling process $\nseq \emap$, which in turn is encoded by its sequence of peeling events $\nseq \Step$ as described in Section~\ref{sec:peeling}. We have seen in Section~\ref{sec:limit S} that the distribution $\Prob_{p,q}$ of the peeling events $\nseq \Step$ converges towards the limits $\Prob\py$ and $\Prob\yy$. 

To recover the local limit of the original Ising-triangulations, we will try to invert the above encoding. Namely, we will try to recover the sequence of explored maps $\nseq \emap$ from the peeling events $\nseq \Step$, and then to recover the infinite Ising-triangulation $\bt$ from the sequence of finite maps $\nseq \emap$. The first step is straightforward and will be carried out in the next paragraph under both $\Prob\py$ and $\Prob\yy$. The second step is significantly more technical and requires different treatments under $\Prob\py$ and under $\Prob\yy$. This will be the subject of the rest of this section. We summarize the relations between the above objects in Figure~\ref{fig:proof-scheme}. Recall that we denote by $\law_{p,q} X$ (respectively by $\law\py X$ and $\law\yy X$) a random variable having the same distribution as $X$ under $\Prob_{p,q}$ (respectively under $\Prob\py$ and $\Prob\yy$). With a slight abuse, we extend this notation to random variables defined under $\prob_{p,q}$ and under the to-be-constructed measures $\prob\py$ and $\prob\yy$.

\subsection{Convergence of the peeling process}

\paragraph{Definition of $\law_p \nseq\emap$ and $\law_\infty \nseq\emap$.}
We will treat the two cases in a unified way by fixing some $p\in\natural \cup \{\infty\}$. To recover the sequence of explored maps $\nseq \emap$ from the peeling events $\nseq \Step$, one only needs to know the initial condition $\emap_0$ and the finite Ising-triangulations which are possibly swallowed at each step.

For $\emap_0$, consider $\integer$ with its usual nearest-neighbor graph structure and canonical embedding in the complex plane. We view it as an infinite planar map rooted at the corner at the vertex $0$ in the lower half plane. The upper-half plane is its unique internal face and is a hole. Then $\law\py \emap_0$ is defined as the deterministic map $\integer$ in which a boundary edge has spin \+ if it lies in the interval $[0,p]$ and spin \< otherwise.

\newcommand{\pqn}[1][n]{\tilde p,\tilde q,#1}
\newcommand{\upqn}[1][n]{\umap_{\pqn[#1]}^*}

Let $(\upqn)_{\pqn \ge 0}$ be a family of independent random variables which are also independent of $\nseq \Step$, such that $\upqn$ is a Boltzmann Ising-triangulation of the $(\tilde p,\tilde q)$-gon. Under $\prob_{p,q}$, one can recover the distribution of $\emap_n$ as a deterministic function of $\emap_{n-1}$, $\Step_n$ and $(\upqn)_{\tilde p,\tilde q \ge 0}$. For example, when $\Step_n = \rn$ with some $k \le P_{n-1}$, then one reveals a triangle in the configuration $\rn$ in the unexplored region of $\emap_{n-1}$, and uses $\umap_{k,1,n}^*$ to fill in the region swallowed by this new face. The result has the same law as $\emap_n$ under $\prob_{p,q}$.
We define $\law\py \nseq \emap$ by iterating the same deterministic function on $\law\py \emap_0$, $\law\py \nseq \Step$ and $(\upqn)_{\pqn \ge 0}$. 

Let $\filtr_n$ be the $\sigma$-algebra generated by $\emap_n$. Then the above construction defines a probability measure on $\filtr_\infty = \sigma(\cup_n \filtr_n)$, which we denote by $\Prob\py$ by a slight abuse of notation.

\paragraph{Convergence towards $\law_p \nseq\emap$ and $\law_\infty \nseq\emap$.}
Since $(\upqn)_{\pqn \ge 0}$ has a fixed distribution and is independent of $\nseq \Step$, Proposition~\ref{prop:comb cv} implies that $\law_{p,q} \nseq \Step$ and $(\upqn)_{\pqn \ge 0}$ converge jointly in distribution when $q\to\infty$ and $p\to\infty$. Here we are considering the convergence in distribution with respect to the discrete topology, namely, for any element $\omega$ in the (countable) state space of the sequences $\nseq \Step$ and $(\upqn)_{\pqn \ge 0}$ up to time $n_0<\infty$, we have $\Prob_{p,q}(\omega) \cv[]q \Prob\py(\omega) \cv[]p \Prob\yy(\omega)$.

\begin{figure}
\centering
\includegraphics[scale=1,page=2]{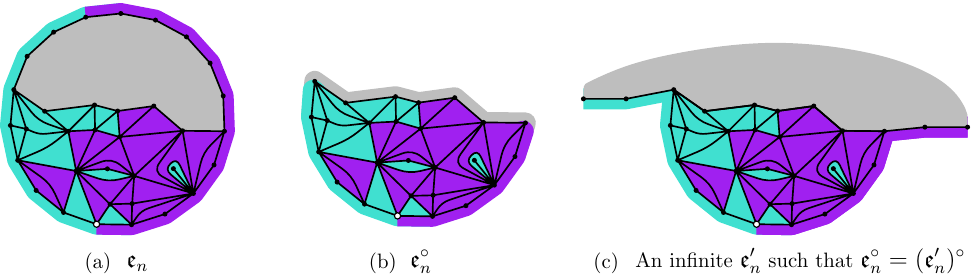}
\caption{Definition of the truncated map $\emapo_n$.}\label{fig:def-emapo}
\end{figure}

A caveat here is that the initial condition $\law_{p,q} \emap_0$ does not converge in the above sense, simply because $\law_{p,q} \emap_0$ is deterministic and takes a different value for each $(p,q)$. However, for any positive integer $K$, the restriction of $\law_{p,q} \emap_0$ (respectively, $\law\py \emap_0$) on the interval $[-K,K]$ does stabilize at the value that is equal to the restriction of $\law\py \emap_0$ (respectively, $\law\yy \emap_0$) on $[-K,K]$.
With this observation in mind, let us consider the truncated map $\emapo_n$, obtained by removing from $\emap_n$ all boundary edges adjacent to the hole, as in Figure~\ref{fig:def-emapo}. It is easily seen that the number of remaining boundary edges is finite and only depends on $(\Step_k)_{k\le n}$. It follows that for each $n$ fixed, $\emapo_n$ is a deterministic function of $(\Step_k)_{k\le n}$, $(\upqn[k])_{\tilde p,\tilde q\ge 0; k\le n}$ and $\emap_0$ restricted to some finite interval $[-K,K]$ where $K$ is determined by $(\Step_1,\dots, \Step_n)$. As the arguments of this function converge jointly in distribution with respect to the discrete topology (under which every function is continuous), the continuous mapping theorem implies that
\begin{equation}\label{eq:peeling cvg}
\Prob_{p,q}(\emapo_n=\bmap) \cv[]q \Prob\py(\emapo_n=\bmap) 
                            \cv[]p \Prob\yy(\emapo_n=\bmap)
\end{equation}
for all bicolored map $\bmap$ and for all integer $n\ge 0$. The following lemma says that one can replace $n$ in the above convergence by any finite stopping time.

\begin{lemma}[Convergence of the peeling process]\label{lem:stopped peeling}
Let $\filtr^\circ_n$ be the $\sigma$-algebra generated by $\emapo_n$.
If $\theta$ is an $\nseq{\filtr^\circ}$-stopping time that is finite $\Prob\py$-almost surely, then for all bicolored map $\bmap$,
\begin{equation}\label{eq:stopped peeling cvg}
\Prob_{p,q}(\emapo_\theta = \bmap) \cv[]q \Prob\py(\emapo_\theta = \bmap) \,.
\end{equation}
The same statement holds when $\Prob_{p,q}$ and $\Prob\py$ are replaced by $\Prob\py$ and $\Prob\yy$, respectively.
\end{lemma}

\begin{proof}
First assume that the map $\bmap$ is finite. Since the state of the explored region uniquely determines the past of the peeling process, for every fixed $\bmap$, there exists some finite $n = n(\bmap)$ such that $\{\emapo_\theta = \bmap\} = \{\emapo_n = \bmap\} \cap \{\theta = n\}$. Since $\theta$ is an $\nseq{\filtr^\circ}$-stopping time, the event $\{\theta=n\}$ is a measurable function of $\emapo_n$. Therefore $\{\emapo_n = \bmap\} \cap \{\theta = n\}$ is either empty or equal to $\{\emapo_n = \bmap\}$. Hence \eqref{eq:stopped peeling cvg} follows from \eqref{eq:peeling cvg}.

Obviously $\emapo_\theta$ is finite if and only if $\theta$ is. By Fatou's lemma, summing \eqref{eq:stopped peeling cvg} over the finite maps $\bmap$ gives $\liminf_{q\to\infty} \Prob_{p,q}(\theta<\infty) \ge \Prob\py(\theta<\infty) = 1$. It follows that 
\begin{equation*}
    \lim_{q\to\infty} \Prob_{p,q}(\theta=\infty) = 0
\end{equation*}
In particular, \eqref{eq:stopped peeling cvg} also holds when $\bmap$ is infinite (the right hand side is zero).

The same proof goes through when $\Prob_{p,q}$ and $\Prob\py$ are replaced by $\Prob\py$ and $\Prob\yy$ respectively.
\end{proof}

\begin{remark*}
Notice that we have not yet specified the peeling algorithm $\algo$, which chooses the initial vertex of the peeling in the case of a monochromatic $\<$ boundary. This means that the results up to this point are valid for any choice of $\algo$.
\end{remark*}

\subsection{Convergences towards $\prob\py$}\label{sec:def P(p)}

Although the convergences of peeling processes $\law_{p,q} \emapo_n \to \law\py \emapo_n$ and $\law\py \emapo_n \to \law\yy \emapo_n$ are proved exactly in the same way, the local convergence of the underlying random triangulation is much simpler in the first case, namely $\prob_{p,q} \to \prob\py$. 
As mentioned after Theorem~\ref{thm:cv}, this is thanks to the fact that, the peeling process $\nseq \emap$ eventually explores the entire triangulation almost surely under $\Prob\py$, provided one chooses an appropriate peeling algorithm. In this section we will specify one such algorithm $\algo$, use it to construct $\prob\py$, and then prove the local convergences $\prob_{p,q} \cv[]q \prob\py$ and $\prob_\pqq \cv[]{q_1,q_2} \prob\py[0]$ in Theorem~\ref{thm:cv}.

In the introduction we sketched the definition of the local distance on the set $\bts$ of bicolored triangulations of polygon. Now let us expand it in more details and in the general context of colored maps, so that the definition also applies to objects like the explored maps $\emap_n$, $\emapo_n$ or the balls in them.

\paragraph{Local limit and infinite colored maps.}

For a map $\map$ and $r\ge 0$, we denote by $[\map]_r$ the \emph{ball of radius $r$} in $\map$, defined as the subgraph of $\map$ consisting of all the \emph{internal} faces which are adjacent to at least one vertex within a graph distance $r-1$ from the origin. (The ball of radius 0 is the root vertex.) The ball $[\map]_r$ inherits the planar embedding and the root corner of $\map$. Thus $[\map]_r$ is also a map. By extension, if $\sigma$ is a coloring of \emph{some faces} and \emph{some edges} of $\map$, we define the ball of radius $r$ in $(\map,\sigma)$, denoted $[\map,\sigma]_r$, as the map $[\map]_r$ together with the restriction of $\sigma$ to the faces and edges in $[\map]_r$. In particular, we have $[[\map,\sigma]_{r'}]_r = [\map,\sigma]_r$ for all $r\le r'$.
Also, if an edge $e$ is in the ball of radius $r$ in a bicolored triangulation of polygon $\bt$, then one can tell whether $e$ is a boundary edge by looking at $\btsq_r$: only boundary edges are colored. See Figure~\ref{fig:def-ball} for an example.

\begin{figure}
\centering
\includegraphics[scale=1]{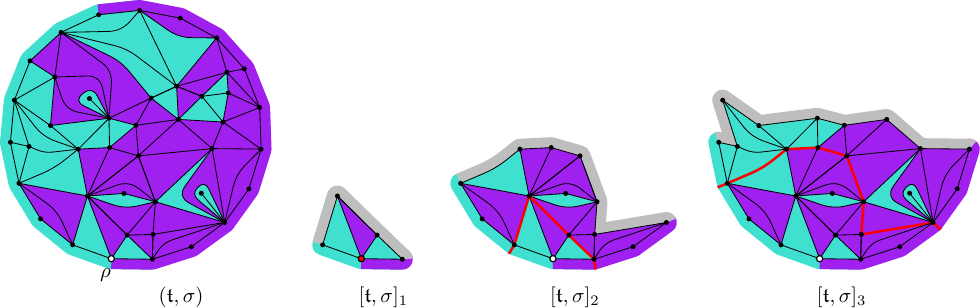}
\caption{The balls of radius 1, 2 and 3 in a bicolored triangulation of a polygon with a Dobrushin boundary condition.}
\label{fig:def-ball}
\end{figure}

\newcommand{\CM}{\mathcal{C\hspace{-1pt}M}}
The \emph{local distance} for colored maps is defined in a similar way as for uncolored maps: for colored maps $(\map,\sigma)$ and $(\map',\sigma')$, let
\begin{equation*}
d\1{loc}((\map,\sigma),(\map',\sigma')) = 2^{-R}\qtq{where}
	R = \sup\Set{r\geq 0}{ [\map,\sigma]_r=[\map',\sigma']_r }
\end{equation*}
The set $\CM$ of all (finite) colored maps is a metric space under $d\1{loc}$. Let $\overline \CM$ be its Cauchy completion. Similarly to the uncolored maps (see e.g.\ \cite{CurPeccot}), the space $(\overline \CM, d\1{loc})$ is Polish (i.e.\ complete and separable). The elements of $\overline \CM \setminus \CM$ are called \emph{infinite colored maps}. By the construction of the Cauchy completion, each element of $\CM$ can be identified as an increasing sequence of balls $(\bmap_r)_{r\ge 0}$ such that $[\bmap_{r'}]_r = \bmap_r$ for all $r\le r'$. Thus defining an infinite colored map amounts to defining such a sequence.
Moreover, if $(\prob\0n)_{n\ge 0}$ and $\prob\0\infty$ are probability measures on $\overline \CM$, then $\prob\0n$ converges weakly to $\prob\0\infty$ for $d\1{loc}$ if and only if
\begin{equation*}
\prob \0n([\map,\sigma]_r=\bmap) \ \cv[]n\ \prob \0\infty([\map,\sigma]_r=\bmap)
\end{equation*}
for all $r\ge 0$ and all balls $\bmap$ of radius $r$.

When restricted to the bicolored triangulations of the polygon $\bts$, the above definitions construct the corresponding set $\overline \bts \setminus \bts$ of infinite maps. Recall from Section~\ref{sec:intro} that $\bts_\infty$ is the set of \emph{infinite bicolored triangulation of the half plane}, that is, elements of $\overline \bts \setminus \bts$ which are one-ended and have an external face of infinite degree.

\paragraph{The covering time $\theta_r$ and the peeling algorithm $\algo$.}
Recall that the explored map $\emap_n$ contains an uncolored face with a simple boundary called its hole. The unexplored map $\umap_n$ fills the hole to give $\bt$. We denote by $\frontier_n$, called the \emph{frontier} at time $n$, the path of edges around the hole in $\emap_n$.

For all $r\ge 0$, let $\theta_r = \inf\Set{n\ge 0}{ d_{\emap_n}(\rho,\frontier_n)\ge r}$, where $d_{\emap_n}(\rho,\frontier_n)$ is the minimal graph distance in $\emap_n$ between $\rho$ and vertices on $\frontier_n$. It is clear that this minimum is always attained on the truncated map $\emapo_n$, therefore $d_{\emap_n}(\rho,\frontier_n)$ is $\filtr^\circ_n = \sigma(\emapo_n)$-measurable and $\theta_r$ is an $\nseq{\filtr^\circ}$-stopping time. Expressed in words, $\theta_r$ is the first time $n$ such that all vertices around the hole of $\emap_n$ are at a distance at least $r$ from $\rho$. Since $\bt$ is obtained from $\emap_n$ by filling in the hole, it follows that
\begin{equation*}
 \btsq_r\ =\ [\emapo_{\theta_r}]_r
\end{equation*}
for all $r\ge 0$. In particular, the peeling process $\nseq\emap$ eventually explores the entire triangulation $\bt$ if and only if $\theta_r<\infty$ for all $r\ge 0$.

Recall that in our context of peeling along the leftmost interface, the peeling algorithm is used to choose the origin $\rho_n$ of the unexplored map $\umap_n$ \emph{when its boundary $\frontier_n$ is monochromatic of spin \<}. (See Section~\ref{sec:peeling}.) Under $\Prob\py$, we can ensure $\theta_r<\infty$ almost surely for all $r\ge 0$ with the following choice of the peeling algorithm $\algo$: let $\rho_n = \algo(\emap_n)$ be the leftmost vertex on $\frontier_n$ that realizes the minimal distance $d_{\emap_n}(\rho,\frontier_n)$ from the origin.
The idea is that whenever $\frontier_n$ is monochromatic of spin $\<$, the peeling process tries to peel off the faces closest to the origin. But by Lemma~\ref{lem:hit 0}, the number of \+ edges on $\frontier_n$ drops to zero infinitely often $\Prob\py$-almost surely, so that every face will eventually be covered. More precisely:

\begin{lemma}\label{lem:cover r-ball}
$\theta_r$ is finite $\Prob\py$-almost surely for all $r\ge 0$ and $p\ge 0$.
\end{lemma}

\begin{proof}
The almost surely statements in this proof are with respect to $\Prob\py$. We have $\theta_0=0$. Assume that $\theta_r<\infty$ almost surely for some $r\ge 0$. Then the ball $[\tmap]_r$ is also almost surely finite. For $t\ge \theta_r$, let $v_t$ be the leftmost vertex in $[\tmap]_r \setminus [\tmap]_{r-1}$ that remains on the frontier $\frontier_t$ at time $t$. Then at every time $n\ge t$ such that $\frontier_n$ becomes monochromatic with spin \<, we have $\algo(\emap_n) = v_t$. By construction, the next peeling step peels the edge immediately on the left of $v_t$. Since $\Step_{n+1}$ has the law of $\law\py[0] \Step_1$, the vertex $v_t$ is swallowed at time $n+1$ with a fixed non-zero probability conditionally on $\filtr_n$.

By Lemma~\ref{lem:hit 0}, the frontier $\frontier_n$ becomes monochromatic of spin \< almost surely in finite time, and hence infinitely often by the spatial Markov property. Therefore the above construction implies that every vertex of $[\tmap]_r \setminus [\tmap]_{r-1}$ is swallowed by the peeling process almost surely in finite time. It follows that $\theta_{r+1}<\infty$ almost surely. 

By induction, $\theta_r$ is finite almost surely for all $r\ge 0$.
\end{proof}

\paragraph{Definition of $\prob\py$.} Lemma~\ref{lem:cover r-ball} implies that for every fixed $r$ the sequence $([\emap_n]_r)_{n\ge 0}$ stabilizes $\Prob\py$-almost surely for $n$ large enough. We define the infinite Boltzmann Ising-triangulation of law $\prob\py$ by its finite balls $\law\py \btsq_r := \lim\limits_{n \to\infty} \law\py{} [\emap_n]_r$. Since every finite subgraph of $\bt$ is covered by $\emap_n$ for $n$ large enough $\prob\py$-almost surely, its complement only has one infinite connected component, namely the one containing the unexplored map $\umap_n$. Therefore $\law\py \bt$ is almost surely one-ended. The external face of $\law\py \bt$ obviously has infinite degree. So it is indeed an infinite bicolored triangulation of the half plane.

\begin{proof}[Proof of the convergence $\prob_{p,q} \protect{\cv q} \prob\py$] 
The $(\filtr^\circ_n)$-stopping time $\theta_r$ is almost surely finite under $\Prob_{p,q}$ and $\Prob\py$, and $\btsq_r = [\emapo_{\theta_r}]_r$ is a measurable function of $\emapo_{\theta_r}$. Thus it follows from Lemma~\ref{lem:stopped peeling} that $\prob_{p,q}(\btsq_r = \bmap) \cv[]q \prob\py(\btsq_r = \bmap)$ for all $r\ge 0$ and every ball $\bmap$. This implies the local convergence $\prob_{p,q} \cv[]q \prob\py$.
\end{proof}

\begin{figure}
\centering
\includegraphics[scale=1,page=2]{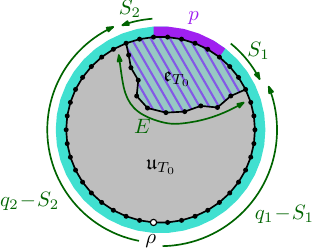}
\caption{Definition of $(E,S_1,S_2)$ in the proof of $\prob_\pqq\cv{q_1,q_2} \prob_0 $.}\label{fig:q1q2}
\end{figure}

\begin{proof}[Proof of the convergence $\prob_\pqq \protect{\cv{q_1,q_2}} \prob_0$] 
Recall that $\prob_\pqq$ is the law of $\law_{p,q_1+q_2}\bt$ after its origin is translated $q_1$ edges to the left along the boundary, see Figure~\ref{fig:q1q2}. 
Since the peeling process follows the leftmost interface, it is not affected by the translation of the origin up to $T_0$, the time when the leftmost interface is completely explored. It follows that $\law_\pqq \emap_{T_0}$ has the same law as $\law_{p,q_1+q_2}\emap_{T_0}$ up to the change of origin. So Lemma~\ref{lem:stopped peeling} implies that after removing the origin,
\begin{equation*}
\law_\pqq \emapo_{T_0} \cv[]{q_1,q_2} \law\py \emapo_{T_0}
\end{equation*}
in distribution with respect to the discrete topology. 

As in Figure~\ref{fig:q1q2}, let $E$ be the number of edges of $\frontier_{T_0}$ which are not on the boundary of $\bt$. Also, let $S_1$ (resp.\ $S_2$) be the number of \< boundary edges swallowed by $\emap_{T_0}$ on the right (resp.\ left) of the origin. It is clear that $(E,S_1,S_2)$ is a measurable function of $\emapo_{T_0}$ which does not depend on the position of the origin. Thus the above convergence in law of $\emapo_{T_0}$ implies that 
\begin{equation*}
\law_\pqq(E,S_1,S_2) \cv[]{q_1,q_2} \law_p(E,S_1,S_2)
\end{equation*}
in law. As shown in Figure~\ref{fig:q1q2}, the perimeter of $\umap_{T_0}$ satisfies $Q_{T_0}=E+ (q_1-S_1)+(q_2-S_2)$. So we have $\law_\pqq Q_{T_0} \to \infty$ in probability and thus $\prob_{0,Q_{T_0}} \to \prob\py[0]$ weakly as $q_1,q_2 \to \infty$. By the spatial Markov property, $\prob_{0,Q_{T_0}}$ is the law of $\umap_{T_0}$ conditionally on $\filtr_{T_0}$. It follows that
\begin{equation}\label{eq:cv shifted unexplored}
\law_\pqq\umap_{T_0}\cv{q_1,q_2} \law_0 \bt
\end{equation}
in distribution.

For a fixed $r\ge 0$, the ball $\btsq_r$ differs from $[\umap_{T_0}]_r$ only if the latter contains one of the edges counted by $E$. These edges are at a distance at least $\min(q_1-S_1,q_2-S_2)$ from the origin along the boundary of $\umap_{T_0}$. As $q_1,q_2 \to\infty$, this distance goes to $\infty$ in probability whereas $\law_\pqq [\umap_{T_0}]_r$ converges to $\law_0 \btsq_r$ in distribution. Thus the probability that $[\umap_{T_0}]_r$ differs from $\btsq_r$ converges to zero when $q_1,q_2 \to\infty$. Then it follows from \eqref{eq:cv shifted unexplored} that for all $r\ge 0$ and ball $\bmap$, 
\begin{equation*}
\prob_\pqq( \btsq_r=\bmap )\ \cv[]{q_1,q_2}\ 
\prob\py[0]( \btsq_r=\bmap )\,,
\end{equation*}
that is, $\law_\pqq\bt \cv{q_1,q_2} \law_0( \tmap,\sigma)$ in distribution.
\end{proof}

\subsection{Definition of $\prob\yy$}\label{sec:def P(infty)}

\newcommand{\rib}{\emapo_\infty}

Recall that $\theta_r$ is the first time $n$ that the explored map $\emap_n$ covers the ball of radius $r$ in $\bt$, so that $[\emapo_n]_r = \btsq_r$ for all $n\ge \theta_r$. By definition, it is a stopping time \wrt\ the filtration $\filtr^\circ_n=\sigma(\emapo_n)$ defined above Lemma~\ref{lem:stopped peeling}.
We have seen that, with an appropriate choice of the peeling algorithm, $\theta_r$ is finite $\Prob\py$-almost surely. This implied that
\begin{enumerate}
\item $\law\py{} [\rib]_r = \lim\limits_{n\to\infty} \law\py{} [\emapo_n]_r$ for all $r$ defines a bicolored triangulation $\rib$ of the half plane.
\item If $\prob\py$ is the law of the bicolored triangulation in (i), then $\prob_{p,q} \cv q \prob\py$ in distribution.
\end{enumerate}

In Section~\ref{sec:XY} we have seen that the perimeter processes $\nseq X$ and $\nseq Y$ drift to $+\infty$ almost surely under $\Prob\yy$. In particular they are bounded from below, that is, some vertices on the boundary of $\bt$ are never reached by the peeling process. Therefore, the analog of (ii) cannot be true for the limit $\prob\py \cv p \prob\yy$. However, we will show that the analog of (i) still holds. 
The resulting Ising-triangulation $\law\yy \rib$, called the \emph{ribbon} for reasons that shall be clear later, corresponds to the region in $\law\yy \bt$ that is eventually explored by the peeling process. It will be glued to other pieces of maps to construct $\law\yy \bt$.

\paragraph{Construction of the ribbon $\law\yy \rib$.}
To prove the analog of (i), one needs to check that the sequence $(\law\yy [\emapo_n]_r,\,n\ge 0)$ stabilizes in finite time for all $r\ge 0$, and that the resulting infinite bicolored triangulation $\law\yy \rib$ is one-ended, almost surely.

Let $r_0 = \sup\Set{r\ge 0}{ (\law\yy [\emapo_n]_r,\,n\ge 0)\text{ stabilizes in finite time}}$. If $r_0<\infty$, then there exists a vertex $v$ on the boundary of the ball $\lim_{n\to\infty} \law\yy [\emapo_n]_{r_0}$ such that the peeling process reveals infinitely many edges incident to $v$. By inspection of the possible peeling steps, one can see that when a new edge incident to $v$ is revealed, the distance between $\rho_n$ and $v$ along the frontier $\frontier_n$ is at most 2. (Recall that $\rho_n$ is the vertex where the \+ and \< parts of $\frontier_n$ meet.) This implies that, if the peeling process revealed infinitely many edges incident to $v$, then either $\nseq X$ or $\nseq Y$ would visit the same level infinitely many times. We know that this is not the case $\Prob\yy$-almost surely. Therefore $r_0=\infty$ almost surely, that is, $(\law\yy [\emapo_n]_r,\,n\ge 0)$ stabilizes in finite time for all $r\ge 0$, and $\law\yy \rib$ is well defined.

For each $n$, consider the graph $\rib \setminus \emapo_n$ consisting of all the edges and vertices incident to the triangles revealed after time $n$. Almost surely under $\Prob\yy$, the frontier $\frontier_n$ has both \+ and \< spins for all $n$. One can check that in this case the triangles revealed by two consecutive peeling steps always share an vertex, therefore $\rib \setminus \emapo_n$ is connected.
On the other hand, the argument in the previous paragraph shows that for every vertex $v$, no face incident to $v$ is revealed after some finite time. Hence if $V$ is the complement of some finite subset of vertices of $\rib$, then $V$ contains the vertices of $\rib \setminus \emapo_n$ for $n$ large enough. It follows that $V$ can have only one infinite connected component, namely the one containing $\rib \setminus \emapo_n$. This proves that $\law\yy \rib$ is almost surely one-ended.

\begin{figure}
\centering
\includegraphics[scale=1,page=1]{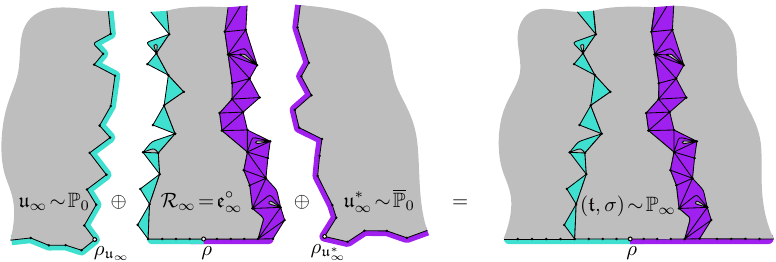}
\caption[caption]{The construction of $\prob\yy$.}\label{fig:ribbon-gluing}
\end{figure}

\newcommand*{\rmap}[1][m]{\mathcal{R}_{#1}}
\newcommand*{\uleft}[1][T_m]{\umap_{#1}}
\newcommand*{\uright}[1][T_m]{\umap^*_{#1}}
\newcommand{\kk}{\mathcal{K}_m}
\newcommand{\pjump}{\mathcal{P}}
\newcommand{\qjump}{\mathcal{Q}}
\newcommand{\pleft}{\mathcal{P}}
\newcommand{\qleft}{\mathcal{Q}}
\newcommand{\pright}{\mathcal{P}^*}
\newcommand{\qright}{\mathcal{Q}^*}

\paragraph{Definition of $\prob\yy$.}
The reasons for choosing the following notations will be clear in the next subsection.
Let $\rmap[\infty] = \rib$ be the ribbon under $\prob\yy$, and denote by $\overline \prob\py[0]$ the image of $\prob\py[0]$ by the inversion of spins. Let $\law\yy \uleft[\infty]$ and $\law\yy \uright[\infty]$ be two random variables of law $\prob\py[0]$ and $\overline \prob\py[0]$, respectively, such that $\law\yy \uleft[\infty]$, $\law\yy \uright[\infty]$ and $\law\yy \rmap[\infty]$ are mutually independent.

The boundary of $\law\yy \rmap[\infty]$ is partitioned into three intervals: one finite interval consisting of edges of $\emap_0$, and the two infinite intervals on its left and on its right. We glue $\law\yy \uleft[\infty]$ (resp.\ $\law\yy \uright[\infty]$) to the left (resp.\ right) interval as in Figure~\ref{fig:ribbon-gluing}. Since each piece is one-ended and the gluing between any two pieces occurs at infinitely many edges, the resulting bicolored triangulation is also one-ended. We call $\prob\yy$ its law. It is clear that $\law\yy \nseq \emap$ is indeed the peeling process of a random bicolored triangulation of law $\prob\yy$.

\subsection{Convergence of the ribbon}\label{sec:ribbonconvergence}
We have defined the Ising triangulation $\law\yy \bt$ as the disjoint union of the ribbon $\law\yy \rmap[\infty]$ and the two unexplored maps $\law\yy \uleft[\infty]$ and $\law\yy \uright[\infty]$. To prove the local convergence $\prob\py \to \prob\yy$, we would like to partition the Ising triangulation $\law\py \bt$ into three disjoint parts which converge locally to $\law\yy \rmap[\infty]$, $\law\yy \uleft[\infty]$ and $\law\yy \uright[\infty]$, respectively. After that, we can use the fact that gluing two locally converging maps of the half plane along their boundary results in a locally converging map (see Lemma~\ref{lem:gluing of loc cv}).

\begin{figure}[t]
\centering
\includegraphics[scale=1,page=2]{Fig12.pdf}
\caption[]{The ribbon $\rmap$ in the case (a) $\kk\ge 0$, or (b) $\kk \le 0$. 
The unexplored map $\uleft$ on the left of the ribbon (i.e.\ outside) is rooted at $\rho_\umap$ and has boundary condition $(\pleft,(\qleft_1,\qleft_2))$, with $\qleft_2=\infty$.  The unexplored map $\uright$ on the right of the ribbon (i.e.\ inside) is rooted at $\rho_{\umap^*}$ and has boundary condition $((\pright_1,\pright_2),\qright)$. We encode the spin of the triangle \includegraphics[scale=1]{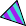} revealed at time $T_m$ by $\delta\in\{0,1\}$. The sign $\approx$ means equal up to a difference of 1 depending on $\delta$.
}
\label{fig:ribbon-at-large-jump}
\end{figure}

However, since the peeling process eventually explores the map $\law\py \bt$ entirely, there is no canonical way to define the ribbon $\rmap[\infty]$ under $\prob\py$. Instead, let us fix some arbitrary $m>0$ and define $\rmap$ to be the explored map $\emapo_{T_m-1}$ plus the triangle revealed at $T_m$. (Recall that $T_m$ is the first time $n\geq 0$ such that $P_n\le m$.) With this definition, $\rmap$, $\uleft$ and $\uright$ form a partition of the Ising-triangulation $\bt$ under $\prob\py$, where $\uright$ is the triangulation swallowed by the peeling step at $T_m$. 
Since we are interested in local limits with respect to the vertex $\rho$, we will reroot the unexplored map $\uleft$ close to $\rho$, more precisely at the vertex $\rho_\umap$ as shown in Figure~\ref{fig:ribbon-at-large-jump}. With the notation of Theorem~\ref{thm:cv}, the boundary condition of $\uleft$ is of the form $(\pleft,(\qleft_1,\qleft_2))$, with $\qleft_2=\infty$. Similarly, we root $\uright$ at the vertex $\rho_{\umap^*}$ as in Figure~\ref{fig:ribbon-at-large-jump} and denote its boundary condition by $((\pright_1,\pright_2),\qright)$. By inspection of the possible peeling events, one can confirm that $P_n$ may decrease only when $\Step_n$ is of type $\rp$ or $\rn$. Thus the condition 
\begin{equation*}
\Step_{T_m} \in \{ \rp[P_{(T_m-1)}+\kk], \rn[P_{(T_m-1)} +\kk] \}
\end{equation*}
uniquely defines an integer $\kk$. As shown in Figure~\ref{fig:ribbon-at-large-jump}, $\kk$ represents the position relative to $\rho^\dagger$ of the vertex where the triangle revealed at time $T_m$ touches the boundary. 

We want the triple $\law\py (\rmap, \uleft,\uright)$ to converge in distribution to $\law\yy(\rmap[\infty], \uleft[\infty], \uright[\infty])$ with respect to the local topology. However this cannot be true without a further amendment, because for any fixed $m$, there is always a non-vanishing probability that the large jump of the process $\nseq X$ occurs before $T_m$. (For example, we have $\tauxy = T_{m+1}<T_m$, i.e.\ the large jump arrive at $X = m+1$ instead of $X=m$, with some positive probability.)
Instead, we can only say that the convergence in distribution takes place on some event of large probability. This is formulated as follows.

\newcommand{\PrE}[4]{\prob#1 \mb({ ([#2]_r,[#3]_r,[#4]_r) \in \mathcal{E} }}

\begin{lemma}[Convergence of the ribbon]\label{lem:loc cv on big jump}
For fixed $\epsilon,x,m>0$, the triple $\law\py (\rmap, \uleft, \uright)$ converges locally in law to $\law\yy (\rmap[\infty], \uleft[\infty], \uright[\infty])$ on the event $\mathcal{J} \equiv \mathcal{J}^\epsilon_{x,m} := \{\tauxy = T_m \ge \epsilon p\} \cap \{ \kk \le m\}$, in the sense that for any $r\ge 0$,
\begin{equation}\label{eq:loc cv on big jump}
\begin{aligned}
& \limsupp \abs{ \PrE\py{\rmap}{\uleft}{\uright}
               - \PrE\yy{\rmap[\infty]}{\uleft[\infty]}{\uright[\infty]} }
\\   \le \ &
\limsupp \prob\py (\mathcal{J}^c) + \prob\yy(\tauxy<\infty)
\end{aligned}
\end{equation}
where $\mathcal{E}$ is any set of triples of balls.
\end{lemma}

\begin{remark*}
\begin{enumerate}
\item
Under $\prob\yy$, the integer $\kk$ is not well-defined, while $T_m=\infty$ almost surely. So the event $\{\tauxy < \infty\}$ on the right hand side of \eqref{eq:loc cv on big jump} is essentially $\mathcal{J}^c$ under $\prob\yy$.
\item
If $\mathcal{J}$ had probability one under both $\prob\py$ and $\prob\yy$, then the right hand side of \eqref{eq:loc cv on big jump} would vanish, and \eqref{eq:loc cv on big jump} would express exactly the local convergence in distribution $\law\py (\rmap, \uleft, \uright) \cv[]p \law\yy (\rmap[\infty], \uleft[\infty], \uright[\infty])$.
\item
If there exists a triple $(\tilde{\rmap[]}_m, \tilde{\umap}_{T_m},\tilde{\umap}_{T_m}^*)$ such that $\law\py (\tilde{\rmap[]}_m, \tilde{\umap}_{T_m},\tilde{\umap}_{T_m}^*) \cv[]p \law\yy (\rmap[\infty], \uleft[\infty], \uright[\infty])$ locally in distribution and that $(\tilde{\rmap[]}_m, \tilde{\umap}_{T_m},\tilde{\umap}_{T_m}^*) = (\rmap,\uleft,\uright)$ on the event $\mathcal{J}$, then 
\eqref{eq:loc cv on big jump} will follow. This is roughly how we will show \eqref{eq:loc cv on big jump} in the proof below.
\end{enumerate}
\end{remark*}

\begin{figure}[b!]
\centering
\includegraphics[scale=1]{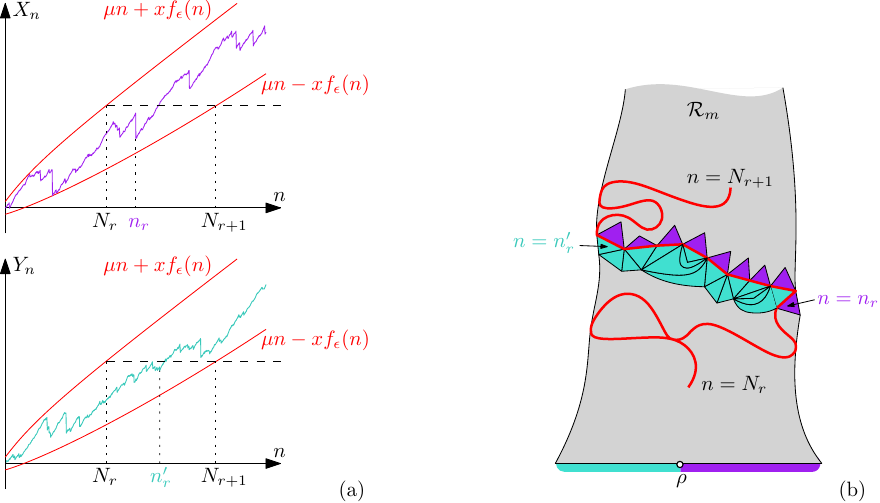}
\caption{%
(a) Before time $\tauxy$, the perimeter processes $\nseq X$ and $\nseq Y$ stay between the increasing barriers $\mu n\pm x f_\epsilon(n)$. Thus one can define a deterministic sequence of times $N_r$ up to $\tauxy$, such that $X_{N_r}$ is smaller than the minimum of $X_n$ on $[N_{r+1},\tauxy)$.
It follows that between $N_r$ and $N_{r+1}$, both $X_n$ and $Y_n$ must visit some level for the last time before $\tauxy$.
(b) Since the peeling process explores consecutively the triangles along the Ising interface, between the last visit times $n_r$ and $n_r'$, it must have explored a continuous layer of triangles spreading between the left and the right boundaries of $\emapo_{\tauxy-1}$. On the event $\{\tauxy=T_m\}$, one can replace $\emapo_{\tauxy-1}$ by $\rmap$.}
\label{fig:layer-times}
\end{figure}

\begin{proof}
As in the statement of Lemma~\ref{lem:loc cv on big jump}, we fix the numbers $\epsilon,x,m>0$ and drop them from the notation of quantities depending on them.
Recall that $\tauxy$ is the first time that either $X_n$ or $Y_n$ violates the bounds
\begin{equation*}
\mu n -x f_\epsilon(n) \ \le\  X_n,Y_n \ \le\  \mu n +x f_\epsilon(n) \,.
\end{equation*}

For some $N_0$ large enough, the left and the right hand side of the above inequality are strictly increasing in $n$ for $n\ge N_0$.
Let us define $(N_r)_{r\ge 0}$ inductively by
\begin{equation*}
N_{r+1} = \min \Set{n\ge 0}{ \mu n - xf_\epsilon(n) \ge \mu N_r + xf_\epsilon(N_r) \,}
\end{equation*}
for all $r\ge 0$. In other words, $N_{r+1}$ is the first time that the lower bound at time $N_{r+1}$ exceeds the upper bound at time $N_r$. Assume $N_{r+1} < \tauxy$. Then there exists a time $n_r \in (N_r,N_{r+1}]$ such that $X_{n_r} < \min_{n\in (n_r,\tauxy)} X_n$, that is, at time $n_r$ the process $\nseq X$ visits $(-\infty,X_{n_r}]$ for the last time before $\tauxy$. See Figure~\refp{a}{fig:layer-times}. Geometrically, this means that the triangle revealed at time $n_r$ stays on the \+ boundary of $\emapo_n$ up to time $\tauxy-1$. For the same reason, there is an $n_r'\in (N_r,N_{r+1}]$ such that the triangle revealed at time $n_r'$ stays on the \< boundary of $\emapo_n$ up to time $\tauxy-1$.

As shown in Figure~\refp{b}{fig:layer-times}, the above discussion implies that if $N_{r+1} < \tauxy$, then by the time $N_{r+1}$, the peeling process must have covered $\emapo_{N_r}$ by at least one layer of explored triangles spreading continuously from the \+ boundary to the \< boundary of $\emapo_{\tauxy-1}$. On the event $\mathcal{J}$, we have $\tauxy=T_m$, thus $\rmap$ is by definition equal to $\emapo_{\tauxy-1}$ plus one triangle. It follows that $\emapo_{N_{r+1}}$ contains all the vertices at distance 1 from $\emapo_{N_r}$ \emph{with respect to the graph distance inside $\rmap$}.
By induction, we have
\begin{equation*}
    [\rmap]_r \subseteq \emapo_{N_r}    \qtq{and thus}
    [\rmap]_r = [\emapo_{N_r}]_r
\end{equation*}
provided that $N_r < \tauxy$. Since $N_r$ is solely determined by $x$ and $\epsilon$, the previous condition is always satisfied on the event $\mathcal{J} \subseteq \{\tauxy \ge \epsilon p\}$, for any fixed $r$ and for $p$ large enough.

Next let us find a simple bound for the boundary conditions of $\uleft$ and $\uright$ on the event $\mathcal{J}$. The boundary condition $((\pright_1,\pright_2),\qright)$ of $\uright$ can be related to the perimeter processes by considering the following quantities (see Figure~\ref{fig:ribbon-at-large-jump}):
\begin{align*}
\pright_1 + \pright_2 + \qright \,=\ & P_{T_m-1} + \kk + 1  \,,   
    &&\qt{(total perimeter of $\uright$)}    \\
S^\+ + \pright_2 \,-\, &\min(0,\kk) \,=\, p                 \,,
    &&\qt{(number of edges between $\rho$ and $\rho^\dagger$)}    \\
\qright \,=\ &\max(0,\kk) + (1-\delta)                      \,.
    &&\qt{(number of \< edges on the boundary of $\uright$)}
\end{align*}
After rearranging the terms, we get
\begin{equation*}
\pright_1 = X_{T_m-1} + S^\+ + \delta  \ ,\ \ \ 
\pright_2 = p-S^\+ + \min(0,\kk)          \ \ \ \text{and}\ \ \
\qright = \max(0,\kk) + (1-\delta) \,.
\end{equation*}
On the event $\mathcal{J} = \{\tauxy = T_m \ge \epsilon p\} \cap \{ \kk \le m\}$ and for $p$ large enough, we have 
\begin{align*}
X_{T_m-1} &\ \ge\ \mu(T_m-1) -xf_\epsilon(T_m-1) 
           \ \ge\ \mu(\epsilon p-1) -xf_\epsilon(\epsilon p-1)    \\
S^\+      &\  = \ \delta' - \min_{n< T_m} X_n 
           \ \in\ \mb[{ 0\,,\, 1 - \min_{n\ge 0} (\mu n -xf_\epsilon(n))}  
\quad\qtq{and} \abs{\kk} \ \le\ m    \,.
\end{align*}
where $\delta'$ is either 0 or 1, depending on the peeling step that reveals the vertex $\rho_{\umap^*}$ (see Figure~\ref{fig:ribbon-at-large-jump}). It follows that
\begin{align*}
\pright_1 &\ \ge\ \mu (\epsilon p -1) - x f_\epsilon(\epsilon p -1) 
           \  =:\ \underline{\pright_1}         \\
\pright_2 &\ \ge\ p + \min_{n \ge 0} (\mu n -x f_\epsilon(n)) -1-m
           \  =:\ \underline{\pright_2}         \qtq{and}
\qright   \ \le\ m+1
\end{align*}
on $\mathcal{J}$. Notice that $\underline{\pright_1} \to\infty$ and $\underline{\pright_2} \to\infty$ when $p\to\infty$.
Similarly, one can show that 
\begin{equation*}
\qleft_1 \ \ge\ \underline{\qleft_1}    \qtq{and}
\pleft   \ \le\ m+1
\end{equation*}
on $\mathcal{J}$ for some deterministic number $\underline{\qleft_1} = \underline{ \qleft_1}(\epsilon,x,m,p)$ such that $\underline{\qleft_1} \cv[]p \infty$.

\newcommand{\tuleft}{\tilde{\umap}_{T_m}}
\newcommand{\turight}{\tilde{\umap}_{T_m}^*}

Consider two random bicolored triangulations $\tuleft$ and $\turight$ such that conditionally on $\rmap$, they are independent Ising-triangulations of respective boundary conditions $(\pleft \wedge (m+1)$, $(\qleft_1 \vee \underline{\qleft_1}, \infty))$ and $((\pright_1 \vee \underline{\pright_1}, \pright_2 \vee \underline{\pright_2}), \qright \wedge (m+1))$.
Thanks to the estimates in the previous paragraph, we have 
\begin{equation}\label{eq:proxy}
    ([\rmap]_r, \uleft,\uright) \ =\ ([\emapo_{N_r}]_r, \tuleft,\turight)
\end{equation}
on the event $\mathcal{J}$. (More precisely, there is a suitable coupling between the two sides such that the equality holds.)

According to \eqref{eq:peeling cvg}, we have $\law\py \emapo_{N_r} \cv[]p \law\yy \emapo_{N_r}$ with respect to the discrete topology. On the other hand, since $\qleft_1 \vee \underline{\qleft_1} \to \infty$ uniformly when $p\to\infty$, and $\pleft \wedge (m+1)$ takes values in the finite set $\{0,\dots,m+1\}$, the convergence $\prob_\pqq \cv{q_1,q_2} \prob_0$ implies that $\law\py \tuleft \cv[]p \law\py[0] \bt = \law\yy \uleft[\infty]$ locally in distribution. (Remark that the proof of $\prob_\pqq \to \prob_0$ in Section~\ref{sec:def P(p)} also works when $q_2=\infty$.) Similarly, $\law\py \turight \cv[]p \law\yy \uright[\infty]$ locally in distribution. These two convergences takes place conditionally on $\rmap$, and the limits do not depend on $\rmap$. It follows that we have the joint convergence
\begin{equation}\label{eq:proxy cv}
\law\py ([\emapo_{N_r}]_r, \tuleft,\turight) \cv[]p \law\yy ([\emapo_{N_r}]_r, \uleft[\infty], \uright[\infty])
\end{equation}
where three components on the right hand side are mutually independent, as prescribed by the definition of $\prob\yy$.

Equations \eqref{eq:proxy} and \eqref{eq:proxy cv} imply respectively
\begin{align}
&\ \ \prob\py \mb({ ([\emapo_{N_r}]_r ,\tuleft, \turight) 
              \ne ([\rmap]_r, \uleft,\uright)}   
\ \le\  \prob\py (\mathcal{J}^c)     \label{eq:proxy'}\\
\text{and}\ \ \  \lim_{p\to\infty} & \abs{ \,
      \prob\py \mb({ ([\emapo_{N_r}]_r, [\tuleft]_r, [\turight]_r) \in \mathcal{E} }
    - \prob\yy \mb({ ([\emapo_{N_r}]_r, [\uleft[\infty]]_r, [\uright[\infty]]_r)
                    \in \mathcal{E} }
\, }  \ =\ 0    \,.                  \label{eq:proxy cv'}
\end{align}
On the event $\{\tauxy=\infty\}$ we have $[\emapo_{N_r}]_r \subseteq \rmap[\infty]$ almost surely with respect to $\prob\yy$, so
\begin{equation}\label{eq:proxy infty}
       \prob\yy \mb({ [ \emapo_{N_r} ]_r \ne [ \rmap[\infty] ]_r } 
\ \le\ \prob\yy (\tauxy < \infty) \,.
\end{equation}
Then \eqref{eq:loc cv on big jump} follows from \eqref{eq:proxy'}, \eqref{eq:proxy cv'} and \eqref{eq:proxy infty} by the triangle inequality.
\end{proof}


\subsection{Convergence towards $\prob_\infty$}\label{sec:fullconvergence}

\newcommand{\pglued}[1]{\prob#1 \mb({ [\mop]_r \in \mathcal{E} }}
\newcommand{\mop}{\map \oplus \map'}

The triangulation $\law\py \bt$ (respectively, $\law\yy \bt$) can be seen as the result of gluing the triple $\law\py (\rmap,\uleft,\uright)$ (respectively, $\law\yy (\rmap[\infty],\uleft[\infty],\uright[\infty])$) along their boundaries. To deduce $\prob\py \to \prob\yy$ from Lemma~\ref{lem:loc cv on big jump}, one wants to show that local convergence is preserved by this gluing operation. First, let us look into the simpler setting of gluing two maps at their roots.
 
To keep familiar notations, let us consider probability measures $\prob\py$ $(p\ge 0)$ and $\prob\yy$ on some probability space $\Omega$. Let $\map$ and $\map'$ be two colored, possibly infinite random maps defined on $\Omega$. Assume that $\map$ and $\map'$ always have simple boundaries, and that $\law\yy \map$ and $\law\yy \map'$ are almost surely maps of the half plane. (As before, $\law\yy \map$ is a random variable having the law of $\map$ under $\prob\yy$.) Denote by $\rho$ and $\rho'$ the root vertices of $\map$ and $\map'$. Let $L$ be a random variable on $\Omega$ taking positive integer or infinite values, such that
\begin{equation}\label{eq:gluing length}
\law\py L \cv[]p \infty \text{ in distribution and }\law\yy L = \infty \text{ almost surely.}
\end{equation}

Finally, let $\mop$ be the map obtained by gluing the $L$ boundary edges of $\map$ on the right of $\rho$ to the $L$ boundary edges of $\map'$ on the left of $\rho'$. The dependence on $L$ is omitted from this notation because the local limit of $\mop$ is not affected by the precise value of $L$, provided that \eqref{eq:gluing length} is true. The following lemma affirms this claim, and relates the local convergence of $\mop$ to the local convergence of $\map$ and $\map'$.

\begin{lemma}[Gluing of locally convergent maps]\label{lem:gluing of loc cv}
Let $\varepsilon \ge 0$. If for all $r\ge 0$ and all sets $\mathcal{E}$,
\begin{equation}\label{eq:cv not glued}
\limsupp \abs{ \,
  \prob\py \mb({ ([\map]_r,[\map']_r) \in \mathcal{E} }
- \prob\yy \mb({ ([\map]_r,[\map']_r) \in \mathcal{E} } \,} 
\ \le\ \varepsilon     \,,
\end{equation}
then $\mop$ satisfies the same inequality, that is, for all $r\ge 0$ and all sets $\mathcal{E}$,
\begin{equation}\label{eq:cv glued}
\limsupp \abs{\, \pglued\py - \pglued\yy \,}  \ \le\ \varepsilon  \,.
\end{equation}
\end{lemma}

\begin{remark*}
When $\varepsilon = 0$, the lemma says that if the pair $\law\py (\map,\map')$ converges jointly in law to $\law\yy (\map,\map')$ with respect to the local topology, then so does their gluing $\mop$.

For $\epsilon>0$, one should interpret \eqref{eq:cv glued} as saying that $\law\py (\mop)$ converges locally in distribution to $\law\yy (\mop)$ on some event of probability at least $1-\epsilon$. Similarly for \eqref{eq:cv not glued}.
\end{remark*}

\begin{figure}[t!]
\centering
\includegraphics[scale=0.71]{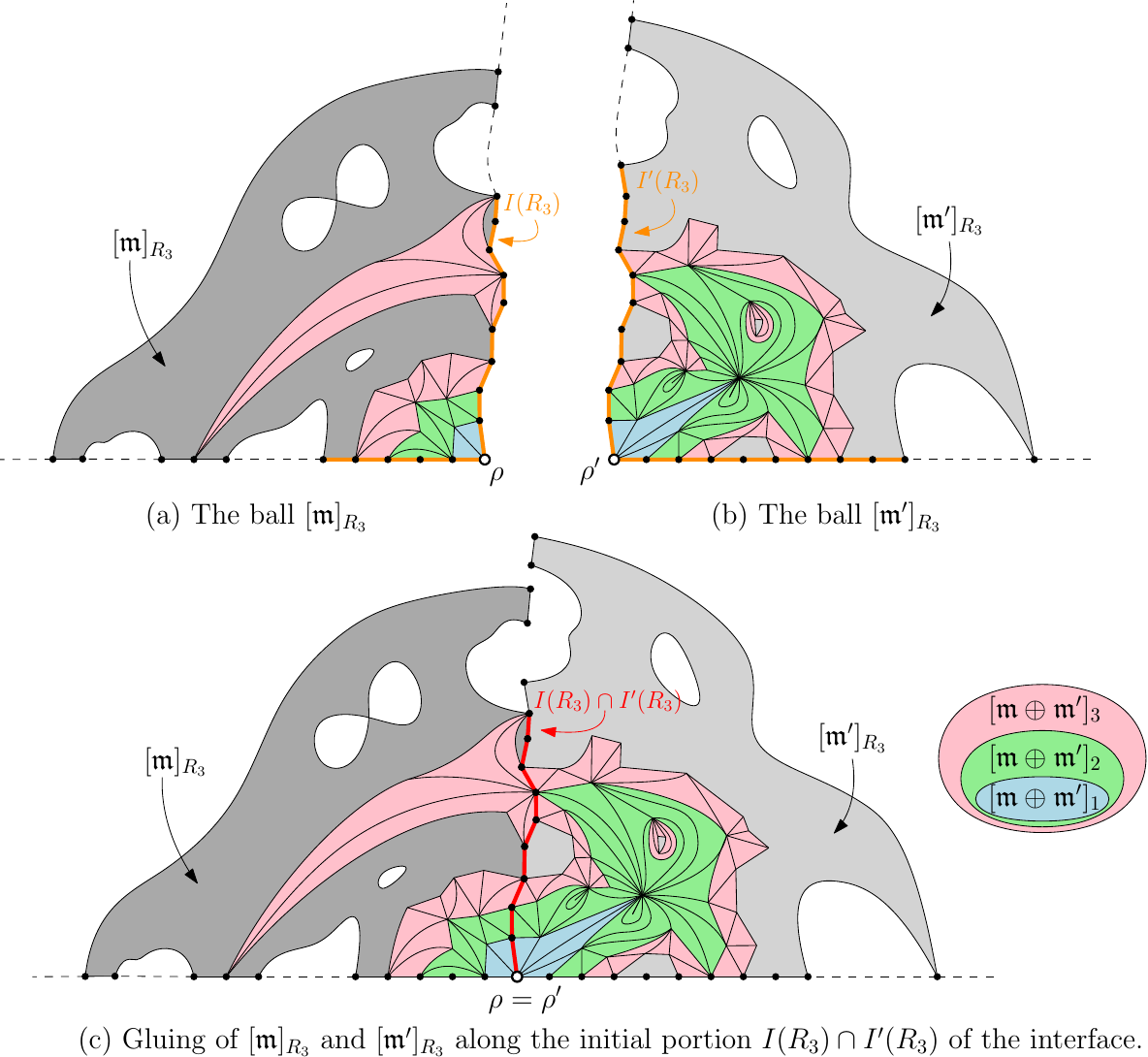}
\caption{(a--b) The grey regions represent the balls of radius $R_3$ in $\map$ and $\map'$, respectively. The triangles belonging to the ball of radius 3 in the glued map $\map\oplus\map'$ are highlighted. By definition, $I(R_3)$ is the maximal interval of boundary vertices of $\map$ containing the \mbox{root $\rho$}.
(c) The knowledge of $[\map]_{R_3}$ and $[\map']_{R_3}$ suffices to determine how they are glued together along the initial portion $I(R_3) \cap I'(R_3)$ of the gluing interface between $\map$ and $\map'$. Therefore, if the ball $[\map \oplus \map']_3$ intersects the gluing interface only inside $I(R_3) \cap I'(R_3)$, then $[\map]_{R_3}$ and $[\map']_{R_3}$ determine the ball $[\map \oplus \map']_3$, that is, $[\map \oplus \map']_3$ is $\filtr[G]_{R_3}$-measurable.}
\label{fig:gluing-of-2}
\end{figure}

\begin{proof}
For all $r\ge 0$, let $\filtr[G]_r$ be the $\sigma$-algebra generated by $([\map]_r, [\map']_r)$. The assumption \eqref{eq:cv not glued} says that 
$\limsupp \abs{\prob\py(E)-\prob\yy(E)} \le \varepsilon $
for every event $E\in \filtr[G]_r$. 

For $R\ge 0$, the ball $[\mop]_r$ is $\filtr[G]_R$-measurable (that is, for any set of balls $\mathcal{E}$, the event $\{[\mop]_r \in \mathcal{E}\}$ is $\filtr[G]_R$-measurable) if one can reconstruct $[\mop]_r$ from $[\map]_R$ and $[\map']_R$.
This certainly implies that $[\mop]_r$ is contained in the union $[\map]_R \cup [\map']_R$. However this is not enough, because to determine how a boundary vertex $v$ of $[\map]_R$ is glued to $[\map']_R$, one has to know how many vertices there are between $v$ and $\rho$ along the boundary of $\map$. Even if $v$ is in $[\map]_R$, all the vertices between $v$ and $\rho$ are not necessarily in $[\map]_R$. Let $I(R)$ be the maximal interval of \emph{consecutive} boundary vertices of $\map$ that are in $[\map]_R$, and define $I'(R)$ similarly for $\map'$. 
Then a sufficient condition for $[\mop]_r$ to be $\filtr[G]_R$-measurable is that the ball $[\mop]_r$ is contained in $[\map]_R \cup [\map']_R$ \emph{and} its intersection with the gluing interface is contained in $I(R)\cap I'(R)$, since in this case one can reconstruct $[\mop]_r$ from $[\map]_R$ and $[\map']_R$ by gluing them along a subinterval of $I(R)\cap I'(R)$. See Figure~\ref{fig:gluing-of-2}. For a given $r\ge 0$, let $R_r$ be the minimal radius $R\ge 0$ such that the above sufficient condition is satisfied. By observing the balls $[\map]_R$ and $[\map']_R$, one can determine whether $R_r \le R$, therefore the event $\{R_r \le R\}$ is in $\filtr[G]_R$. Similarly, for any set of balls $\mathcal{E}$, the intersection $\{[\mop]_r \in \mathcal{E}\} \cap \{R_r \le R\}$ is also in $\filtr[G]_R$. In other words, $R_r$ is a $\filtr[G]$-stopping time, and the event $\{[\mop]_r \in \mathcal{E}\}$ is in $\filtr[G]_{R_r}$.
Also, $R_r<\infty$ almost surely because the union of the balls $[\map]_R$ and $[\map']_R$ eventually covers the whole map $\mop$ when $R\to\infty$.

\newcommand{\Ei}{\underline E_R}
\newcommand{\Es}{\overline E_R}

For any $R \ge 0$, let $\Ei := \{ [\mop]_r \in \mathcal{E} \text{ and } R_r \le R \}$ and $\Es := \Ei \cup \{R_r>R\}$. Since $\Ei \subseteq \{ [\mop]_r \in \mathcal{E} \} \subseteq \Es$, we have
\begin{align*}
  \pglued\py - \pglued\yy 
\ \le&\ \prob\py(\Es) - \mb({\prob\yy(\Es) - \prob\yy(R_r>R)}  \\ \text{and~~~}
  \pglued\yy - \pglued\py 
\ \le&\ \mb({\prob\yy(\Ei) + \prob\yy(R_r>R)} - \prob\py(\Ei)    \,.
\end{align*}
Since $\Ei$ and $\Es$ are $\filtr[G]_R$-measurable, the assumption of the lemma yields
\begin{align*}
&\limsupp \abs{\, \pglued\py - \pglued\yy \,}   \\
\le\ & \max\mB({\limsupp \abs{\,\prob\py(\Es) - \prob\yy(\Es)\,},
                \limsupp \mb|{\,\prob\py(\Ei) - \prob\yy(\Ei)\,}}
                   + \prob\yy (R_r>R)    \\
\le\ & \varepsilon + \prob\yy (R_r>R) \,.
\end{align*}
The right hand side tends to $\epsilon$ when $R\to\infty$ since $R_r<\infty$ almost surely. This gives \eqref{eq:cv glued}.

In the above reasoning we have ignored the possibility that $L$, the total number of glued edges in $\map \oplus \map'$, may be smaller than $R$. Taking this into account adds an extra error term of $\limsup_{p\to\infty} \prob\py(L\le R) + \prob\yy(L\le R)$ to the right hand side of the last display. But this term is zero if $L$ satisfies the assumption \eqref{eq:gluing length}. This completes the proof of the lemma.
\end{proof}

The Ising-triangulation $\bt$ is obtained either by gluing $\uleft$ and $\uright$ to $\rmap$ under $\prob\py$, or by gluing $\uleft[\infty]$ and $\uright[\infty]$ to $\rmap[\infty]$ under $\prob\yy$. To be precise, one needs to move the root vertex of $\rmap$ before each gluing: Given a map $\map$ with a simple boundary, and an integer $S$, let us denote by $\overrightarrow \map^{S}$ (resp. $\overleftarrow \map^{S}$) the map obtained by translating the root vertex of $\map$ by a distance $S$ to the right (resp. to the left) along the boundary. The proof of the following lemma is left to the reader.

\begin{lemma}[Local convergence is preserved by a finite translation of the root]\label{lem:loc cv & reroot}
Assume that $S$ is almost surely finite under $\prob\py$ and $\prob\yy$.
If $\law\py \map \to \law\yy \map$ locally in distribution jointly with $\law\py S \to  \law\yy S$ as $p\to\infty$, then $\law\py \overrightarrow \map^{S}$ also converges to $\law\yy \overrightarrow \map^{S}$ locally in distribution.
\end{lemma}

Under $\prob\py$, we have 
\begin{equation}\label{eq:gluing-of-3}
\bt = \overrightarrow{(\umap\rmap)}^{S^\+ + S^\<} \oplus \uright \qtq{where} \umap\rmap = \uleft \oplus \overleftarrow{(\rmap)}^{S^\<}
\end{equation}
and $S^\+$ and $S^\<$ are the distances from $\rho$ to $\rho_{\umap^*}$ and $\rho_\umap$, respectively. See Figure~\ref{fig:ribbon-at-large-jump}. Similarly, $\law\yy \bt$ can be expressed in terms of $\uleft[\infty]$, $\rmap[\infty]$, $\uright[\infty]$ and $S^\jj$ using gluing and root translation.

On the event $\mathcal{J}$, the perimeter processes $\nseq X$ and $\nseq Y$ stay above the barrier $\mu n-xf_\epsilon(n)$ up to time $\tauxy$. Thus their minima over $[0,\tauxy)$ are reached before the deterministic time $N_{\min} = \sup\Set{n\ge 0}{\mu n-xf_\epsilon(n)\le 0}$ and $S^\+$ and $S^\<$ are measurable functions of the explored map $\emapo_{N_{\min}}$. It follows that $\law\py S^\jj$ converges in distribution to $\law\yy S^\jj$ on the event $\mathcal{J}$, in a sense similar to \eqref{eq:cv not glued} and \eqref{eq:cv glued} in Lemma~\ref{lem:gluing of loc cv}. This convergence also takes place jointly with the one in Lemma~\ref{lem:gluing of loc cv}. Using the relation \eqref{eq:gluing-of-3}, it is not hard to adapt the proof of Lemma~\ref{lem:gluing of loc cv} and deduce from Lemma~\ref{lem:loc cv on big jump} the local convergence of the Ising-triangulation $\bt$ on the event $\mathcal{J}$, in the following sense.

\begin{corollary}
Fix any $x,m,\epsilon > 0$. Then for any radius $r\ge 0$ and any set $\mathcal{E}$ of balls, we have
\begin{equation*}
\limsupp \mb|{\, \prob\py ( \btsq_r \in \mathcal{E} ) 
               - \prob\yy( \btsq_r \in \mathcal{E} ) \, }    \ \le\ 
\limsupp \prob\py (\mathcal{J}^c) + \prob\yy (\tauxy<\infty)    \,.
\end{equation*}
\end{corollary}
The left hand side does not depend on the parameters $x,m$ and $\epsilon$ used to define the ribbon $\rmap$ and the event $\mathcal{J}$. Therefore to conclude that $\prob\py$ converges locally to $\prob\yy$, it suffices to prove that:

\begin{lemma}\label{lem:bad gluing proba upper bound}
$\displaystyle \limsupp \prob\py (\mathcal{J}^c) + \prob\yy (\tauxy<\infty)$\, converges to zero when $x,m\to\infty$ and $\epsilon\to 0$.
\end{lemma}

\begin{proof}
When $x\to\infty$, we have $\tauxy \to\infty$ almost surely under $\prob\yy$, hence $\prob\yy(\tauxy<\infty) \cv[]x 0$. For the probability of $\mathcal{J}^c$, we use the union bound
\begin{equation*}
\prob\py(\mathcal{J}^c) \ \le\ \prob\py (\tauxy< T_m) + \prob\py(T_m<\epsilon p) + \prob\py (\tauxy = T_m \text{ and } \kk>m) \,.
\end{equation*}
The first two terms on the right can be bounded using Lemma~\ref{lem:one jump} and Proposition~\ref{prop:scaling}:
\begin{align*}
\lim_{m,x \to\infty}  & \limsupp \prob\py(\tauxy<T_m)    \ =\ 0 \\ \tq{and}
\lim_{\epsilon\to 0}\ & \limsupp \prob\py(T_m<\epsilon p)\ =\ 
\lim_{\epsilon\to 0} 1-(1+\mu \epsilon)^{-4/3} \ =\ 0    \,.
\end{align*}

For the last term, let us first fix some $n\ge 1$ and consider $\prob\py(\tauxy = T_m = n \text{ and } \kk>m)$.
Since $g_{x,\epsilon} := \max_{n\ge 0} xf_\epsilon(n) -\mu n$ is finite, for large $p$ we have $p+\mu n-xf_\epsilon(n) \ge p-g_{x,\epsilon} > m$ for all $n\ge 0$. It follows that $\tauxy\le T_m$ and $\{\tauxy = T_m = n\} = \{\tauxy>n-1 \text{ and } P_n\le m\}$. Notice that the event $\{\tauxy>n-1\}$ is $\filtr^\circ_n$-measurable and $P_{n-1}\ge p-g_{x,\epsilon}$ on that event. Hence by the spatial Markov property,
\begin{align*}
   &\ \prob\py(\tauxy = T_m = n \text{ and } \kk > m)                       \\
  =&\ \E\py \mb[{ \idd{\tauxy>n-1} \prob\py[p'] \mb({ 
                  \Step_1\in \{\rp[p+k],\rn[p+k] \,:\, k>m\} \text{ and } P_1\le m} 
               \big|_{p'=P_{n-1}} }                                         \\
\le&\ \E\py \mb[{ \idd{\tauxy>n-1} \prob\py[P_{n-1}] (P_1\le m) } \cdot
            \sup_{p'\ge p - g_{x,\epsilon} } \!\! \prob\py[p'] \mmb({
                  \Step_1\in \{\rp[p+k],\rn[p+k] \,:\, k>m\} }{P_1\le m}    \\
  =&\ \prob\py (\tauxy = T_m = n) \cdot
            \sup_{p'\ge p - g_{x,\epsilon} } \!\! \prob\py[p'] \mmb({
                  \Step_1\in \{\rp[p+k],\rn[p+k] \,:\, k>m\} }{P_1\le m}    \,.
\end{align*}
Summing over $n\ge 1$ and then taking the limit $p\to\infty$ gives
\begin{equation*}
     \limsupp \prob\py(\tauxy = T_m \text{ and } \kk > m)
\ \le\ \limsupp \prob\py \mmb({\Step_1\in \{\rp[p+k],\rn[p+k] \,:\, k>m\} }{P_1\le m} \,.
\end{equation*}

The number $P_1$ is determined by $\Step_1$. More precisely, from the relation between $\Step_1$ and $X_1$ in Table~\ref{tab:prob(p)} one can see that $\{P_1\le m\} = \{\Step_1 \in \{\rp[p+k],\rn[p+k] \,:\, k>-m\} \cup \{\rp[p-m]\} \}$. Thus the right hand side of the above inequality can be rewritten as
\begin{equation}\label{eq:kk bound ratio}
\frac{ \sum \limits_{k> m} 
       \lim \limits_{p\to\infty} p\cdot \prob\py(\Step_1 \in \{\rp[p+k], \rn[p+k]\})
    }{ \lim \limits_{p\to\infty} p\cdot \prob\py(\Step_1 = \rp[p-m])
  \  + \sum \limits_{k>-m} 
       \lim \limits_{p\to\infty} p\cdot \prob\py(\Step_1 \in \{\rp[p+k], \rn[p+k]\}) }
\end{equation}
provided that the limits in the numerator and the denominator exist and commute with the summations. The existence of the limits can be verified directly using the data in Table~\ref{tab:prob(p)}:
\begin{equation*}
\lim_{p\to\infty} p\cdot \prob\py(\Step_1 \in \{\rp[p+k], \rn[p+k]\})
\ =\ -\frac{4t_c}{3b} u_c^{|k|} \begin{cases}
    \nu_c a_0 a_{k+1} + a_1 a_k     &\text{if } k\ge 0    \\
a_0 a_{|k|+1} + \nu_c a_1 a_{|k|}   &\text{if } k\le 0 \,.
\end{cases}
\end{equation*}
One can also compute their sum over $k\ge 0$ and check that it commutes with the limit:
\begin{align*}
\lim_{p\to\infty} \sum_{k\ge 0} p\cdot \prob\py(\Step_1 \in \{\rp[p+k], \rn[p+k]\})
\ =\ & 
\lim_{p\to\infty} p\cdot \mB({ \nu_c\frac{t_c}{u_c} a_0 \frac{Z_p(u_c)-z_{p,0}}{a_p}
                            + t_c u_c a_1 \frac{Z_{p+1}(u_c)}{a_p} } \\
\ =\ -\frac{4t_c}{3b} \mB({ \nu_c a_0 \frac{A(u_c) - a_0}{u_c} + a_1 A(u_c) }
\ =\ &
\sum_{k\ge 0} \lim_{p\to\infty} p\cdot \prob\py(\Step_1 \in \{\rp[p+k], \rn[p+k]\})
\,.
\end{align*}
It follows that \eqref{eq:kk bound ratio} is indeed an upper bound of $\limsupp \prob\py(\tauxy = T_m \text{ and } \kk>m)$. It is clear that the ratio in \eqref{eq:kk bound ratio} converges to zero when $m\to\infty$. Therefore
\begin{equation*}
\lim_{m\to\infty} \limsupp \prob\py(\tauxy = T_m \text{ and } \kk>m) \ =\ 0 \,.
\end{equation*}
This completes the proof.
\end{proof}

\begin{remark*}
The upper bound of $\prob\py(\tauxy=T_m\text{ and }\kk>m)$ in the proof of Lemma~\ref{lem:bad gluing proba upper bound} can be refined to the following identity in the limit $p\to\infty$: for any fixed $m\ge 0$ and $k\ge 0$, the random variables $\kk$ and $\delta$ (the latter is defined in Figure~\ref{fig:ribbon-at-large-jump}) satisfy
\begin{align*}
\lim_{p\to\infty} \prob\py \mm({\kk=k \text{ and } \delta=1}{\tauxy=T_m} \ &=\
\lim_{p\to\infty} \prob\py \mm({\Step_1 = \rp[p+k]}{P_1\le m}    \\
\lim_{p\to\infty} \prob\py \mm({\kk=k \text{ and } \delta=0}{\tauxy=T_m} \ &=\
\lim_{p\to\infty} \prob\py \mm({\Step_1 = \rn[p+k]}{P_1\le m}    \,.
\end{align*}
The limits on the left hand side of the above equalities define a probability distribution supported on $\Set{(k,\delta) \in \integer \times \{0,1\}}{ \delta-k\le m}$, where the condition $\delta-k\le m$ comes from the fact that $P_{T_m} = \delta + \max(0,-\kk) \le m$.

One can further take the limit $m\to\infty$ in the above equalities. The result defines a probability distribution given by normalizing the weights $w(k,\delta)$ on $(k,\delta) \in \integer \times \{0,1\}$, where 
$w(k,1) = \lim\limits_{p\to\infty} p\cdot \prob\py (\Step_1 = \rp[p+k])$ and 
$w(k,0) = \lim\limits_{p\to\infty} p\cdot \prob\py (\Step_1 = \rn[p+k])$, or explicitly,
\begin{equation*}
\frac{3|b|}{4t_c} \cdot \omega(k,\delta) \ =\ \nu_c^{1-\delta} u_c^{|k|}\cdot
\begin{cases}
a_\delta \, a_{k+1-\delta}     &\text{if }k\ge 0    \\
a_{1-\delta} \, a_{|k|+\delta} &\text{if }k\le 0    \ .
\end{cases}   
\end{equation*}
We interpret this distribution as the distribution of the peeling event immediately after the large jump of the perimeter process $\nseq P$ in the infinite Ising-triangulation of law $\prob\yy$. (Of course, the large jump $\prob\yy$-almost surely never occurs. So this is only an interpretation.)

Recall that the peeling process explores the triangles adjacent to the leftmost interface from $\rho$ to $\rho^\dagger$. This set of triangles is invariant in distribution when $\rho$ and $\rho^\dagger$ swap their roles. For this reason, the distribution in the last paragraph should be related to the distribution of $S^\+$ and $S^\<$ (see Figure~\ref{fig:ribbon-at-large-jump}) under $\prob\yy$. The derivation of the exact relation, though conceptually straightforward, is very tedious and will not be carried out here.
\end{remark*}

\section{Properties of interfaces and spin clusters}\label{sec:interface}

In this section, we discuss some properties of the interfaces and the spin clusters in the infinite Ising-triangulations of the laws $\prob\py$ and $\prob\yy$ which are direct consequences of our construction of the laws. These include the statements \hyperref[thm:cv]{(2)} and \hyperref[thm:cv]{(3)} of Theorem~\ref{thm:cv}.
First, let us take a closer look at the definition of the spin clusters and their relation to the interfaces.

\paragraph{Vertex-connected clusters and edge-connected clusters.}
Since in our model the spins are on the faces of the triangulation, there are two equally natural definitions of the spin clusters. Two faces can be considered adjacent as soon as they share a vertex, or they can be considered adjacent only when they share an edge. The resulting connected components of faces of the same spin will be called \emph{vertex-connected clusters} in the first case, and \emph{edge-connected clusters} in the second case.
Obviously vertex-connected clusters are larger than their edge-connected counterparts. Notice that an edge-connected cluster of spin \< is surrounded by vertex-connected clusters of spin \+, and vice versa, see Figure~\refp{a}{fig:ribbon-lrmp}.
\begin{figure}
\centering
\includegraphics[scale=1]{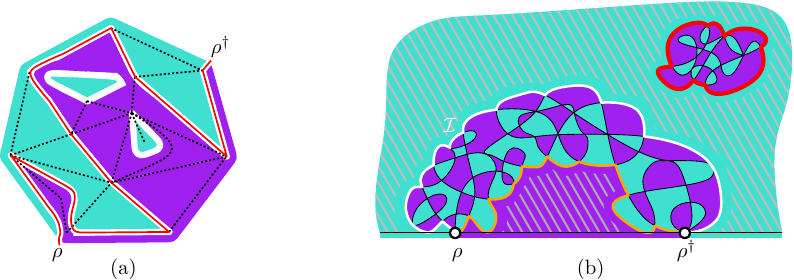}
\caption[caption]{(a) This example contains one vertex-connected cluster of spin \+ and three edge-connected clusters of spin \<. The leftmost interface from $\rho$ to $\rho^\dagger$ is highlighted in red.
(b) The leftmost interface $\mathcal I$ (white) and the rightmost interface (yellow) from $\rho$ to $\rho^\dagger$ in a bicolored triangulation with Dobrushin boundary condition. A vertex-connected cluster not touching the boundary is also shown. Its outermost interface is highlighted in red.}
\label{fig:ribbon-lrmp}
\end{figure}

Notice that we have not specified the type of the infinite clusters in Theorem~\refp{2-3}{thm:cv}. By this we mean that the two statements are valid for both edge-connected and vertex-connected clusters. The same applies to the following discussion.

\paragraph{Cluster structure of bicolored triangulations with a Dobrushin boundary condition.}
By convention, we shall consider consecutive boundary edges of the same spin to be in the same cluster, as in Figure~\refp{a}{fig:ribbon-lrmp}. This implies that, in a bicolored triangulation with a non-monochromatic Dobrushin boundary condition, there will be exactly one cluster containing the \+ boundary edges, and one cluster containing the \< boundary edges. All the other clusters are non-adjacent to the external face.

As shown in Figure~\refp{b}{fig:ribbon-lrmp}, the leftmost interface $\iroot$ from $\rho$ to $\rho^\dagger$ separates the \emph{edge-connected cluster} containing the \< boundary from the \emph{vertex-connected cluster} containing the \+ boundary. Similarly, the rightmost interface from $\rho$ to $\rho^\dagger$ separates the \emph{vertex-connected cluster} of containing the \< boundary from the \emph{edge-connected cluster} containing the \+ boundary. On the other hand, a spin cluster that does not touch the boundary has an \emph{outermost interface}, as highlighted in the example in Figure~\refp{b}{fig:ribbon-lrmp}.

\paragraph{Peeling process along the leftmost interface.}
Recall that, when the boundary of the unexplored map is \emph{not} monochromatic, we defined the peeling process to reveal the triangle adjacent to the \< boundary edge on the left of the root $\rho_n$ of the unexplored map. As shown in Figure~\refp{b}{fig:peeling-interface-ribbon}, as long as the revealed triangle has spin \< and does not swallow $\rho_n$, the peeling process turns around the vertex $\rho_n$ and does not extend the Ising interface. When the peeling process reveals a triangle of spin \+ incident to $\rho_n$, the Ising interface is extended by one edge which is the leftmost non-monochromatic edge adjacent to $\rho_n$.
Therefore the peeling process indeed explores the triangulation along the leftmost Ising interface from $\rho$.

\begin{figure}
\centering
\includegraphics[scale=0.85,page=8]{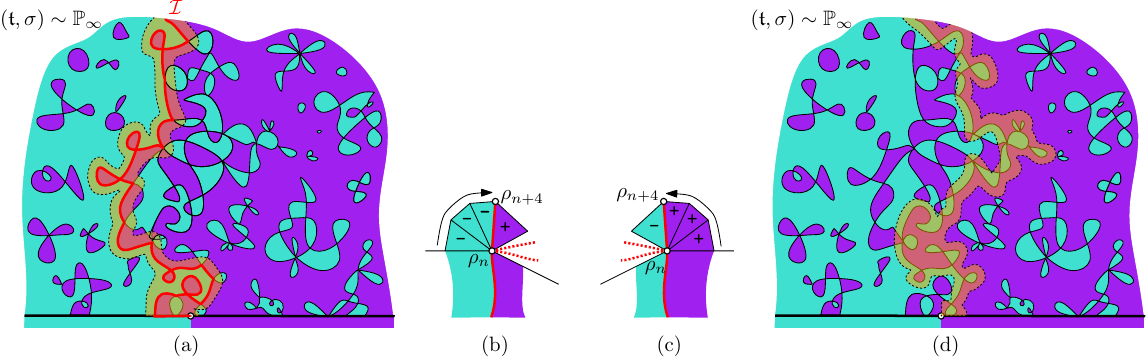}
\caption{(a) Position of the ribbon (shadowed region) relative to the spin clusters in the Ising-triangulation of law $\prob\yy$. The peeling process follows the leftmost interface $\iroot$ from $\rho$. 
(b) If the peeling process peels on the left of the root $\rho_n$ of the unexplored map, then it explores the leftmost interface.
(c),(d) the same pictures as (a),(b), but in the case when the peeling process explores the rightmost interface from $\rho$.
}
\label{fig:peeling-interface-ribbon}
\end{figure}

When the boundary of the unexplored map becomes monochromatic \< (that is, when $n=T_0$), the peeling process chooses \emph{some} triangle on its boundary to reveal (according to the peeling algorithm $\algo$) until the boundary becomes non-monochromatic again. In terms of the clusters, this means that after exploring the entire leftmost interface from $\rho$ to $\rho^\dagger$, the peeling process wanders into the bulk of the edge-connected cluster containing the \< boundary, and waits until the first time that it encounters again a triangle $\Delta$ of spin \+. 
After that, the peeling process turns around the vertex-connected cluster containing the triangle $\Delta$ in the clockwise direction. It finishes exploring it when the boundary of the unexplored map becomes monochromatic \< again.

The above observations on the relation between peeling process and the cluster structure allow us to deduce Theorem~\refp{2-3}{thm:cv} from what we know about the perimeter processes.

\begin{proof}[Proof of Theorem~\refp{2-3}{thm:cv}]
Under $\prob\py$, the stopping time $T_0$ is almost surely finite, therefore the leftmost interface $\iroot$ from $\rho$ to $\rho^\dagger$ is finite. Since the Ising-triangulation of law $\prob\py$ is one-ended, it follows that the vertex-connected cluster containing the \+ boundary is finite almost surely. By the Markov property, the perimeter process $\nseq P$ hits zero infinitely often almost surely, which shows that every vertex-connected cluster of spin \+, as shown in Figure~\refp{b}{fig:ribbon-lrmp}, is almost surely finite. This proves Theorem~\refp{2}{thm:cv}.

Under $\prob\yy$, the Ising-triangulation is composed of two copies of $\law\py[0] \bt$, the Ising-triangulation of law $\prob\py[0]$ (with a spin inversion in one of them), and a ribbon consisting of triangles adjacent to the leftmost interface $\iroot$, see Figure~\refp{a}{fig:peeling-interface-ribbon}. As shown in the previous paragraph, there is exactly one infinite cluster in each copy of $\law\py[0]\bt$. In the ribbon, there are two infinite clusters (one of each spin) along the two sides of its boundary. However, since the ribbon is glued to the copies of $\law\py[0]\bt$ along infinitely many edges, the two infinite clusters in the ribbon almost surely merge with the infinite clusters in the copies of $\law\py[0]\bt$ after gluing, thus leaving only \emph{two} infinite clusters in the Ising-triangulation of law $\prob\yy$. The fact that the ribbon touches the boundary only in a finite interval is due to the positive drift of the perimeter processes $\nseq X$ and $\nseq Y$.
\end{proof}

\paragraph{Peeling process along the rightmost interface.}
When constructing the peeling process, we could have chosen to reveal the triangle adjacent to the \+ boundary edge on the right of $\rho_n$ instead of the \< boundary edge on the left of $\rho_n$. By symmetry, this would define a peeling process along the \emph{rightmost} interface from $\rho$ (see Figure~\refp{c}{fig:peeling-interface-ribbon}).
Under $\prob\py$, this new peeling process would explore the boundary of the \emph{edge-connected} cluster containing the \+ boundary edges, as shown in Figure~\refp{b}{fig:ribbon-lrmp}.
Under $\prob\yy$, the new peeling process would explore the boundary separating the infinite 
\emph{vertex-connected} cluster of spin \< from the infinite \emph{edge-connected} cluster of spin \+ (Figure~\refp{d}{fig:peeling-interface-ribbon}).

This change from left to right will change the law of the first peeling event $\Step_1$ and the relation between $\Step_1$ and $(X_1,Y_1)$ in Table~\ref{tab:prob(p,q)}, thus changing the law of the peeling process and the perimeter processes. However, the results in Theorem~\ref{thm:z_p,q}, \ref{thm:scaling limit}, \ref{thm:cv} and Proposition~\ref{prop:comb cv} will not change except for the value of the constants $b$, $c_x$ and $c_y$. Their proofs can also be carried out in the same way. We leave the reader to check the above claim by constructing the counterpart of Table~\ref{tab:prob(p,q)} and carrying out calculations using the data in it.

Interestingly, peeling along the rightmost interface gives a different construction of the law $\prob\yy$, which splits the infinite triangulation with a different ribbon. The relation between the ribbon in the old construction and the ribbon in the new construction is illustrated in Figure~\refp{a,d}{fig:peeling-interface-ribbon}. Of course, the two constructions yield the same result because they both construct the local limit of $\prob_{p,q}$ when $q\to\infty$ and then $p\to\infty$.

Under a global spin inversion and a mirror reflection of the triangulation, an Ising-triangulation of law $\prob_{p,q}$ becomes an Ising-triangulation of law $\prob_{q,p}$ and the leftmost interface in the former is mapped to the rightmost interface in the latter. Therefore our claim that the peeling along the leftmost interface and the peeling along the rightmost interface defines the same law $\prob\yy$ implies that $\prob\yy$ is also the local limit of $\prob_{p,q}$ when $p\to\infty$ and then $q\to\infty$. This is one of the facts that support the conjecture that $\Prob_{p,q} \to \Prob\yy$ when $p,q\to\infty$ at any relative speed.

\paragraph{Perimeters of the clusters.}
More quantitative properties of the clusters in the infinite Ising-triangulations can also be derived from the construction of $\prob\py$ and $\prob\yy$. In the rest of this section we will discuss the relation between the perimeter processes $\nseq{X_n,Y}$ and the actual perimeter of the spin clusters.

We have seen that when the boundary of the unexplored map is non-monochromatic, the peeling process explores the perimeter of spin clusters: either the leftmost interface of the cluster containing the \+ boundary, or the outermost interface around a cluster of spin \+ not touching the boundary.
More precisely, each peeling step contributes additively to the length of the perimeter being explored, and conditionally on the sequence $\nseq[1] \Step$ of peeling events, the contributions of the different steps are independent random variables.

\begin{figure}
\centering
\includegraphics[scale=1.1]{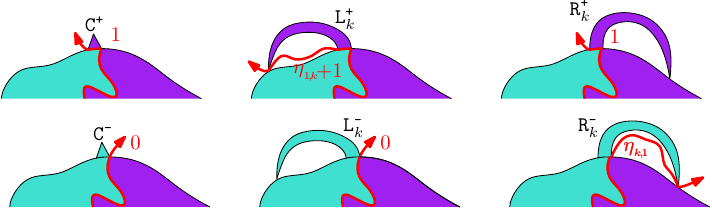}
\caption{The contribution to the length of the interface by the 6 types of peeling events.}
\label{fig:perimeter-increase}
\end{figure}

Let $\eta_{p,q}$ denote the total length of the leftmost interface $\iroot$ from $\rho$ to $\rho^\dagger$ in an Ising-triangulation of law $\prob_{p,q}$.
The contribution to the length $\eta_{p,q}$ made by each peeling event is summarized in Figure~\ref{fig:perimeter-increase}. Notice that when $\Step_n = \lp$ or $\Step_n = \rn$, the peeling step swallows a region that contains the interface being explored, and the contribution to the total length is given by $\eta_{1,k}+1$ or $\eta_{k,1}$, respectively. It follows that $\eta_{p,q}$ satisfies the following equation in distribution:
\begin{equation*}
\eta_{p,q} \ \overset{(d)}=\ \sum_{n=1}^{T_0-1} F(\Step_n) + G(\Step_{T_0})
\end{equation*}
where the random variables on right hand side are taken under $\prob_{p,q}$, and $F(\Step_n)$ is 0, 1, or an independent random variable with the law of $1+\eta_{1,k}$ or $\eta_{k,1}$, determined according to Figure~\ref{fig:perimeter-increase}. The last peeling step along $\iroot$ occurs at time $T_0$. Its contribution to the total length depends on $\Step_{T_0}$ in a different way than the previous steps. We leave the interested reader to work out its exact distribution $G(\Step_{T_0})$.

The above discussion is also valid when $q=\infty$. If we assume in addition that $p$ is large, then the contribution of the last step $\Step_{T_0}$ to the total length of the leftmost interface will be negligible, and 
$\eta_p \equiv \eta_{p,\infty}$ satisfies
\begin{equation*}
\frac1p \eta\py \ \overset{(d)}=\ \frac1p \sum_{n=1}^{T_0-1} F(\Step_n) + o(1) \,.
\end{equation*}
When $p\to\infty$, Proposition~\ref{prop:scaling} states that $p^{-1} \law\py T_0$ has a limit in distribution. Moreover, the terms in the above sum converge in distribution to an i.i.d.\ sequence of law $F(\law\yy \Step_1)$. 
Recall that $\prob\yy(\Step_1 = \lp) \sim c_1 k^{-7/3}$ and $\prob\yy(\Step_1 = \rn) \sim c_2 k^{-7/3}$ when $k\to\infty$, for some constants $c_1$ and $c_2$ (see Table~\refp{b}{tab:prob(p)}). Therefore according to Figure~\ref{fig:perimeter-increase}, the random variable $F(\law\yy \Step_1)$ has a finite expectation if and only if $\sum_{k\ge 0} (k+1)^{-7/3} \E[\eta_{1,k}+\eta_{k,1}]<\infty$. If this is indeed the case, then the perimeter $\eta_p$ will have a scaling limit similar to the one of $\law\py T_0$:

\begin{proposition}\label{prop:real perimeter}
Assume that $\sum_{k\ge 0} (k+1)^{-7/3} \E[\eta_{1,k}+\eta_{k,1}]<\infty$. Then the total length $\eta\py$ of the leftmost interface in $\law\py \bt$ has the scaling limit
\begin{equation*}
\prob(\eta\py >xp) \cv[]p (1+\mu' x)^{-4/3}
\end{equation*}
where $\mu' = \mu / \E\yy[F(\Step_1)]$.
\end{proposition}

We will not prove this claim in this paper. Its proof is an adaptation of the proof of Proposition~\ref{prop:scaling}.
While we do not have a proof of the first moment condition in Proposition~\ref{prop:real perimeter}, we could easily verify a similar condition for the Ising-triangulation with a general (i.e.\ not necessarily simple) boundary. Since the finiteness of the expected perimeter should be a geometric property of the scaling limit of the model, by universality, we believe that the same property also holds for the current model.

In private communications, Bertrand Duplantier, Ewain Gwynne and Scott Sheffield suggested that the interface length discussed above should converge in distribution to the quantum length of the interface resulted from the gluing of a quantum disk with a \emph{thick quantum wedge} (a ``half-plane like'' quantum surface) of $\sqrt{3}$-Liouville quantum gravity (see \cite{DMS14, AG19}). More precisely, consider a quantum disk with two distinguished boundary points $\rho$ and $\rho^\dagger$. Denote by $L$ (resp.~$R$) the quantum lengths of the boundary of the disk from $\rho$ to $\rho^\dagger$ in the clockwise (resp.~counter-clockwise) order.
If the total perimeter $L+R$ is sampled from the infinite measure $t^{-7/3}\idd{t>0}\dd t$ (i.e.~the Lévy measure of a spectrally positive $4/3$-stable Lévy process), then conditionally on $R=1$, the law of $L$ is given by \begin{equation*}\label{quantumperimeter}
\prob(L>x)=(1+x)^{-4/3}.
\end{equation*} 
This piece of boundary of length $L$ can then be glued conformally to a segment of length $L$ on the boundary of an independent thick quantum wedge, which will give rise to the picture in Figure~\refp{b}{fig:ribbon-lrmp} (with all the interfaces from $\rho$ to $\rho^\dagger$ crashed to a simple curve).

It is widely believed that the scaling limit of the Ising model is related to Liouville quantum gravity of parameter $\sqrt 3$ \cite{Ber16, DKRV16, DMS14}. On the other hand, $t^{-7/3}\idd{t>0}\dd t$ is the natural law of the perimeter of a quantum disk in $\sqrt 3$-Liouville quantum gravity, in the sense that it is the law of the perimeter (measured by the quantum length) of a typical disk cut out by the SLE$_{16/3}$ curve from the $\sqrt 3$-quantum wedge in the standard SLE--LQG coupling picture \cite{DMS14, AG19}. The boundary of such a disk should also correspond to a CLE$_3$ loop, which in turn is believed to be a scaling limit of the boundary of a finite Ising cluster.
Notice that although the process $(\mathcal X_t,\mathcal Y_t)_{t\ge 0}$ in Theorem~\refp{1}{thm:scaling limit} has the same law (up to rescaling) as the horodistance process associated with the exploration of the quantum wedge by an SLE$_{16/3}$ curve \cite{DMS14}, the two are not directly related as they describe different phenomena at different length scales: The former describes the fluctuations of the Ising interface around its mean (which should be a simple curve related to SLE$_3$) on a scale smaller than the length of the interface, while the latter describes the interface in the dual FK-Ising model.

\appendix

\section{Elimination of the second catalytic variable in Tutte's equation}\label{sec:2nd cat}
In Section~\ref{sec:1st cat} we showed how to eliminate one of the two catalytic variables $(u,v)$ in Tutte's equation by extracting appropriate coefficients of the series $Z(u,v)$. In the end we obtained an algebraic equation with one catalytic variable of the form
\begin{equation}\label{eq:1cat'}
Z_0(u) = 1+ \nu u^2 + t\, \mathcal{R}(Z_0(u),u,z_1,z_3;\nu,t) \tag{\ref{eq:1cat}'}
\end{equation}
where $\mathcal R(y,u,z_1,z_3;\nu,t)$, given explicitly by \eqref{eq:R as rational}, is a polynomial in $\frac yu, u,z_1,z_3,t$ and $\nu$.

To eliminate the second catalytic variable $u$, we use a generalization of the quadratic method used by Tutte in his study of properly colored triangulations \cite{Tut82b,Tut95}. It is later adapted in \cite[Section~12]{BBM11} to treat bicolored maps with monochromatic boundary condition. In our setting, this method consists of finding two rational functions $J(u,y)$, $L(u,y)$ and a polynomial $C(x)$ whose coefficients do not depend on $u$ or $Z_0(u)$, such that \eqref{eq:1cat'} can be written in the form
\begin{equation*}
	A\cdot L(u,Z_0(u))^2 = C(J(u,Z_0(u)))
\end{equation*}
where $A$ is some polynomial that may depend on all the variables. Then the square factor on the left hand side would suggest that $C(x)$ has a double root, in the same way as the classical quadratic method (see e.g.~\cite[Section~2.9]{GJ83}).

With some trial-and-errors, we discovered the following choice of $J$ and $L$:
\begin{align*}
	J(u,y) &= (\nu-1)\mB({ tu + \m({\frac{t y}u}^2 } - \frac{t y}u \,,
\\	L(u,y) &= 2\frac{t y}u + (\nu+1)J(u,y) \,.
\end{align*}
Notice that the mapping $(u,y)\mapsto(J,L)$ is invertible. Thus we can make the reverse change of variable and rewrite \eqref{eq:1cat'} as a polynomial equation satisfied by the variables $J$ and $L$, with coefficients in the space of formal power series $\complex(\nu)[[t]]$. As shown in \cite{CAS1}, we obtain the following equation as the result:
\begin{equation}\label{eq:doubly quadratic}
L^4 -2C_2(J)\, L^2 = C_0(J)
\end{equation}
where $L=L(u,Z_0(u))$, $J=J(u,Z_0(u))$, and $C_2$, $C_0$ are the following polynomials with coefficients in $\complex(\nu)[[t]]$:
\begin{align}
C_2(J)\ &=\
(\nu+1)^2 J^2 + \frac{2(\nu+3)}{\nu-1} J - 2(\nu^2-1)t^2 + \frac2{(\nu-1)^2} \,,\\
C_0(J)	\ &=\ -(\nu+1)^2
	\m({ \m({(\nu+1)J^2 -2(\nu-1)t^2}^2 +4J^3 +16(\nu-1)(t^3 z_1) J }
\\ & \phantom{=-(\nu+1)^2} -4 J^2\,+\,16(\nu+1)\nu t^2 J\, +\,16 w        \notag
\end{align}
where $w= -(\nu^2-1)^2 t^5 z_3 +(\nu^2-1)^2 \m({ 2t^5 z_1^3 -\frac3{\nu +1}t^4 z_1^2 -\frac34 t^4 } + (\nu-3)\nu t^3 z_1 +t^2$. Notice that $z_3\mapsto w$ is just a linear change of variable for fixed $z_1$.

Now we derive heuristically an algebraic equation satisfied by $z_1$ and $t$. We will check \emph{a posteriori} that they lead to the right solution.
We can write \eqref{eq:doubly quadratic} in two ways:
\begin{equation*}
			(L^2-2C_2(J))L^2 = C_0(J)
\qtq{and}	(L^2-C_2(J))^2 = C_0(J)+C_2^2(J)
\end{equation*}
If we view $t$ and $J$ as two independent variables, and view $L$ as a function of $(t,J)$. Then the above equations suggest that both $C_0$ and $C_0+C_2^2$, viewed as polynomials of $J$, have double roots. It is well known that this is characterized by their discriminants being zero.
\begin{equation*}
			D_1 = \mathtt{Discriminant}_J(C_0)=0
\qtq{and}	D_2 = \mathtt{Discriminant}_J(C_0+C_2^2)=0
\end{equation*}
$D_1$ and $D_2$ are polynomials in $t$, $z_1$ and the auxiliary variable $w$. Since they both vanish for the same value of $w$, their resultant with respect to $w$ must be zero. This provides a polynomial equation $R_w$ satisfied by $z_1(t)$ and $t$. We compute this equation in \cite{CAS1}. After removing irrelevant factors, we get an equation of degree 15 in $z_1$ and $t$.
Since $z_1(t)$ is an odd function of $t$, one can make the change of variable $\tilde z_1=t^3 z_1$ and $\tilde t=t^2$ in the equation satisfied by $z_1(t)$. This leads to an equation of degree 6 in $\tilde z_1$ and $\tilde t$, see \cite{CAS1}.

The discriminant $D_2=0$ provides an equation that relates $z_3(t)$ to $z_1(t)$ and $t$. Under the change of variables $\tilde z_3=t^9 z_3$, $\tilde z_1=t^3 z_1$ and $\tilde t=t^2$ and after removing irrelevant factors, it gives a quadratic equation for $\tilde z_3$. In \cite{CAS1} we check that this equation, as well as the equation of degree 6 relating $\tilde z_1$ to $\tilde t$, are both satisfied by the rational parametrizations \eqref{eq:RP z1-3}.

\section{Singularity analysis via rational parametrization}\label{sec:RP}

In this section we present a method to locate the dominant singularity of a combinatorial generating function from a proper rational parametrization of it. First let us clarify the definition of a rational parametrization.

\paragraph{Definition.} Let $\mathcal E\in\complex[x,y]$ be an irreducible polynomial. A pair of rational functions $\mathcal P=(\hat x,\hat y)$ is an \emph{(affine) rational parametrization} of the curve $\mathcal E(x,y)=0$ if $\mathcal E(\hat x(s),\hat y(s))=0$ for all but finitely many $s\in\complex$. Here a rational function is seen as a continuous mapping from $\ocomplex$ to $\ocomplex$. The rational parametrization $\mathcal P$ is
\begin{itemize}[topsep=4pt,noitemsep]
\item	\emph{real} if $\hat x$ and $\hat y$ can be written with real coefficients.
\item	\emph{proper} if $\mathcal P(s)=(x,y)$ has a unique solution $s$ for all but finitely many $(x,y)$ on $\mathcal E=0$.
\end{itemize}
We call $s\in\complex$ a \emph{critical point} of $\mathcal P$ if either $\hat x'(s)=0$ or $\hat y(s)=\infty$.

\bigskip


Proper parametrizations are minimal in the following sense. For all irreducible polynomial $\mathcal E(x,y)$, if $\mathcal E=0$ has a rational parametrization, then it also has a proper one, and if $\mathcal P$ is one proper parametrization of $\mathcal E=0$, then every rational parametrization of $\mathcal E=0$ is of the form $\mathcal P\circ h$ with some non-constant rational function $h$ \cite[Lemma 4.17]{SWP08}. It is not hard to see that $\mathcal P\circ h$ is itself proper if and only if $h(s)=\frac{a+bs}{c+ds}$. One can use this property to move the poles of $\hat x(s)$, e.g.\ to place one pole at $s=\infty$, while keeping the parametrization proper. It is also easy determine whether a given rational parametrization is proper by looking at its degrees \cite[Theorem 4.21]{SWP08}. One can check that all univariate rational parametrizations used in Section~\ref{sec:Tutte solution} are real and proper.

A rational parametrization $\mathcal P=(\hat x,\hat y)$ is defined with respect to an algebraic equation $\mathcal E=0$. But it is not immediately clear how $\mathcal P$ is related to the value of a function $\phi$ satisfying $\mathcal E(x,\phi(x))=0$, since a solution of the equation does not necessarily lie on the graph of the function.
To study properties of the function, we want the relation $\hat y=\phi\circ\hat x$.
If this relation holds in a neighborhood of $s_*\in\complex$, we say that \emph{$(\mathcal P,s_*)$ parametrizes $\phi$ locally at $x_*=\hat x(s_*)$}.

\begin{lem}\label{lem:proper para}
Assume that $\mathcal P=(\hat x,\hat y)$ is a proper parametrization of $\mathcal E(x,y)=0$.
\begin{enumerate}
\item If a function $\phi$ satisfies $\mathcal E(x,\phi(x))=0$ in a neighborhood of $x_*\in\complex\setminus \{\hat x(\infty)\}$, then there exists a unique $s_*\in\complex$ such that $(\mathcal P,s_*)$ parametrizes $\phi$ locally at $x_*$.
\item For all $s_*\in\complex$ such that $x_*:=\hat x(s_*)\ne \infty$, $(\mathcal P,s_*)$ parametrizes a finite-valued function $\phi$ locally if and only if $s_*$ is not a critical point of $\mathcal P$. In this case, $\phi$ is analytic at $x_*$.
\end{enumerate}
\end{lem}

\begin{proof}
(i) \emph{Existence.} Consider a sequence $(x_n)_{n\ge 0}$ of distinct complex numbers converging to $x_*$ such that $\mathcal E(x_n,\phi(x_n))=0$ for all $n$. According to the definition of proper parametrization, for all $n$ large enough there exists $s_n\in\complex$ such that $(x_n,\phi(x_n))=(\hat x(s_n),\hat y(s_n))$. Let $s_*$ be an accumulation point of $(s_n)_{n\ge n_0}$ in $\ocomplex$. By the continuity of $\hat x:\ocomplex\to\ocomplex$, $x_*=\hat x(s_*)\ne \hat x(\infty)$, thus $s_*\in\complex$. The analytic functions $\hat y$ and $\phi\circ\hat x$ coincide on a sequence of distinct points converging to $s_*$, so they must be equal in a neighborhood of $s_*$.

\emph{Uniqueness.} Assume that $(\mathcal P,s_*)$ parametrizes $\phi$ locally at $x_*$. Since a rational function is an open mapping, there exists a neighborhood $V$ of $x_*$ such that for all $x\in V$, $\mathcal P(s)=(x,\phi(x))$ has a solution close to $s_*$. But these solutions are unique except for finitely many values of $x$. Thus there is at most one $s_*\in\complex$ having the above property.
\medskip

\noindent(ii) If $\hat x'(s_*)\ne 0$ and $\hat y(s_*)\ne \infty$, then by the implicit function theorem, $\phi:=\hat y\circ(\hat x^{-1})$ is a well defined analytic function such that $\hat y=\phi\circ\hat x$ in a neighborhood of $s_*$.

Inversely, assume $\hat y=\phi\circ \hat x$ in a neighborhood of $s_*$ for some finite-valued function $\phi$. Then $\hat y(s_*)=\phi(x_*)\ne\infty$.
If $\hat x'(s_*)=0$, then for all $x\ne x_*$ in some neighborhood of $x_*$, $\hat x(s)=x$ has at least two distinct solutions. But this gives two distinct solutions to $\mathcal P(s) = (x,\phi(x))$ for infinitely many $x$, contradicting the properness of $\mathcal P$. Thus $\hat x'(s_*)\ne 0$.
\end{proof}

\newcommand{\disk}{\mathbb{D}}
\newcommand{\cdisk}{\overline{\mathbb{D}}}

Let $A,B$ be subsets of $\complex$.
We say that $\phi$ defines a \emph{conformal bijection} from $A$ to $B$ if $\phi|_A$ is a homeomorphism from $A$ to $B$ such that both $\phi|_A$ and $(\phi|_A)^{-1}$ are holomorphic in the interior of their domains.
The following proposition gives a general method for identifying the real dominant singularity of a generating function $\phi(x)$ with positive coefficients from one of its real proper rational parametrizations $(\hat x, \hat y)$, and verifying the uniqueness of the dominant singularity (iii). It also gives information on how domains in the parameter plane $s\in \complex$ are mapped conformally to (slit) disks in the plane $x\in \complex$ by $\hat x$.

\begin{proposition}\label{prop:RP dom sing}
Let $\phi(x)=\sum \phi_n x^n$ be a non-polynomial analytic function in a neighborhood of $0$, such that $\phi_n\ge 0$ for all $n$. Assume that $\phi$ satisfies an algebraic equation with a real proper rational parametrization $\mathcal P=(\hat x,\hat y)$ such that $\hat x(\infty)=\infty$.
\begin{enumerate}
\item
Let $s_0$ be the unique value of $s$ such that $(\mathcal P,s_0)$ parametrizes $\phi$ locally at $0$. Then $s_0$ is real.
\item
$\mathcal P$ has a unique real critical point $s_c$ characterized by the properties $\hat x(s_c)>0$, and $\mathcal P$ has no critical point strictly between $s_0$ and $s_c$. Moreover, $x_c:=\hat x(s_c)$ is the radius of convergence and a dominant singularity of $\phi$. In addition, there exists a compact neighborhood $\overline V_0$ of $s_0$ such that $s_c \in \partial \overline V_0$, and $\hat x$ defines a conformal bijection from $\overline V_0$ to $\cdisk_{x_c}$.
\item
If $\hat x$ has no critical point in $\partial \overline V_0 \setminus \{s_c\}$, then for all $\epsilon>0$ small enough, there exists a compact set $\overline V_\epsilon\supset \overline V_0$ such that $s_c\in \partial \overline V_\epsilon$ and $\hat x$ defines a conformal bijection $\overline V_\epsilon \to \cdisk^{|\epsilon}_{x_c}$.

If $\hat y$ has no pole in $\partial \overline V_0 \setminus \{s_c\}$ as well, then  $\phi$ has an analytic continuation on $\disk^{|\epsilon}_{x_c}$.

(In practice, instead of $\partial \overline V_0 \setminus \{s_c\}$, one usually checks the absence of critical points in the larger set $\Setn{s\in \complex}{|\hat x(s)|=x_c \text{ and } s\ne s_c}$.)
\end{enumerate}
\end{proposition}

\begin{proof}
(i) Lemma~\refp{i}{lem:proper para} ensures that $s_0$ exists and is unique. Since $\mathcal P$ and $\phi$ are real, the complex conjugate of $s_0$ satisfies the same condition. By uniqueness, it must be equal to $s_0$, that is, $s_0\in \real$.
\medskip

\noindent
(ii) Up to replacing $s$ by $-s$, we can assume that $\hat x'(s_0)>0$.
Then the characterization of $s_c$ reads: $s_c=\inf\{s\ge s_0\,:\, \hat x'(s)=0\text{ or }\hat y(s)=\infty\}$. Clearly, $\hat y\circ(\hat x^{-1})$ is an analytic continuation of $\phi$ on $[0,x_c)$. By Pringsheim's theorem, the radius of convergence of $\phi$ is at least $x_c$. It is well known that the only entire functions that satisfy algebraic equations are polynomials. Therefore $x_c<\infty$. Since $\hat x(\infty) = \infty$, we also have $s_c\ne \infty$, and hence $s_c$ is a critical point of $\mathcal P$.

If $\phi$ were analytic at $x_c$, then by analytic continuation, we would have $\hat y=\phi\circ\hat x$ in a neighborhood of $s_c$, i.e.\ $(\mathcal P,s_c)$ parametrizes $\phi$ locally at $x_c$. This contradicts Lemma~\refp{ii}{lem:proper para}.
Therefore, $x_c$ is a dominant singularity of $\phi$.

Let $V_0$ be the connected component of $\hat x^{-1}(\disk_{x_c})$ containing $s_0$.
Clearly, we have $s_c\in \partial V_0$. Since $\hat x(\infty)=\infty$, $V_0$ is bounded, so its closure $\overline V_0$ is a compact neighborhood of $s_0$.

By analytic continuation, $\hat y=\phi\circ\hat x$ on $V_0$. So Lemma~\refp{ii}{lem:proper para} implies that $V_0$ contains no critical point of $\mathcal P$. In particular, $\hat x'$ does not vanish on $V_0$.
On the other hand, $\hat x$ is continuous and $V_0$ is bounded, so $\hat x|_{V_0}$ is a proper map. Moreover, $\disk_{x_c}$ is simply connected.
Then by Hadamard's global inversion theorem \cite[Theorem 6.2.8]{ImplicitBook}, $\hat x|_{V_0}$ is a conformal bijection from $V_0$ to $\disk_{x_c}$.

The continuity of $\hat x$ implies that $\hat x|_{\overline V_0}$ is a surjection onto $\cdisk_{x_c}$. Now take $s_1,s_2 \in \partial V_0$ such that $\hat x(s_1)=\hat x(s_2)$. Since $\hat x$ is holomorphic and hence an open map, for any neighborhood $U_i$ of $s_i$ in $\overline V_0$, $\hat x(U_i)$ is a neighborhood  of $\hat x(s_i)$ in $\cdisk_{x_c}$. But $\hat x(U_1)\cap \hat x(U_2) \cap \disk_{x_c}$ is always non-empty and $\hat x$ is injective on $V_0$. This implies that $U_1\cap U_2$ is always non-empty. It follows that $s_1=s_2$. Therefore $\hat x|_{\overline V_0}$ is a bijection onto $\cdisk_{x_c}$. Since it is continuous and defined on a compact domain, its inverse is also continuous.

\medskip

\noindent
(iii)
Let $V_\epsilon$ be the connected component of $\hat x^{-1}(\disk^{|\epsilon}_{x_c})$ containing $s_0$.

Assume that $\hat x$ has no critical point in $\partial V_0 \setminus \{s_c\}$.
Then each $s\in \partial V_0 \setminus \{s_c\}$ has a neighborhood $N_s$ on which $\hat x$ is invertible. On the other hand, $s_c$ has a neighborhood $N_{s_c}$ in which $\hat x$ behaves like $x_c+ c\cdot (s-s_c)^n$ for some $c\ne 0$ and integer $n\ge 1$. Since $\partial V_0$ is compact, it can be covered by a finite number of these neighborhoods. With some thoughts, one can combine the local behavior of $\hat x$ in these neighborhoods to show that the complement of $\hat x^{-1}(\disk^{|\epsilon}_{x_c})$ contains a closed curve that converges to $\partial V_0$ when $\epsilon \to 0$. This implies that $\bigcap_{\epsilon>0} V_\epsilon \subseteq \overline V_0$.

It is not hard to see that $\bigcap_{\epsilon>0} V_\epsilon = V_0 \cup (\partial V_0 \setminus \{s_c\})$. Therefore $\hat x$ has no critical point in $\bigcap_{\epsilon>0} V_\epsilon$. Since $\hat x$ has only finitely many critical points in any bounded set, there exists $\epsilon>0$ such that $V_\epsilon$ contains no critical point of $\hat x$. Then the same construction as in part (ii) shows that $\hat x|_{\overline V_\epsilon}$ is a conformal bijection onto $\cdisk^{|\epsilon}_{x_c}$ (provided that one views the slit disk $\disk^{|\epsilon}_{x_c}$ as an open set in the double cover of the Riemann sphere with two simple branch points at $x_c$ and $\infty$).

If $\hat y$ has no pole in $\partial V_0 \setminus \{s_c\}$ as well, then $\phi = \hat y \circ (\hat x|_{V_\epsilon})^{-1}$ gives the analytic continuation of $\phi$ on $\disk^{|\epsilon}_{x_c}$.
\end{proof}

\begin{proof}[Proof of Lemma~\ref{lem:dom sing Z}]
Recall that we derived in Section~\ref{sec:sing anal} a rational parametrization of $Z$ of the form $(u,v)=(\hat u(H),\hat u(K))$ and $Z=\hat Z(H,K)$. We obtain a rational parametrization of $u\mapsto Z(u,u)$ by taking $K=H$:
\begin{equation}\label{eq:RP Z(u,u)}
\left\{
\begin{aligned}
	u = \hat u(H)	&\ =\ \frac{u_c}3 H (10 -12H +6H^2 -H^3)
\\	Z = \hat{Z}(H,H)	&\ =\ \frac{10 -12H +6H^2 -H^3}{10 -14H +7H^2 -H^3} Q(H)\,,
\end{aligned}
\right.
\end{equation}
where $Q$ is some polynomial of degree 6. In \cite{CAS1}, we check by explicit computation that
\begin{enumerate}[label=(\arabic*)]
\item $H_0=0$ is the only value of $H$ such that $\hat Z(H,H)=1$ and $\hat u(H)=0$.
\item $H_c=1$ is the (unique) real critical point of the rational parametrization \eqref{eq:RP Z(u,u)} such that $\hat u(H)>0$ and that there are no other critical points on $[H_0,H_c]=[0,1]$.
\item $H_c=1$ is the unique critical point of \eqref{eq:RP Z(u,u)} such that $\abs{\hat u(H)} = u_c$.
\end{enumerate}
Therefore by Proposition~\ref{prop:RP dom sing}, there is a neighborhood $V$ of $H_0=0$ such that $H_c=1\in \partial V$ and that $\hat u|_V$ is a conformal bijection from $V$ onto a slit disk $\slit{u_c}$ at $u_c$. It is clear that $\hat u(H)\uparrow u_c$ when $H\uparrow 1$ on the real axis. This proves (ii) of Lemma \ref{lem:dom sing Z}.

To prove (i), we notice that the coefficients $z_{p,q}$ are all positive. Thus the monotone convergence theorem implies
\begin{equation*}
\sum_{p,q\ge 0}z_{p,q}u_c^{p+q}\ =\ \lim_{u\uparrow u_c} Z(u,u)\ =\ \lim_{H\uparrow 1} \hat Z(H,H)\ <\ \infty\,.
\end{equation*}
It follows that $\sum_{p,q\ge 0} z_{p,q}u^pv^q$ is absolutely convergent for all $(u,v)\in \Dc^2$. On the other hand, if the series is absolutely convergent for some $(u,v)$ with $\abs{u}>u_c$, then by monotonicity the series $Z_0(u)=Z(u,0)$ will have a radius of convergence strictly larger than $u_c$. This is not the case because the rational parametrization \eqref{eq:RP H} implies that $Z_0(u)$ has a singularity of type $(u_c-u)^{4/3}$ at $u=u_c$.

Now let us fix a $u\in\Dc$ and prove (iii). Since the coefficients of the series $v\mapsto Z(u,v)$ are not necessarily non-negative, Proposition~\ref{prop:RP dom sing} does not apply. Instead, we will check (iii) directly using the formula $Z(u,v)=\hat Z(\hat u^{-1}(u),\hat u^{-1}(v))$ and the analytic properties of the function $\hat u$.
Recall that $\hat u$ induces a conformal bijection from some neighborhood $V$ of $H=0$ onto a slit disk $\slit{u_c}$ at $u_c$, which extends bi-continuously to $H=1$ by $\hat u(1) = u_c$.
Let $\overline U$ be the preimage of $\Dc$ by $\hat u|_{V\cup\{1\}}$, then
it suffices to show that
\begin{center}
(iii') for each $H\in \overline U$, $K\mapsto \hat Z(H,K)$ has no pole in $\overline U\setminus \{1\}$.
\end{center}
Indeed, since the poles of a \emph{univariate} rational function are isolated, (iii') implies that $K=1$ is the only possible pole of $K\mapsto \hat Z(H,K)$ in some neighborhood $U'$ of the compact $\overline U$. Its image $\hat u(U')$ is a neighborhood of the disk $\Dc$.
Since $\hat u$ is a conformal bijection onto $\slit{u_c}$, the composed function $v\mapsto \hat Z(\hat u^{-1}(u),\hat u^{-1}(v))$ is analytic on the intersection $\hat u(U') \cap \slit{u_c}$, which contains a slit disk at $u_c$.
On the other hand, $v\mapsto Z(u,v)$ must have a singularity at $u_c$, otherwise its radius of convergence would be strictly larger than $u_c$, contradicting (i). We conclude that $u_c$ is the unique dominant singularity of $v\mapsto Z(u,v)$ for all $u\in \Dc$.

In order to prove (iii'), we will show the following stronger statement: the denominator of $\hat Z(H,K)$ has no zero in $\overline U^2$ except at $(H,K)=(1,1)$.
We denote by $N$ and $D$ the numerator and the denominator of $\hat Z$ written in reduced form. The polynomial $D$ cannot have a zero $(H,K)\in \overline U^2$ which is not a zero of $N$, otherwise we would have $Z(u,v)\to\infty$ when $(u,v)\to(\hat u(H),\hat u(K))\in \Dc^2$, contradicting the fact that $\abs{Z(u,v)}\le Z(u_c,u_c)<\infty$ for all $(u,v)\in\Dc^2$. Now assume that $(H,K)$ is a common zero of $D$ and $N$ in $\overline U^2$. Then $H$ must be a zero of $Res(H)$, the resultant of $D(H,K)$ and $N(H,K)$ with respect to $K$. In \cite{CAS1}, we check by explicit computation that $H=0$ and $H=1$ are the only zeros of $Res(H)$ in $\overline U$. Moreover, $D(0,K)$ and $N(0,K)$ has no common zero in $\overline U$, and $K=1$ is the only common zero of $D(1,K)$ and $N(1,K)$. We conclude that on $\overline U^2$, the denominator $D(H,K)$ only vanishes at $(H,K)=(1,1)$, therefore (iii') is true.

The assertion (iv) follows from Proposition~\ref{prop:RP dom sing} thanks to the known properties of $\hat u$ and the fact that $\hat A$ has no pole on $[H_0,H_c]=[0,1]$. The application of Proposition~\ref{prop:RP dom sing} here assumes that the coefficients of the series $A(u)=\sum_{p\ge0} a_p u^p$ are non-negative. This is derived in Section~\ref{sec:sing anal} using only (i)-(iii) of Lemma~\ref{lem:dom sing Z}.
\end{proof}

\section{A one-jump lemma for the process $\law\py\nseq{X_n,Y}$}\label{sec:lemma proof}

\newcommand*{\ea}{\asymp}
\newcommand{\cst}{\mathrm{cst}}
\newcommand{\pp}[2]{\mathfrak{p}_{#1,#2}}
\newcommand{\pxx}[1][k]{\mathfrak{p}^x_{#1}}
\newcommand{\pyy}[1][k]{\mathfrak{p}^y_{#1}}
\newcommand{\Py}{\Prob\py{}}
\newcommand{\PY}{\Prob\yy{}}
\newcommand{\Ey}{\EE\py{}}
\newcommand{\EY}{\EE\yy{}}

We have seen in the discussion above Lemma~\ref{lem:one jump} that the lemma would become a standard law of iterated logarithm if the process $\law\py\nseq{X_n,Y}$ were replaced by $\law\yy \nseq{X_n,Y}$. Our proof of Lemma~\ref{lem:one jump} is based on the idea of comparing the transition probabilities of the Markov chain $\law\py \nseq{P_n,Y}$ (recall that $P_n=p+X_n$) to the step distribution of the random walk $\law\yy \nseq {X_n,Y}$, and the fact that $P_n\gtrsim p\to\infty$ for all $n<T_m$.  The mean technical difficulty is that the convergence $\law \py (X_1,Y_1) \to \law \yy (X_1,Y_1)$ of transition probabilities only implies the convergence of the process $\nseq {X_n,Y}$ up to finite time. But we want to estimate probabilities about the behavior of $\law\py \nseq{X_n,Y}$ up to time $T_m$, which is of order $\Theta(p)$. 

The proof follows the general strategy used in \cite{Borovkov2} to establish asymptotic behaviors of heavy-tailed random walks. It comes in three steps. 

First, we establish two estimates on the step distribution $\law\py(X_1,Y_1)$: one for the probability that $(X_1,Y_1)$ is far from the origin (Lemma~\ref{lem:estimates}\ref{item:estimate large jump}) and the other for the exponential moments of $(X_1,Y_1)$, restricted on the event that it remain close to the origin (Lemma~\ref{lem:estimates}\ref{item:estimate small jump}). 
Next, we bound the probability that $\law\py \nseq{X_n,Y}$ deviates to a distance $x\approx \barrier[\chi](N)$ from its mean on a time scale $N$ (Lemma~\ref{lem:cst barrier}). The process $\nseq{X_n,Y}$ may realize such a deviation either by making a jump of size $x$, or by accumulating steps of size smaller than $x$. We use the two estimates in Lemma~\ref{lem:estimates} to bound the probabilities of these two situations.
Finally, we complete the proof of Lemma~\ref{lem:one jump} by applying Lemma~\ref{lem:cst barrier} to an exponentially increasing sequence of time intervals.

To simplify notation, let us write 
\begin{equation*}
\pp k{k'} = \Prob\yy ( -(X_1,Y_1) = (k,k') )    \qtq{and}
\pxx = \Prob\yy (-X_1=k)    \ ,\quad
\pyy = \Prob\yy (-Y_1=k)    
\end{equation*}
The comparison of the distributions $\law\py(X_1,Y_1)$ and $\law\yy(X_1,Y_1)$ is based on the following observation: for all $k\le p-2$,
\begin{equation}\label{eq:Doob}
\Py ( -(X_1,Y_1) = (k,k') ) = \frac{a_{p-k} u_c^{p-k}}{a_p u_c^p} \pp k{k'}    \, .
\end{equation}
This can be seen by checking in Table~\ref{tab:prob(p)} that $\Prob\py(\Step_1= \step) = \frac{a_{p+X_1(\step)} u_c^{p+X_1(\step)} }{a_p u_c^p} \Prob\yy(\Step_1 = \step)$ for every peeling event $\step \in \steps$ such that $-X_1(\step) \le p-2$. (Recall that $X_1$ and $Y_1$ are determined by the peeling event $\Step_1$.)

\begin{figure}[b!]
\centering
\includegraphics[scale=1]{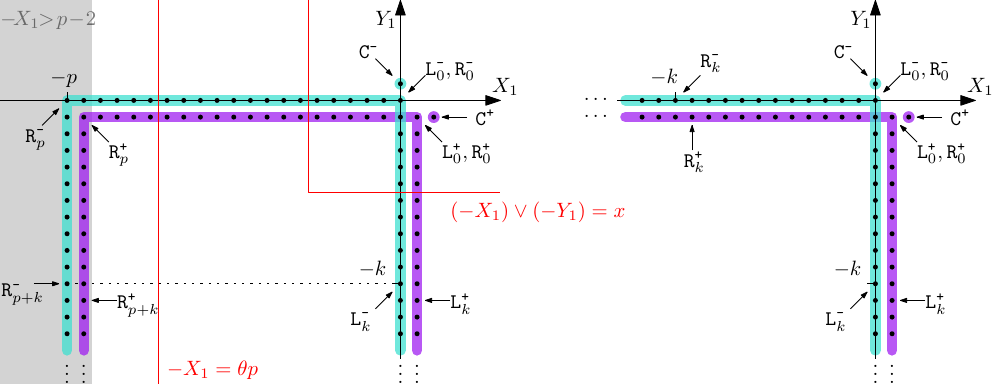}
\caption{The support of the distribution of $(X_1,Y_1)$ under $\Py$ (left) and $\PY$ (right).}
\label{fig:xy-support}
\end{figure}

If \eqref{eq:Doob} were valid for all $k$, it would mean that $\law\py \nseq{X_n,Y}$ is a Doob $h$-transform of the random walk $\law\py \nseq{X_n,Y}$. However, \eqref{eq:Doob} breaks down for $k>p-2$. More precisely, the supports of $(X_1,Y_1)$ under $\Py$ and $\PY$ differ: as illustrated in Figure~\ref{fig:xy-support}, the support of $\law\yy (X_1,Y_1)$ is contained in the L-shape defined by $-1\le X_1\vee Y_1 \le 1$ (except for one point), whereas the support of $\law\py (X_1,Y_1)$ stops at $X_1=-p$ (for the simple reason that $P_1=p+X_1\ge 0$) and continues in the negative $y$-direction. We control the probabilities in the part $\{-X_1>p-2\}$ of the support by the crude bound that for $e \in \{0,1\}$,
\begin{equation*}
        \Py(-(X_1,Y_1)=(p-e,k') ) 
\ \le \ \Py(-X_1 = p-e) 
\ \sim\ \cst_e \cdot p^{-1}
\qt{as }p \to\infty,
\end{equation*}
where the equivalence can be read from Table~\ref{tab:prob(p)}, and was seen in the proof of Proposition~\ref{prop:scaling}.

The probabilities in the rest of the support is controlled using \eqref{eq:Doob} in conjunction with the following asymptotics, seen respectively in Section~\ref{sec:XY} and in Theorem~\ref{thm:z_p,q}.
\begin{align*}
c_x^{-1} \pxx &\sim c_y^{-1} \pyy \sim k^{-7/3}    \qt{as }k \to\infty, \\
\tq{and} a_p u_c^p &= \cst \cdot p^{-4/3} (1 + O(p^{-1/3}))    \qt{as }p \to\infty.
\end{align*}

\paragraph{Notation:} If $A$ and $B$ are two \emph{positive} functions defined on some set $\Lambda$, we say that
\begin{itemize}
\item $A(\lambda) \la B(\lambda)$ for $\lambda \in \Lambda$, \emph{if}
there exists $C>0$ such that $A(\lambda)\le CB(\lambda)$ for all $\lambda \in\Lambda$;
\item $A(\lambda) \ea B(\lambda)$ for $\lambda \in \Lambda$, \emph{if}
$A(\lambda) \la B(\lambda)$	and $B(\lambda) \la A(\lambda)$.
\end{itemize}

The proof of the following properties of $\la$ and $\ea$ is left to the reader.

\begin{lemma}\label{lem:asymp}
\begin{enumerate}
\item
If $A_1\la B_1$ and $A_2\la B_2$, then $A_1A_2\la B_1B_2$ and $A_1+A_2\la B_1+B_2$.
\item
More generally,	if $A(\lambda)\la B(\lambda)$ for $\lambda \in \bigcup_{i\in I} \Lambda_i$, where $I$ is some arbitrary index set, \\then $\sum_{\lambda\in\Lambda_i} A(\lambda) \la \sum_{\lambda\in\Lambda_i} B(\lambda)$ for $i\in I$.
The same is true when $\la$ is replaced by $\ea$.
\item
When $\Lambda=\natural$, we have $A(\lambda)\la B(\lambda)$ \Iff{}
$A(\lambda)=O(B(\lambda))$ as $\lambda\to\infty$.\\
In particular, if $A(\lambda)\!\eqv\lambda\! B(\lambda)$, then $A(\lambda) \ea B(\lambda)$.
\end{enumerate}
\end{lemma}

We fix some $\theta \in [\frac12,1)$ and let $p_\theta = \frac2{1-\theta}$ so that $\theta p \le p-2$ for all $p\ge p_\theta$ (for example, $\theta=\frac12$ and $p_\theta = 4$).

\newcommand{\Xp}{\{-X_1\le p-2\}}
\newcommand{\ykXp}[1][=k]{\{-Y_1#1\} \cap \Xp}
\newcommand{\afrac}[1][k]{\frac{a_{p-#1} u_c^{p-#1}}{a_p u_c^p}}

\newcommand{\Ap}{\mathcal{A}_x}
\newcommand{\zg}[1][h]{\{W \ge #1\} \cap \Ap}
\newcommand{\zl}[1][h]{\{W \le #1\} \cap \Ap}
\newcommand{\idzg}[1][h]{\id_{\zg[#1]}}
\newcommand{\idzl}[1][h]{\id_{\zl[#1]}}

\begin{lemma}\label{lem:estimates}
\begin{enumerate}[label=(\roman*)]
\item\label{item:estimate infty}
$\pxx \ea \pyy \ea k^{-7/3}$ for $k\ge 1$.
\item\label{item:estimate Y}
$\Py(\ykXp) \ea k^{-7/3}$ for $p\ge 2$ and $k\ge 1$.
\item\label{item:estimate X}
$\Py(-X_1=k) \ea k^{-7/3}$ and $\Py(-X_1=p-k) \ea p^{-1} k^{-4/3}$ for all $p\ge p_\theta$ and $1\le k\le \theta p$.
\item\label{item:estimate afrac}
$\abs{\afrac -1} \la p^{-1} \abs{k} + p^{-1/3}$ for any $(k,p)$ such that $-2\le k \le \theta p$.

\item\label{item:estimate large jump}
For $p\ge p_\theta$, $x\in [1,\theta p]$ and $m\ge 1$,
\begin{equation*}
\Py(-X_1 < p-m \text{ and } (-X_1)\vee (-Y_1)\ge x)
\ \la \   x^{-4/3} + p^{-1} m^{-1/3}   \,.
\end{equation*}

In the following, let $\Ap=\{ (-X_1) \vee (-Y_1) \le x\}$ and $W$ be either $\mu-X_1$ or $\mu-Y_1$. 

\item\label{item:estimate S}
$\Py(\zg) \la h^{-4/3}$, $\Ey[W \idzg] \la h^{-1/3}$ and $\Ey[W^2 \idzl] \la h^{2/3}$ for $p\ge p_\theta$, $x\in [1,\theta p]$ and $h\in [1,x]$.

\item\label{item:estimate E}
$|\Ey[W  \id_{\Ap}]| \la x^{-1/3}$ for $p \ge p_\theta$ and $x\in [1,\theta p]$.

\item\label{item:estimate small jump}
For $p\ge p_\theta$, $x\in [1,\theta p]$ and $\lambda \in [2x^{-1},1]$,
\begin{equation*}
\log \m({ \Ey[e^{\pm \lambda W} \id_{\Ap}] }  
\ \la\  x^{-4/3} e^{\lambda x}.
\end{equation*}
\end{enumerate}
\end{lemma}

\newcommand{\varsubset}{\subset}
\renewcommand{\subset}{\subseteq}

\newcommand{\absb}[1]{\mb|{#1}}
\newcommand{\absB}[1]{\mB|{#1}}
\newcommand{\absh}[1]{\mh|{#1}}
\newcommand{\absH}[1]{\mH|{#1}}

\begin{proof}
\begin{enumerate}[label=\textbf{(\roman*)},wide=0pt,listparindent=\parindent,]
\item
follows from Lemma~\refp{iii}{lem:asymp} and the asymptotics $c_x^{-1} \pxx \sim c_y^{-1} \pyy \sim k^{-7/3}$.

\item
When $k=1$, it is not hard to see the \lhs\ is bounded away from zero when $p\to \infty$, therefore $\prob_p(\{-Y_1=1\} \cap \Xp) \ea 1$.

When $k \ge 2$, Figure~\ref{fig:xy-support} shows that $\Py$-almost surely, $-Y_1=k$ and $-X_1 \le p-2$ imply that $X_1\in \{0,1\}$. Therefore by \eqref{eq:Doob} we have $\Py(\ykXp[=1]) = \pp0k + \frac{a_{p+1}u_c}{a_p} \pp{-1} k$. Moreover $\frac{a_{p+1} u_c}{a_p} \cv[]p 1$, thus by Lemma~\refp{iii}{lem:asymp}, $\frac{a_{p+1} u_c}{a_p} \ea 1$ for $p\ge 1$. It follows that $\Py(\ykXp) \ea \pp0k + \pp{-1}k = \pyy \ea k^{-7/3}$.

\item
Summing \eqref{eq:Doob} over $k'$ gives that $\Py(-X_1=k) = \afrac \pxx$ for all $k \le p-2$. Since $a_p u_c^p \sim \cst \cdot p^{-4/3}$ and $\pxx \ea k^{-7/3}$, Lemma~\ref{lem:asymp} implies that $\Py(-X_1=k) \ea \frac{ (p-k)^{-4/3} }{ p^{-4/3} } k^{-7/3}$. The same estimate also holds for $k=p-1$ because $\Py(-X_1=p-1) \eqv p \cst \cdot p^{-1}$.

For $p\ge p_\theta$ and $k\in [1,\theta p]$, we have $(p-k)^{-4/3} \ea p^{-4/3}$, hence $\Py(-X_1=k) \ea k^{-7/3}$.

For $p\ge p_\theta$ and $k\in [p-\theta p,p-1]$, we have $k^{-7/3} \ea p^{-7/3}$, so $\Py(-X_1=k) \ea \frac{(p-k)^{-4/3}}p$, or equivalently, $\Py(-X_1 = p-k) \ea p^{-1} k^{-4/3}$ for $p\ge p_\theta$ and $k\in [1, \theta p]$.

\item
The asymptotic relation $a_p u_c^p = \cst \cdot p^{-4/3} (1+O(p^{-1/3}))$ implies that there exist constants $C=C(\theta)$ and $p_0=p_0(\theta)$ such that for all $p\ge p_0$ and $-2\le k \le \theta p$,
\begin{equation*}
\frac{(p-k)^{-4/3}}{p^{-4/3}} (1-C p^{-1/3}) \le \afrac \le 
\frac{(p-k)^{-4/3}}{p^{-4/3}} (1+C p^{-1/3}) \,.
\end{equation*}
By writing down the Taylor expansion of the left and right hand side, it is not hard to see that the above inequality implies that $\abs{\afrac -1} \la p^{-1}|k| + p^{-1/3}$ for $p \ge p_0$ and $-2\le k \le \theta p$.
Finally, we can extend the uniform bound to all the values of $p$ because the set $\Set{(k,p)}{ p\le p_0 \text{ and} -2 \le k \le \theta p }$ is finite.

\item
We split the event on the \lhs\ into three parts and use respectively the three estimates provided by \ref{item:estimate Y} and \ref{item:estimate X}:
\begin{align*}
   &\ \Py(-X_1 < p-m \text{ and } (-X_1)\vee (-Y_1)\ge x)    \\    
\le&\ \Py(\ykXp[\ge x])
    + \Py(-X_1 \in [x, \theta p]) 
    + \Py(-X_1 \in [\theta p, p-m])         \\
\la&\ \sum_{k=x}^\infty k^{-7/3}
    + \sum_{k=x}^{\theta p} k^{-7/3}
    + \sum_{k=m}^{(1-\theta) p} \frac1p k^{-4/3}
\ \la\ x^{-4/3} + \frac1p m^{-1/3} .
\end{align*}

\item
When $W=\mu -X_1$, we use the inclusions $\zg \subset \{ h-\mu \le -X_1 \le x \}$ and $\zl \subset \{-2 \le -X_1 \le h-\mu \}$. Then the three inequalities are obtained by summing the first estimate of \ref{item:estimate X} over $k$:
\begin{align*}
\Py(\zg) \le \sum_{k=h-\mu}^x \Py(-X_1=k)   
       & \la \sum_{k=h}^{\theta p}  k^{-7/3} \la h^{-4/3}    \\
\Ey[W \idzg] \le \sum_{k=h-\mu}^x (\mu+k) \cdot \Py(-X_1=k)
           & \la \sum_{k=h}^{\theta p} k \cdot k^{-7/3} \la h^{-1/3}    \\
\Ey[W^2\idzl] \le  \sum_{k=-2}^{h-\mu} (\mu+k)^2 \cdot \Py(-X_1=k)
            & \la 1+\sum_{k=1}^h k^2 \cdot k^{-7/3} \la h^{2/3}.
\end{align*}

When $W=\mu -Y_1$, we write $\zg \subset \{ h-\mu \le -Y_1 \le x \} \cap \{-X_1 \le p-2\}$ and $\zl \subset \{-2 \le -Y_1 \le h-\mu \} \cap \{-X_1 \le p-2\}$.
Then the three inequalities are obtained by summing the estimate \ref{item:estimate Y} over $k$, similarly to the case $W=\mu-X_1$.

\item
We first compare the measure $\Py$ to $\PY$ using \eqref{eq:Doob}: for all $k\le p-2$ and $k'$, we have
\begin{equation*}
\absb{ \Py((-X_1,Y_1)=-(k,k')) - \pp{k}{k'} } 
\le \pp{k}{k'} \abs{\afrac -1} \ .
\end{equation*}

When $W=\mu-X_1$, we sum the above bound over $k'$ and use \ref{item:estimate infty} and \ref{item:estimate afrac} to obtain
\begin{align*}
     \absb{ \Ey[W \id_{\Ap}] - \EY[W \id_{\Ap}] }
\le & \sum_{k=-2}^x \abs{\mu+k} \pxx \cdot \absB{\afrac -1}    \\
\la & \sum_{k=-2}^x \abs{\mu+k} (k+3)^{-7/3} \m({ \frac{|k|}p + \frac1{p^{1/3}} }
\end{align*}
It is not hard to see that the first three terms in the sum above are dominated by the rest of the sum, so
\begin{equation*}
     \absb{ \Ey[W \id_{\Ap}] - \EY[W \id_{\Ap}] }
\la    \frac1p \sum_{k=1}^x k^{-1/3} + \frac1{p^{1/3}} \sum_{k=1}^x k^{-4/3}
\ea    \frac1p x^{2/3} + \frac1{p^{1/3}} \la p^{-1/3} \,.
\end{equation*}

When $W=\mu-Y_1$, we use the fact that under both $\Py$ and $\PY$, we have almost surely either $-X_1 \in \{0,1\}$ or $-Y_1 \le 1$. Dividing the event $\Ap$ according these two cases, we get
\begin{align*}
     \absb{ \Ey[W \id_{\Ap}] - \EY[W \id_{\Ap}] } 
&\le  \sum_{k'=2}^x \abs{\mu+k'} \pyy \cdot \absB{\frac{a_{p+1} u_c}{a_p} -1}
    + \sum_{k=-2}^x (\mu+1) \pxx \cdot \absB{\afrac -1}    \\
&\la  \sum_{k=1}^x k \cdot k^{-7/3} \m({ \frac1p + \frac1{p^{1/3}} }
    + \sum_{k=1}^x k^{-7/3} \m({ \frac kp + \frac1{p^{1/3}} }
\la p^{-1/3}
\end{align*}
We conclude that with both $W=\mu-X_1$ and $W=\mu-Y_1$,
\begin{equation*}
\absb{ \Ey[W \id_{\Ap}] - \EY[W \id_{\Ap}] } \la p^{-1/3}
\end{equation*}
for $p \ge p_\theta$ and $x\in[1,\theta p]$.

On the other hand, since $\EY[\mu-X_1]=\EY[\mu-Y_1] = 0$, we have $\EY[W \id_{\Ap}] = -\EY[W \id_{\Ap^c}]$. On the event $\Ap^c$, almost surely $0\le W \le \max(\mu-X_1,\mu-Y_1)$, therefore
\begin{align*}
\EY[W \id_{\Ap^c}]
\le \EY[(\mu -X_1) \idd{-X_1>x}] + \EY[(\mu -Y_1) \idd{-Y_1>x}]
\la \sum_{k=x}^\infty k\cdot k^{-7/3} \ea x^{-1/3}
\end{align*}
according to the estimate from \ref{item:estimate infty}. It follows that 
$\absb{\Ey[W \id_{\Ap}]} \la p^{-1/3} + x^{-1/3} \ea x^{-1/3}$.

\item
First consider the $+$ sign. We decompose the expectation into three terms:
\begin{equation}\label{eq:proof exp moment bound}
\Ey[e^{\lambda W} \id_{\Ap}] = 
\Py(\Ap) + \lambda \Ey[W \id_{\Ap}] + \Ey[(e^{\lambda W} -1 -\lambda W) \id_{\Ap}] \,.
\end{equation}
The first term is bounded by 1. The second term will be taken care of by \ref{item:estimate E}. For the last term, notice that $W \le \mu+x$ on the event $\Ap$. We cut the interval $(-\infty,\mu+x)$ at $\lambda^{-1}$ and $x/2$ and bound the expectation separately on each subinterval:
\begin{eqnarray*}
\Ey[(e^{\lambda W} -1 -\lambda W) \idzl[\lambda^{-1}]]
\la&\hspace{-8pt} \lambda^2 \Ey[W^2 \idzl[\lambda^{-1}]] 
\hspace{-8pt}&\la \lambda^2 \cdot \lambda^{-2/3} = \lambda^{4/3} ,  \\
\Ey[(e^{\lambda W} -1 -\lambda W) \id_{\{W\in [\lambda^{-1},x/2]\} \cap \Ap} ]
\le&\hspace{-8pt} e^{\lambda x/2} \Py( \zg[\lambda^{-1}] )
\hspace{-8pt}&\la e^{\lambda x/2} \lambda^{4/3}              ,  \\
\Ey[(e^{\lambda W} -1 -\lambda W) \id_{\{W\in [x/2,\mu+x]\} \cap \Ap} ]
\la&\hspace{-8pt} e^{\lambda x  } \Py( \zg[x/2] )
\hspace{-8pt}&\la e^{\lambda x} x^{-4/3}                    .
\end{eqnarray*}
We used the fact that $e^{\lambda W} -1 -\lambda W \la \lambda^2 W^2$ for $W \le \lambda^{-1}$ in the first line, and the assumption $\lambda\le 1$ so that $e^{\lambda \mu}\la 1$ in the last line. The second inequality in each line follows from \ref{item:estimate S}. Combining these three bounds with \eqref{eq:proof exp moment bound} and \ref{item:estimate E} gives
\begin{equation*}
\Ey[e^{\lambda W} \id_{\Ap} ] -1  \la 
\lambda x^{-1/3} + \lambda^{4/3} + e^{\lambda x/2} \lambda^{4/3} + e^{\lambda x} x^{-4/3}.
\end{equation*}
For $\lambda \ge x^{-1}$, the last term on the \rhs\ dominates the other three terms. And since $\log(x)\le x-1$ for all $x>0$, it follows that $\log( \Ey[e^{\lambda W} \id_{\Ap}] ) \la e^{\lambda x} x^{-4/3}$ for $p\ge p_\theta$, $x\in[1,\theta p]$ and $\lambda \in [2 x^{-1},1]$.

Now consider \eqref{eq:proof exp moment bound} with $\lambda$ replaced by $-\lambda$. The first two terms on the right hand side are controled in the same way as before. For the last term, we split the interval $(-\infty,\mu+x)$ at $\lambda^{-1}$, and observe that $e^{-\lambda W} -1 +\lambda W \la \lambda^2 W^2$ for $W \le \lambda ^{-1}$, while $e^{-\lambda W} -1 +\lambda W \le \lambda W$ for $W \ge \lambda^{-1}$. It follows that
\begin{align*}
\Ey[(e^{-\lambda W} -1 +\lambda W) \id_{\Ap}] 
&\la \lambda^2 \Ey[W^2 \idzl[\lambda^{-1}]] + \lambda \Ey[W \idzg[\lambda^{-1}] ]  \\    
&\la \lambda^2 \cdot \lambda^{-2/3} + \lambda \cdot \lambda^{1/3} 
 \la \lambda^{4/3} ,
\end{align*}
where the second line uses again the estimates from \ref{item:estimate S}. Similarly to the $+$ sign case, we deduce that $\log ( \Ey[e^{-\lambda W} \id_{\Ap}] )  \la \lambda x^{-1/3} + \lambda^{4/3} \la \lambda^{4/3} \la e^{\lambda x} x^{-4/3}$ for the same range of the parameters $p,x$ and $\lambda$.    \qedhere
\end{enumerate}
\end{proof}

\newcommand{\tauxx}[1][x]{\tau_{#1}}

As stated at the beginning of this appendix, we start by considering, instead of $\tauxy$, the first time $\tauxx$ that $(X_n,Y_n)$ deviates from its mean for some constant distance $x$, namely $\tauxx = \inf \Set{n\ge 0}{|X_n-\mu n| \vee |Y_n-\mu n| > x}$.

\begin{lemma}\label{lem:cst barrier}
Fix some $\epsilon>0$ and let $x =\chi \mb({ N (\log N)^{1+\epsilon} }^{3/4}$.
Then for all $p\ge p_\theta/(1-\theta)$, $m\ge 1$ and $\chi,N \ge 2$ such that $x\in [1,\theta p/(1+\theta) ]$, we have
\begin{equation*}
\Py (\tauxx \le N, \tauxx < T_m) \la \frac1{(\log \chi + \log N)^{1+\epsilon/2}} + \frac Np m^{-1/3}
\end{equation*}
\end{lemma}

\begin{proof}
For $n\ge 1$, let $\Delta X_n = X_n-X_{n-1}$ and $\Delta Y_n = Y_n-Y_{n-1}$. Consider 
\begin{equation*}
J_x = \inf\Set{n\ge 1}{ (-\Delta X_n) \vee (-\Delta Y_n) \ge x} ,
\end{equation*} 
the first time that either $\nseq X$ or $\nseq Y$ makes a large negative jump of size $x$. We bound the probability of the event $\{\tauxx \le N, \tauxx < T_m\}$ separately in the case $\{J_x\le \tauxx \}$ (large jump estimate) and in the case $\{\tauxx < J_x\}$ (small jump estimate).

\medskip
\noindent\textbf{Large jump estimate: union bound.} Write
\begin{align*}
     & \Py( \tauxx \le N, \tauxx < T_m \text{ and } J_x \le \tauxx )    \\
\le\ & \Py( J_x \le \tauxx \wedge N    \text{ and } J_x < T_m )
     = \sum_{n=1}^N \Py( n \le \tauxx \text{ and } J_x=n < T_m )    \,.
\end{align*}
On the one hand, $J_x=n<T_m$ implies that $P_n > m$ and $(-\Delta X_n) \vee (-\Delta Y_n) \ge x$. On the other hand, we have $P_{n-1}\ge p-x$ on the event $\{n\le \tauxx \}$. Therefore by the Markov property of $\law\py \nseq{P_n,Y}$, we have
\begin{align*}
\Py( n \le \tauxx \text{ and } J_x=n < T_m )
& \le \Ey\m[{ \Prob\py[P_{n-1}] \mb({ P_1>m \text{ and } (-X_1)\wedge (-Y_1)\ge x }
           \idd{P_{n-1}\ge p-x} }    \\
& \le \sup_{p'\ge p-x} \Prob\py[p'] \mb({ -X_1<p-m \text{ and } (-X_1)\wedge (-Y_1)\ge x}
\end{align*}
If $p\ge p_\theta/(1-\theta)$, $x\in [1,\theta p/(1+\theta)]$ and $p'\ge p-x$, then $p'\ge p_\theta$ and $x\in [1,\theta p']$. Thus we can use the uniform bound of Lemma~\ref{lem:estimates}\ref{item:estimate large jump} to bound the above supremum. It follows that
\begin{equation}\label{eq:large jump estimate}
\Py( \tauxx \le N, \tauxx < T_m \text{ and } J_x \le \tauxx )
\la N \m({ x^{-4/3} + \frac1p m^{-1/3} }
  = \frac{\chi^{-4/3}}{ (\log N)^{1+\epsilon} } + \frac{N}{p} m^{-1/3}.
\end{equation}

\newcommand{\unit}{\mathbf{e}}
\newcommand{\tauu}{\tau^{\unit}_x}
\newcommand{\taum}{\tau^{(1,0)}_x}

\noindent\textbf{Small jump estimate: Chernoff bound.}
For each of the four unit vectors $\unit \in \integer^2$, define
\begin{equation*}
\tauu = \inf\Set{n\ge 0}{ (\mu n-X_n, \mu n-Y_n)\cdot \unit \ge x } \,.
\end{equation*}
The stopping time $\tauu$ gives the first time that the process $\nseq{X_n,Y}$ deviates by a distance $x$ from its mean in the direction $-\unit$, thus $\tauxx = \min_{\unit} \tauu$. It follows that
\begin{align}
    \Py( \tauxx \le N, \tauxx < T_m \text{ and } \tauxx < J_x)
\ &\le\ \Py( \tauxx \le N \text{ and }  \tauxx < J_x ) \notag \\
&\le\ \sum_{\unit} \Py( \tauxx = \tauu \le N \text{ and } \tauu < J_x ) 
\,.   
\label{eq:pre Chernoff}
\end{align}
On the event $\{\tauu = n\}$, we have $(\mu n-X_n, \mu n-Y_n) \cdot \unit \ge x$. 
Thus the Chernoff bound gives:
\begin{align}
\Py(\tauxx = \tauu \le N \text{ and } \tauu <&\, J_x)  
\ = \ \sum_{n=1}^N \Py(\tauxx = \tauu = n < J_x)                            \notag \\
\le&\ e^{-\lambda x} \sum_{n=1}^N \Ey \m[{ \idd{\tauxx=n} \idd{n<J_x} 
                    e^{\lambda (\mu n - X_n, \mu n - Y_n) \cdot \unit}  }  \notag \\
 \le&\ e^{-\lambda x} \sum_{n=1}^N \Ey \m[{ \idd{\tauxx=n} \prod_{i=1}^n 
                    e^{\lambda (\mu-\Delta X_i, \mu-\Delta Y_i) \cdot \unit}   
                        \idd{(-\Delta X_i) \vee (-\Delta Y_i) \le x}    }
\label{eq:stopped Chernoff}
\end{align}
for all $\lambda \ge 0$.

\newcommand{\phix}[1][p]{\varphi^{x,\unit}_{#1}(\lambda)}

For $p\in \natural \cup \{\infty\}$, let $\phix = \Ey[ e^{\lambda (\mu-X_1, \mu-Y_1) \cdot \unit} \id_{\Ap}]$, where $\Ap = \{(-X_1) \vee (-Y_1) \le x \}$ is the same event as in Lemma~\ref{lem:estimates}. Since the pair $(X_1,Y_1)$ takes only finitely many values on the event $\Ap$ and $\law\py(X_1,Y_1) \to \law\yy(X_1,Y_1)$ in distribution, we have $\phix \to \phix[\infty]$ as $p \to\infty$. It follows that there exists $p_* = p_*(x,\unit,\lambda) \in \{p':p'\ge p-x\} \cup \{\infty\}$ such that 
\begin{equation*}
\phix[p_*] = \sup_{p'\ge p-x} \phix[p']
\end{equation*}
Let $\nseq[1]{\Delta X^*_n, \Delta Y^*}$ be a sequence of i.i.d.\ random variables independent of $\nseq{X_n,Y}$ and such that $\law_p (\Delta X_1^*,\Delta Y_1^*) = \law\py[p_*](X_1,Y_1)$ in distribution.
Define the process
\begin{equation*}
(U_i,V_i) = \begin{cases}
			-(\Delta X_i  , \Delta Y_i)   & \text{if }i \le \tauxx
		\\ 	-(\Delta X_i^*, \Delta Y_i^*) & \text{if }i  >  \tauxx.
\end{cases}
\end{equation*}
By definition, on the event $\{\tauxx=n\}$, the future $(U_i, V_i)_{i>n}$ of this process is an i.i.d.\ sequence independent of its past such that
$\Ey[e^{\lambda (\mu+U_i, \mu+V_i) \cdot \unit} \idd{U_i \vee V_i \le x }] = \phix[p_*]$. Therefore we can continue the bound \eqref{eq:stopped Chernoff} with
\begin{align}
& e^{-\lambda x} \sum_{n=1}^N \Ey \m[{ \idd{\tauxx=n} \prod_{i=1}^n 
     e^{\lambda (\mu+U_i, \mu+V_i) \cdot \unit} \idd{U_i \vee V_i \le x }  } 
\notag \\  =  \ 
& e^{-\lambda x} \sum_{n=1}^N \mb({\phix[p_*]}^{-(N-n)} \Ey \m[{ 
     \idd{\tauxx=n} \prod_{i=1}^N 
             e^{\lambda (\mu+U_i, \mu+V_i) \cdot \unit} \idd{U_i \vee V_i \le x }  } 
\notag \\ \le \ 
& e^{-\lambda x} \cdot (1\vee \phix[p_*]^{-N}) \cdot \Ey \m[{ \prod_{i=1}^N 
             e^{\lambda (\mu+U_i, \mu+V_i) \cdot \unit} \idd{U_i \vee V_i \le x }  } \,.
\label{eq:completed Chernoff}
\end{align}
It is easy to see that $\tauxx$ is a stopping time with respect to the natural filtration $\nseq \filtr$ of the process $\nseq{U_n,V}$. Therefore for all $i\ge 1$,
\begin{align*}
	&\ \Ey\Econd{ e^{\lambda (\mu+U_{i+1}, \mu+V_{i+1}) \cdot \unit} \idd{U_{i+1} \vee V_{i+1} \le x} }{\filtr_i} \\
&=\		\idd{i< \tauxx} \cdot \phix[P_i] + \idd{i\ge\tauxx} \cdot \phix[p_*]
\ \le \	\phix[p_*]    \,,
\end{align*}
where the last inequality follows from that fact that $P_i\ge p-x$ and hence $\phix[P_i] \le \phix[p^*]$ on the event $\{i< \tauxx\}$. By expanding the expectation in \eqref{eq:completed Chernoff} with $N$ successive conditionings, we see that it is bounded by $\phix[p_*]^N$. Then we conclude from \eqref{eq:stopped Chernoff} and \eqref{eq:completed Chernoff} that
\begin{equation*}
        \Py(\tauxx = \tauu \le N \text{ and } \tauu < J_x)  
\ \le\  e^{-\lambda x} (\phix[p_*]^N \vee 1) \,.
\end{equation*}

By Lemma~\ref{lem:estimates}\ref{item:estimate small jump}, there exists a constant $C$ such that $\phix \le \exp(C x^{-4/3} e^{\lambda x})$ for all $p\ge p_\theta$, $x\in [1,\theta p]$, $\lambda \in [2 x^{-1},1]$ and unit vector $\unit \in \integer^2$. As we have seen in the derivation of the large jump estimate, the same bound holds for $\phix[p_*] = \sup_{p'\ge p-x} \phix$, provided that $p\ge p_\theta/(1-\theta)$.
Therefore we have
\begin{equation*}
\Py(\tauxx = \tauu \le N, \tauu < J_x)  
\ \le\   \exp(-\lambda x + C\cdot N x^{-4/3} e^{\lambda x}) \,
\end{equation*}
Plugging this in \eqref{eq:pre Chernoff} and taking $\lambda x = c\log \log x$ with $c=1+\epsilon/2$ gives
\begin{equation*}
\Py( \tauxx \le N, \tauxx < T_m \text{ and } \tauxx < J_x )
\ \le\  4 \exp(-c\log\log x + C N x^{-4/3} (\log x)^c)\, .
\end{equation*}
Since $x= \chi \mb({ N(\log N)^{1+\epsilon} }^{3/4}$ by assumption, we have $N x^{-4/3} (\log x)^c \ea \chi^{-4/3} \frac{(\log \chi + \log N)^c}{(\log N)^{1+\epsilon}}$, which is bounded by a constant for $\chi,N\ge 2$. It follows that 
\begin{equation}\label{eq:small jump estimate}
\Py( \tauxx \le N, \tauxx < T_m \text{ and } \tauxx < J_x )
\ \la\  \exp(-c\log\log x) \ea (\log \chi + \log N)^{-c}\, .
\end{equation}
The boundedness of $\chi^{-4/3} \frac{(\log \chi + \log N)^c}{(\log N)^{1+\epsilon}}$ also implies that the large jump estimate \eqref{eq:large jump estimate} can be written as 
$\Py( \tauxx \le N, \tauxx < T_m \text{ and } J_x \le \tauxx) \la (\log \chi + \log N)^{-c} + N p^{-1} m^{-1/3}$.
Adding it to \eqref{eq:small jump estimate}, we conclude that $\Py( \tauxx \le N, \tauxx < T_m ) \la (\log \chi + \log N)^{-c} + N p^{-1} m^{-1/3}$.
\end{proof}

\begin{figure}
\centering
\includegraphics[scale=1]{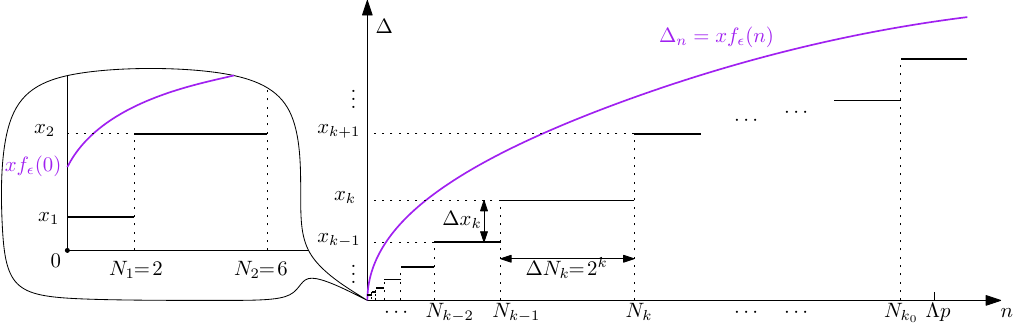}
\caption{The segments $I_k = \Set{(n,x_k)}{n\in (N_{k-1},N_k]}$ and the curve $\Delta_n=\barrier(n)$.}
\label{fig:staircase}
\end{figure}

\begin{proof}[Proof of Lemma~\ref{lem:one jump}]
Let $\Delta_n = |X_n-\mu n| \vee |Y_n-\mu n|$. Recall that $\tauxy = \inf \Set{n\ge 0}{\Delta_n > \barrier(n)}$ where $\barrier[](n) = \mb({ (n+2)(\log(n+2))^{1+\epsilon} }^{3/4}$, and we want to prove that
\begin{equation*}
\lim_{x,m\to\infty} \limsupp \Py(\tauxy<T_m) =0 \,.
\end{equation*}

The idea is to apply Lemma~\ref{lem:cst barrier} to an exponentially increasing sequence of time intervals and to concatenate the results using the Markov property. Consider the sequences $(N_k)_{k\ge 0}$ and $(x_k)_{k\ge 0}$ defined by $N_0=x_0=0$,
\begin{equation*}
\Delta N_k := N_k - N_{k-1} = 2^k    \qtq{and}
\Delta x_k := x_k-x_{k-1} 
= \frac{x}3 \mb({ \Delta N_k \m({ \log \Delta N_k }^{1+\epsilon} }^{3/4}.
\end{equation*}
Then we have $N_k = 2^{k+1}-2$ and 
\begin{align*}
x_k = \frac{x}3 \sum_{i=1}^k 2^{\frac34 i} \cdot (i \log 2)^{\frac34(1+\epsilon)}
\le \frac{x}3 \cdot \frac{2^{\frac34 (k+1)}}{2^{3/4}-1} (k \log 2)^{\frac34 (1+\epsilon)}
\le x \mb({ 2^k (\log 2^k)^{1+\epsilon} }^{3/4}\,.
\end{align*}
In other words, $x_k \le \barrier(N_{k-1})$. 

\newcommand*{\kex}{K^\epsilon_{x,m}}%
Consider the sequence of horizontal segments $I_k = \Set{(n,x_k)}{n\in (N_{k-1}, N_k]}$ depicted in Figure~\ref{fig:staircase}. Thanks to the previous inequality, all of these segments are below the curve $\Delta_n = \barrier(n)$. Let $\kex$ be the index $k$ where $\Delta_n$ goes above $I_k$ for the first time up to $T_m$, that is\
\begin{equation*}
\kex = \inf \Set{k\ge 1}{\exists n \in (N_{k-1},N_k] \text{ s.t.\ } \Delta_n>x_k \text{ and } n<T_m} \,.
\end{equation*}
Then we have $\{\tauxy < T_m\}\subset \{\kex < \infty\}$. 
Remark that $\Delta_{N_{k-1}}\le x_{k-1}$ and $\Delta_{n+N_{k-1}}>x_k$ imply that $\tilde \Delta_n := |X_{n+N_{k-1}} - X_{N_{k-1}} -\mu n| \vee |Y_{n+N_{k-1}} - Y_{N_{k-1}} -\mu n| > \Delta x_k$ for any $n\ge 0$.
Therefore by Markov property of $\law\py \nseq{X_n,Y}$,
\begin{align*}
\Py(\kex=k) &\le 
\Ey\m[{ \Prob\py[P_{N_{k-1}}] \mb({ \exists n \in (0,\Delta N_k] \text{ s.t.\ } \Delta_n >\Delta x_k \text{ and } n<T_m } \id_{\{\Delta_{N_{k-1}} \le x_{k-1} \}} }   \\ &\le
\sup_{p'\ge p-x_{k-1}} \Prob\py[p'](\tauxx[\Delta x_k] \le \Delta N_k, \tauxx[\Delta x_k] < T_m) \,.
\end{align*}

Let $k_0$ be the largest $k$ such that $N_k \le \Lambda p$, where $\Lambda \ge 1$ is some cut-off value that will be sent to infinity after $p$, $x$ and $m$. 
For any fixed $x$, $m$ and in the limit $p\to\infty$, we have $\Delta x_{k-1} \le \theta p$ and $p-x_{k-1} \ge p- \barrier(\Lambda p) >p_\theta/(1-\theta)$ for all $k\le k_0$.
Therefore we can apply Lemma~\ref{lem:cst barrier} to bound the above supremum, which gives
\begin{align*}
\Py(\kex \le k_0) & \la \sum_{k=1}^{k_0} \m({
\frac1{ (\log (x/3) + \log(\Delta N_k))^{1+\epsilon/2}} + \frac{\Delta N_k}p m^{-1/3} }
\\ & = \sum_{k=1}^{k_0} \frac1{ (\log(x/3)+k \log 2)^{1+\epsilon/2} } + \frac{N_{k_0}}p m^{-1/3}
\la \frac1{ (\log x)^{\epsilon/2} } + \Lambda m^{-1/3} \,.
\end{align*}
On the other hand, $k_0<\kex <\infty$ implies that $T_m > N_{k_0}$. Therefore by Lemma~\ref{lem:hit 0},
\begin{equation*}
\Py(k_0 < \kex < \infty) \le \Py(T_m > N_{k_0}) \le \Py(T_0 > N_{k_0}) \le \left(\frac{N_{k_0}}{p}\right)^{-\gamma_0} \,.
\end{equation*}
It is easy to see that $2N_{k_0}+2 = N_{k_0+1} \ge \Lambda p$, which yields $(N_{k_0}/p)^{-\gamma_0}\la \Lambda^{-\gamma_0}.$
We conclude that for every fixed $\Lambda>0$, and uniformly for $x>0$ and $m\ge 1$,
\begin{equation*}
\limsupp \Py(\tauxy < T_m) \le 
\limsupp \Py(\kex < \infty) \la (\log x)^{-\epsilon/2} + \Lambda m^{-1/3} + \Lambda^{-\gamma_0} \,.
\end{equation*}
Taking the limit $m,x\to\infty$ and then $\Lambda \to\infty$ finishes the proof. 
\end{proof}

\paragraph{Acknowledgements.}
We would like to thank M.~Bousquet-M\'elou, J.~Bouttier, N.~Curien, B.~Eynard, K.~Izyurov and A.~Kupiainen for enlightening discussions and useful suggestions. We are grateful to  M.~Albenque, L.~M\'enard and G.~Schaeffer for sharing their work on a similar model while it was still in progress. We thank B.~Duplantier, E.~Gwynne and S.~Sheffield for explaining to us an interpretation of the law of $\zeta$ in Theorem~\ref{thm:scaling limit}. We also thank the anonymous referees for their numerous corrections and remarks.
L.~Chen thanks the  hospitality of Universit\'e Paris–Sud,  where much of this work was conducted during his doctoral study. He also acknowledges the support from the Agence National de la Recherche via Grant ANR-14-CE25-0014 (ANR GRAAL) and ANR-12-JS02-0001 (ANR CARTAPLUS). J.~Turunen acknowledges the hospitality of Institut Henri Poincar\'e during the trimester ``Combinatorics and interactions'', where part of this work was conducted.
Both authors have been supported by the Academy of Finland via the Centre of Excellence in Analysis and Dynamics Research (project No.~271983), and by the ERC Advanced Grant 741487 (QFPROBA).

\bibliographystyle{abbrv}
\bibliography{database-linxiao_JT_3}

\end{document}